\documentclass{amsart}
\usepackage{amsfonts,amscd}

\newtheorem{theorem}{Theorem}[section]
\newtheorem{lemma}[theorem]{Lemma}
\newtheorem{corollary}[theorem]{Corollary}
\newtheorem{proposition}[theorem]{Proposition}
\theoremstyle{remark}
\newtheorem{remark}[theorem]{Remark}
\theoremstyle{definition}
\newtheorem{definition}[theorem]{Definition}
\newtheorem{example}[theorem]{Example}
\newtheorem{notation}[theorem]{Notation}

\numberwithin{equation}{section}
\makeatother

\DeclareMathOperator{\spn}{span}

\DeclareMathOperator{\Kdb}{{\mathbb{K}}}
\DeclareMathOperator{\Cdb}{{\mathbb{C}}}
\DeclareMathOperator{\Zdb}{{\mathbb{Z}}}
\DeclareMathOperator{\Rdb}{{\mathbb{R}}}

\DeclareMathOperator{\Tdb}{{\mathbb{T}}}
\DeclareMathOperator{\Ndb}{{\mathbb{N}}}

\newcommand{\In}{\infty}

\newcommand{\Ball}{\operatorname{Ball}}
\newcommand{\Ran}{{\operatorname{Ran}}}
\renewcommand{\Re}{{\operatorname{Re}}}

\newcommand{\af}{\alpha}
\newcommand{\bt}{\beta}

\newcommand{\dt}{\delta}
\newcommand{\ep}{\varepsilon}
\newcommand{\zt}{\zeta}
\newcommand{\et}{\eta}
\newcommand{\ch}{\chi}

\newcommand{\ld}{\lambda}
\newcommand{\sm}{\sigma}
\newcommand{\kp}{\kappa}
\newcommand{\ph}{\varphi}
\newcommand{\ps}{\psi}
\newcommand{\rh}{\rho}
\newcommand{\om}{\omega}
\newcommand{\ta}{\tau}

\newcommand{\Ld}{\Lambda}

\newcommand{\andeqn}{\qquad {\mbox{and}} \qquad}

\newcommand{\diag}{{\operatorname{diag}}}

\newcommand{\ifo}{if and only if}

\newcommand{\SM}{\setminus}

\begin{document}

\title[$L^p$ operator algebras with approximate
 identities I]{$L^p$ operator algebras
 with approximate identities I} 

\author{David P.\  Blecher}
\address{Department of Mathematics, University of Houston,
 Houston, TX 77204-3008}
\email[David P.\  Blecher]{dblecher@math.uh.edu}

\author{N.~Christopher Phillips}
\address{Department of Mathematics, University of Oregon,
 Eugene, OR 97403-1222}

\subjclass[2010]{46H10, 46H99, 46E30, 47L10, 47L30 (primary),
 46H35, 47B38, 47B44, 47L75 (secondary)}
\keywords{$L^p$ operator algebra, accretive, approximate identity, Banach algebra, Kaplansky density, M-ideal, real positive, smooth Banach space,  strictly convex Banach space,state, unitization.} 

\thanks{This work is based on work supported by
the US National Science Foundation under
Grant DMS-1501144 (Phillips),
a Simons Foundation grant  527078 (Blecher),
and a minigrant from the Mathematics
Department of the University of Houston.
Material from this paper
and its sequel were presented at 2017 conferences in Houston
(August), the East Coast Operator Algebras Symposium,
and the SAMS congress.}

\begin{abstract}
We initiate an investigation into how much the existing theory
of (nonselfadjoint) operator algebras on a Hilbert space
generalizes to algebras acting on $L^p$~spaces.
In particular we investigate the applicability
of the theory of real positivity,
which has recently been useful in the study of
$L^2$-operator algebras and Banach algebras,
to algebras of bounded operators on $L^p$ spaces. 
In the process we answer some open questions on
real positivity in Banach algebras from
work of the first author and Ozawa.
\end{abstract}

\maketitle

\tableofcontents

\section{Introduction}\label{Sec_1}

In a series of recent papers
(see e.g.\  \cite{P1, P2, P3, P4})
the second author has
pointed out that the study
of algebras of bounded operators on $L^p$ spaces,
henceforth,
\emph{$L^p$-operator algebras}, has been somewhat overlooked, and
has initiated the study of these objects.
Subsequently others have followed him into this inquiry
(for example,
Gardella, Thiel, Lupini, and Viola;
see e.g.\  \cite{G2,G3,G4,GL, PV}).
However, as he has frequently stated,
these investigations have been very largely focused on examples;
one still lacks an abstract general theory in this setting.

Here and in a sequel in preparation we initiate an investigation into
how much the existing theory of
(nonselfadjoint) $L^2$-operator algebras
(see e.g.\   \cite{BLM}, \cite{BRI})
generalizes to the $L^p$ case. 
We restrict ourselves almost
exclusively to the
``isometric theory''; we may pursue the isomorphic theory
elsewhere.
In addition to establishing some general facts
about $L^p$ operator algebras,
the main goal of the present paper
is to investigate to what extent the first author's theory
of real positivity
(developed with Read, Neal, Ozawa, and others;
see e.g.\   \cite{BRI, BRII, BRord, BOZ}),
is applicable to
$L^p$-operator algebras,
particularly those which are approximately unital, that is,
have contractive approximate identities.
As an easy motivation,
notice that the canonical approximate identity
for
the compact operators $\Kdb (l^p)$ is real positive,
and the real positive elements span $B (L^p ([0, 1]))$
(as they do any unital Banach algebra).

The theory of real positivity was developed as a tool for
generalizing certain parts of $C^*$-algebra theory
to more general algebras.
In \cite{BOZ}
this was extended to Banach algebras (see also \cite{BSan}
for a survey and some additional results).
All this theory
of course therefore applies to $L^p$-operator algebras.
We refer to  \cite{BOZ} frequently,
although most of our paper may be read without
a deep familiarity with that paper.

Some parts of \cite{BOZ}
applied only to certain classes of Banach algebras defined there,
which were shown to behave in some respects
similarly to $L^2$-operator algebras.
For example,
a nonunital approximately unital Banach algebra $A$
was defined there to be
{\emph{scaled}} if the set of restrictions to $A$
of states on the multiplier unitization $A^1$
equals the quasistate space $Q (A)$ of $A$
(that is, the set of  $\lambda \varphi$ for $\lambda \in [0, 1]$
and $\varphi$ a norm~$1$ functional on $A$
that extends to a state on $A^1$).
All unital Banach algebras are scaled.
In \cite{BOZ},
there are several pretty equivalent conditions
for a Banach algebra to be scaled (see the start of our Section  \ref{Sec_6} for some of these),
and this class of Banach algebras
was shown to have several nice theoretical features,
such as a Kaplansky density type theorem.  
Thus it is natural to ask the following:
\begin{enumerate}
\item\label{7819_InOntro_a}
To which of the classes
defined in~\cite{BOZ} do $L^p$~operator algebras belong?
\item\label{7819_InOntro_b}
For those classes in~\cite{BOZ} to which they do not belong,
to what extent do the theorems for those classes from \cite{BOZ}
still hold for $L^p$-operator algebras?
\item\label{7819_InOntro_c}
To what extent do other parts of the theory of $L^2$-operator algebras
hold for $L^p$-operator algebras?
\end{enumerate}
We focus mostly here on
the parts of the theory of the first author
with Read, Neal, and others referred to above
that were not already extended to the general classes
considered in~\cite{BOZ}.
For example, one may ask if the material
in Section~4 in~\cite{BOZ},
and in particular the theory of hereditary subalgebras, improves
(that is, becomes closer to the $L^2$-operator case) for
$L^p$-operator algebras.
Similarly, one may ask about the
{\emph{noncommutative topology}}
(in the sense of Akemann, Pedersen, L.~G.\  Brown, and others) of
$L^p$-operator algebras.
In papers of the first author
with Read, Neal, and others referred to above, Akemann's noncommutative
topology of $C^*$-algebras
was fused with the classical theory of (generalized) peak sets of
function algebras to create
a \emph{relative} noncommutative topology for closed subalgebras of
$C^*$-algebras that has proved
to have many applications.
Examples given in \cite{BOZ} show that
not much of this will extend to general Banach algebras, and it is
natural to ask if $L^p$-operator algebras
are better in this regard.
Most of the present paper and the sequel in preparation is devoted to
answering these questions.
In the process we answer some open questions from \cite{BOZ}.

We admit from the outset that for $p \neq 2$,
and for some significant part of the theory,
the answer to question (\ref{7819_InOntro_b}) above
is so far in the negative.
This may
change somewhat in the future,
for example if we were able to solve
some of the open problems listed at the end of our paper. 
It should also be
admitted that for $p \neq 2$
the ``projection lattice'' of
$B (L^p(X, \mu))$ is problematic from the perspective of our paper (see
Example \ref{mp} and the sequel paper),
in contrast to the projection lattices of von Neumann algebras and
$L^2$-operator algebras.

Concerning question (\ref{7819_InOntro_a}),
several classes of Banach algebras introduced
and considered in  \cite{BOZ}
coincide for approximately unital $L^p$-operator algebras.
Indeed the classes of \emph{scaled}
and $M$-\emph{approximately unital} Banach algebras
defined in \cite{BOZ}
coincide for $L^p$-operator algebras, and these also turn out to be the
approximately unital $L^p$-operator algebras
which satisfy the aforementioned Kaplansky density property.
(We remark that the usual Kaplansky density theorem variants
for $C^*$-algebras can be shown
to follow easily from the weak* density of the subset of interest
in  $A$ within the matching set in $A^{**}$.
Our Kaplansky density theorems have the latter flavor.)
We show that some approximately unital $L^p$-operator algebras
are scaled and others are not.
This answers the questions from \cite{BOZ}
as to whether every approximately unital Banach algebra
is scaled, or has a Kaplansky density property.
Also, non-scaled approximately unital  $L^p$-operator algebras
may contain no real positive elements
(whereas it was shown in \cite{BOZ}
that if they are scaled
then there is an abundance of real positive elements,
e.g.\  every element in $A$ is a difference of two
real positive elements).

Concerning question (\ref{7819_InOntro_c}) above,
indeed some aspects of the theory improve.
For example, Section~4 of \cite{BOZ} improves drastically
in our setting, and indeed $L^p$-operator algebras do support a
basic theory of noncommutative topology and hereditary subalgebras,
unlike general Banach algebras.
This is worked out in the sequel paper in preparation, where the reader
will find many more positive results than in the present paper.
It is worth remarking that the methods used here
do not seem to extend far beyond the
class of $L^p$-operator algebras as we will discuss elsewhere.
However, most of our results for $L^p$-operator algebras
in Sections \ref{Sec_2} and~\ref{Sec_4} do generalize to the class
of {\emph{${\mathrm{SQ}}_p$-operator algebras}},
by which we mean closed
algebras of operators on an ${\mathrm{SQ}}_p$ space, that is,
a quotient of a subspace of an $L^p$ space.
(See e.g.\   \cite{LeM}.
We thank Eusebio Gardella for suggesting ${\mathrm{SQ}}_p$ spaces
after we
listed in a talk the properties needed for our
results to work.)

On the other hand, except cosmetically, not much to speak of
in Section~3 of \cite{BOZ} improves for $L^p$-operator algebras,
in the sense of becoming significantly more
like the $L^2$-operator algebra case.
However several of the concepts appearing throughout \cite{BOZ}
become much simpler in our setting.
For example as we said above,
three of the main classes of Banach algebras considered there coincide.
Also as we shall see the subscript and superscript ${\mathfrak{e}}$
which appear often in \cite{BOZ} may be erased in our setting,
since we are able to show that all $L^p$-operator algebras
are Hahn-Banach smooth.
Then of course the Arens regularity of $L^p$-operator algebras
means that many irritating features of the bidual
appearing in~\cite{BOZ}
disappear, such as mixed identities in $A^{**}$.

We now describe the contents of our paper.

We will be assuming that $p \in (1,  \infty) \setminus \{ 2 \}$ in all
results in the paper unless
stated to the contrary.
As usual $\frac{1}{p} + \frac{1}{q} = 1$.
In the remainder of Section~\ref{Sec_1}
we give some notation and basic definitions.
In Section~\ref{Sec_2} we discuss further notation and background.
We also collect a large number of useful general facts, many of which
are well known.
They concern topics
such as duals, bidual algebras, the multiplier unitization, states
and real positivity, hermitian elements, representations, etc.
We just mention one sample result from this section:
if $A$ is an
approximately unital $L^p$-operator algebra with $p \in (1, \infty)$,
then there exists a measure space $(X, \mu)$
and a unital isometric representation
$\theta \colon A^{**} \to B (L^p (X, \mu) )$
which is a weak* homeomorphism onto its range,
and such that $\theta(A)$ acts nondegenerately on $L^p (X, \mu)$.

In Section~\ref{Sec_3}
we list the main examples of $L^p$-operator algebras
which we use in this paper for counterexamples,
as well as some other basic examples not in the literature.
Some of these have real positive
approximate identities, and others do not.
We also expose some of the aforementioned bad properties of the
``projection lattice'' of $B (L^p(X, \mu))$.

Section~\ref{Sec_4} contains many miscellaneous results.
Here is a sample of these.
We show
that the quotient of an $L^p$-operator algebra
by an approximately unital
closed ideal satisfying a simple extra condition is again
(isometrically) an $L^p$ operator algebra.
An example is presented to prove that this can fail
if the ideal is only assumed to
be closed and approximately unital.
We show that 
an $L^p$-operator algebra $A$ need not have a unique
unitization, unlike in the case $p = 2$ (Meyer's unitization theorem).
However there is a unique
unitization if we restrict attention
to nondegenerately represented approximately
unital $L^p$-operator algebras.
The nonuniqueness above is related to the fact
that when $p \neq 2$ the Cayley transform
can take a real positive element of $A$
to an element of norm greater than 1.
We study support idempotents of elements of $A$ and their properties. 
We also give some important consequences
of the strict convexity of $L^p$ spaces.
For example, a state on a unital $L^p$-operator algebra
that takes the value zero at
a real positive idempotent~$e$
is zero on the left or right ideal generated by~$e$.
We also deduce that an $L^p$-operator algebra
is Hahn-Banach smooth in its multiplier unitization.
These results have several
significant applications in this paper and its sequel.
For example they yield in Section~\ref{Sec_4} several
foundational properties of states and state extensions.

In Section~\ref{Sec_5} and Section~\ref{Sec_6}
we discuss $M$-ideals and scaled Banach algebras.
Our main result here is that in the setting of
approximately unital $L^p$-operator algebras,
the classes of scaled algebras
and $M$-approximately unital algebras coincide.
These are also the
algebras which satisfy the aforementioned Kaplansky density property,
as we show in Section~\ref{Kaps}.
We will see for example that the algebra
$\Kdb (L^p (X, \mu))$ of compact operators is in this
class if and only if $\mu$ is purely atomic
(Proposition \ref{P_7X14_KEMAppU}).
The $L^p$-operator algebras with a
hermitian contractive identity are also in this class.
We also show for example in these sections
that every $M$-ideal $J$ in an approximately unital
$L^p$-operator algebra $A$ is
an approximately unital closed ideal.
Moreover, if in addition $A$ is scaled
then so is $J$ (this follows from Theorem
\ref{mth} (\ref{mth_Mapp_Down})).

At the end of the paper we provide an index listing
some of the main definitions in this paper
and where they may be found.

In the sequel paper in preparation
we show that the theory of one-sided ideals,
hereditary subalgebras,
open projections, etc.\  for $L^p$-operator algebras
is quite similar to the (nonselfadjoint) $L^2$-operator algebra case.
This is particularly so for certain large
classes of $L^p$-operator algebras.
We feel that this is important, since hereditary subalgebras play
a large role in modern $C^*$-algebra theory, and thus hopefully will be
important for $L^p$-operator algebras too.

We end our introduction with a few definitions and basic lemmas.

We set $\Rdb_{+} = [0, \infty)$.

\begin{notation}\label{N_7818_Ball}
Let $E$ be a normed vector space.
Then $\Ball (E)$ is the closed unit ball of~$E$,
that is,
\[
\Ball (E)
 = \big\{ \xi \in E \colon \| \xi \| \leq 1 \big\}.
\]
\end{notation}

\begin{notation}\label{N_7818_LpSum}
Let $p \in [1, \infty]$.
Let $E$ and $F$ be normed vector spaces.
We denote by $E \oplus^p F$ their $L^p$~direct sum,
that is,
the algebraic direct sum $E \oplus F$
with the norm given for $\xi \in E$ and $\et \in F$
by
$\| (\xi, \et) \| = (\| \xi \|^p + \| \et \|^p)^{1/p}$
if $p < \infty$
and $\| (\xi, \et) \| = \max ( \| \xi \|, \, \| \et \| )$
if $p = \infty$.
\end{notation}

Although many of our Banach algebras
have identities of norm greater than~$1$,
the adjectives ``unital'' or ``approximately unital'' for a Banach
algebra will carry a norm~$1$ requirement.

\begin{definition}\label{D_7618_UBA}
A {\emph{unital}} Banach algebra
is a Banach algebra with an identity $1$
such that $\| 1 \| = 1$.
\end{definition}

\begin{definition}\label{D_7618_cai}
A {\emph{cai}} in a Banach algebra
is a contractive approximate identity,
that is, an approximate identity $(e_{t})_{t \in \Ld}$
such that $\| e_{t} \| \leq 1$
for all $t \in \Ld$.
An {\emph{approximately unital}} Banach algebra
is a Banach algebra which has a cai.
\end{definition}

When we write $L^p$ or $L^p (X)$
we mean the $L^p$ space of some measure space $(X, \mu)$.

\begin{definition}\label{D_smoothconv}
Recall that a Banach space~$E$ is {\emph{strictly convex}}
if whenever $\xi, \et \in E \setminus \{ 0 \}$
satisfy
$\| \xi + \eta \| = \| \xi \| + \| \eta \|$,
then there is $\ld \in (0, \infty)$
such that $\xi = \ld \eta$,
and {\emph{smooth}}
if for given $\xi \in E$ with $\| \xi \| = 1$,
there is a unique $\eta \in \Ball (E^*)$
with $\langle \xi, \eta \rangle = 1$.
\end{definition}

If $1 < p < \infty$,
then $L^p (X)$ is strictly convex
(by the converse to Minkowski's inequality).
Moreover,
still assuming $1 < p < \infty$,
the space $L^p (X)$ is smooth,
with $\et$ above given by the function
\[
\eta (x) = \begin{cases}
  \overline{\xi (x)} |\xi (x) |^{p - 2} & \xi (x) \neq 0
         \\
    0                                   & \xi (x) = 0
\end{cases}
\]
in $L^q (X)$.
We will frequently use the fact that $L^p (X)$ is smooth
and strictly convex if  $1 < p < \infty$.

\begin{definition}\label{D_7618_LpOpAlg}
Let $p \in [1, \infty)$.
An {\emph{$L^p$-operator algebra}}
is a Banach algebra which is isometrically isomorphic
to a norm closed subalgebra of the algebra of bounded operators
on $L^p (X, \mu)$
for some measure space $(X, \mu)$.
When $p = 2$ we simply refer to an {\emph{operator algebra}}.
(See the beginning of Section~2.1 of~\cite{BLM},
except that we do not consider matrix norms
in the present paper.)
\end{definition}

\begin{definition}\label{D_7720_Unitization}
Let $A$ be an $L^p$-operator algebra
(not necessarily approximately unital).
We say that an $L^p$~operator algebra $B$
is an {\emph{$L^p$~operator unitization of~$A$}}
if either $A$ is unital and $B = A$,
or if $A$ is nonunital,
$B$ is unital (in particular, by our convention, $\| 1 \| = 1$),
and $A$ is a codimension one ideal in~$B$.
\end{definition}

\begin{definition}[(A.9) on p.\  364
   in~\cite{BLM}]\label{D_7702_Unitization}
Let $A$ be a nonunital approximately unital Banach algebra
(as in Definition~\ref{D_7618_cai}).
We define its {\emph{multiplier unitization}} $A^{1}$
to be the usual unitization $A + \Cdb \cdot 1$ with the norm
\[
\| a  + \lambda 1 \|_{A^{1}}
    = \sup \big( \big\{ \| a c + \lambda c \|
        \colon c \in \Ball (A) \big\} \big).
\]
for $a \in A$ and $\ld \in \Cdb$.
If $A$ is already unital then we
set $A^1 = A$.
\end{definition}

\begin{remark}\label{R_7702_UnitizationFacts}
We recall the following easy standard facts.
\begin{enumerate}
\item\label{R_7702_UnitizationFacts_Incl}
If $A$ is an approximately unital Banach algebra,
then the standard inclusion of $A$ in $A^{1}$ is isometric.
\item\label{R_7702_UnitizationFacts_AI}
Let $A$ be a Banach algebra,
and let $(e_{t})_{t \in \Ld}$ be any cai in~$A$.
Then
\[
\| a  + \lambda 1 \|_{A^{1}}
   = \lim_t \| a e_t + \lambda e_t \|
   = \sup_t \| a e_t + \lambda e_t \|.
\]
\item\label{R_7702_UnitizationFacts_Small}
If $A$ is any nonunital Banach algebra,
and $B$ is a unital Banach algebra which contains $A$ as
a codimension~$1$ subalgebra,
then the map $\chi_0 \colon B \to \Cdb$,
given by
$\chi_0 (a + \lambda 1_B) = \lambda$
for $a \in A$ and $\lambda \in \Cdb$,
is contractive.
\item\label{R_7702_UnitizationFacts_ToMult}
If $A$ is any nonunital Banach algebra with a cai,
and $B$ is a unital Banach algebra which contains $A$ as
a codimension~$1$ subalgebra,
then the map $\ps \colon B \to A^1$,
given by
$\ps (a + \lambda 1_B) = a + \lambda 1_{A^1}$
for $a \in A$ and $\lambda \in \Cdb$,
is a contractive homomorphism.
Thus $A^1$ has the smallest norm of
any unitization.
This follows e.g.\  by a small variant of the proof
of Lemma~\ref{L_7702_Incl} below.
\end{enumerate}
\end{remark}

\begin{lemma}\label{L_7702_Incl}
Suppose that $A$ is a closed subalgebra
of a nonunital approximately unital
Banach algebra~$B$,
and suppose that $A$ has a cai but is not unital.
Then for all $a \in A$ and $\lambda \in \Cdb$
we have
$\| a + \lambda 1 \|_{A^{1}} \leq \| a + \lambda 1 \|_{B^{1}}$.
\end{lemma}

\begin{proof}
Clearly
\[
\sup \big( \big\{ \| a c  + \lambda c \| \colon
        c \in \Ball (A) \big\} \big)
\leq \sup \big( \big\{ \| a c + \lambda c \| \colon
   c \in \Ball (B) \big\} \big),
\]
as desired.
\end{proof}

It is easy to find examples showing that the homomorphism above need
not be isometric, for example, with notation as in
Example \ref{mp} (or Example \ref{cone}) below,
$\Cdb e_2 \otimes c_0 \subseteq M_2^p \otimes c_0$.
However we have the following result.

\begin{lemma}\label{repuni}
Let $A$ and $B$ be nonunital
approximately unital Banach algebras.
Let $\ph \colon A \to B$ be a contractive
(resp.\  isometric) homomorphism.
Suppose that there is a cai $(e_t)_{t \in \Ld}$ for $A$
such that
$(\ph (e_t))_{t \in \Ld}$ is a cai for~$B$.
Then the obvious unital homomorphism $A^{1} \to B^{1}$
between the multiplier unitizations is contractive (resp.\  isometric).
\end{lemma}

\begin{proof}
If $a \in A$ and $\lambda \in \Cdb$
then
\[
\| \pi(a) \pi (e_t) + \lambda \pi (e_t) \|
  \leq \| a e_t + \lambda e_t \|.
\]
In the isometric case this is an equality.
Taking limits over $t$ and using
Remark \ref{R_7702_UnitizationFacts}~(\ref{R_7702_UnitizationFacts_AI})
gives the result.
\end{proof}

We recall two further standard facts.
The first is that the relation
$\Kdb (L^2 (X))^{**} = B (L^2 (X))$ is true
with $2$ replaced by any $p \in (1, \infty)$.

\begin{definition}\label{D_Arens}
We recall that the bidual $A^{**}$ of a
Banach algebra has in general two canonical products,
called the left and right Arens products \cite[Definition 1.4.1]{Plm2}.
We say that $A$ is {\emph{Arens regular}}
if these two products coincide.
\end{definition}

\begin{theorem}\label{T_7920_KStarStar}
Let $p \in (1, \infty)$
and let $(X, \mu)$ be a measure space.
Let $q \in (1, \infty)$ satisfy $\frac{1}{p} + \frac{1}{q} = 1$.
Then:
\begin{enumerate}
\item\label{T_7920_KStarStar_Dual}
There is an isometric isomorphism
$\Kdb (L^p (X, \mu))^{*}
   \to L^q (X, \mu) \widehat{\otimes} L^p (X, \mu)$
(projective tensor product)
which for $\rh \in L^p (X, \mu)$
and $\et \in L^q (X, \mu)$
sends $\et \otimes \rh$ to the operator
$\xi \mapsto \langle \xi, \et \rangle \rh$.
\item\label{T_7920_KStarStar_2nd}
There is an isometric algebra isomorphism
from $\Kdb (L^p (X, \mu))^{**}$ (with either Arens product)
onto $B (L^p (X, \mu))$ which extends the
inclusion $\Kdb (L^p (X, \mu)) \subseteq B (L^p (X, \mu))$.
\end{enumerate}
\end{theorem}

\begin{proof}
This follows from results of Grothendieck,
as described in the theorem on page 828 of~\cite{Palmer}
the discussion after that, and Theorems 1--3 there.
It is stated there that any Banach space $X$ such that $X$ and $X^*$
have the Radon-Nikodym property
and the approximation
property, satisfies Theorem 1 there
and the aforementioned theorem of Grothendieck,
giving~(\ref{T_7920_KStarStar_Dual}),
and also the case of~(\ref{T_7920_KStarStar_2nd})
for the first Arens product.
By Theorem 2 there
if $X$ is also reflexive then $\Kdb(X)$ is Arens regular,
so (\ref{T_7920_KStarStar_2nd}) holds as stated.  
See also the discussion on page 24, Corollary 4.13,
and Theorem 5.33 of \cite{Ryan} (and we thank M.~Mazowita for
this reference).
The explicit
formulas there are useful to check directly the Arens
product assertion.
One needs to know that $L^p (X, \mu)$ has the Radon-Nikodym property
and the approximation
property,
and this follows e.g.\  from
\cite[Example 4.5 and Corollary 5.45]{Ryan}.
\end{proof}

We remark that the last result and proof works
with $L^p$ replaced by any reflexive space with
the approximation property,
since reflexive spaces have the Radon-Nikodym property,
and indeed \cite[Corollary 4.7]{Ryan} implies that if $E$ is
reflexive and has the approximation property, then so does $E^*$.

By Theorem \ref{T_7920_KStarStar},
a net $(x_t)_{t \in \Lambda}$ in $B (L^p (X))$
converges weak* to~$x$ if and only if,
with $\frac{1}{p} + \frac{1}{q} = 1$,
\[
\sum_{k = 1}^{\infty} \, \langle x_t \xi_k, \eta_k  \rangle
 \to \sum_{k = 1}^{\infty} \, \langle x \xi_k, \eta_k \rangle
\]
for all $\xi_1, \xi_2, \ldots \in L^p (X)$
and $\eta_1, \eta_2, \ldots \in L^q (X)$
with $\sum_{k = 1}^{\infty} \, \| \xi_k \|_p \| \eta_k \|_q < \infty$
(or equivalently, by the usual trick,
with $\sum_{k = 1}^{\infty} \, \| \xi_k \|_p^p < \infty$
and $\sum_{k = 1}^{\infty} \, \| \eta_k \|_q^q < \infty$).
If $(x_t)_{t \in \Lambda}$ is bounded then by Banach duality principles
this is equivalent to $x_t \to x$ in the weak operator topology,
that is
$\langle x_t \xi, \eta  \rangle \to  \langle x \xi, \eta \rangle$
for all $\xi \in L^p (X)$ and $\eta \in L^q (X)$.
We will not use this here
but it is well known that essentially the usual $L^2$ operator proof
shows that the weak operator closure
of a convex set in $B (L^p ([0, 1]))$
equals the strong operator closure.
Indeed, for a Banach space $E$,
the strong operator continuous linear functionals on $B (E)$
are the same as those that are weak operator continuous.

The argument for the following well known lemma
will be reused
several times,
once in the form of an approximate identity bounded by~$M$
converging weak* to an identity in $A^{**}$ of norm at most~$M$.

\begin{lemma}\label{L_7917_AppIdConverge}
Let $A$ be an approximately unital Arens regular Banach algebra.
Then $A^{**}$ has an identity $1_{A^{**}}$ of norm~$1$,
and any cai for~$A$ converges weak* to~$1_{A^{**}}$.
\end{lemma}

\begin{proof}
The argument follows the proof of \cite[Proposition 2.5.8]{BLM}.
Since identities are unique if they exist,
it suffices to show that every subnet of any cai in~$A$
has in turn a subnet which converges to an identity for $A^{**}$.
Using Alaoglu's Theorem and since a subnet of a cai is a cai,
one sees that it is enough to show that if
$e \in A^{**}$ is the weak* limit of a cai,
then $e$ is an identity for $A^{**}$.
Multiplication on $A^{**}$ is separately weak* continuous
by \cite[2.5.3]{BLM},
so $e a = a e = a$ for all $a \in A$.
A second application of separate weak* continuity of multiplication
shows that this is true for all $a \in A^{**}$.
\end{proof}

\section{Notation, background, and general facts}\label{Sec_2}

\subsection{Dual and bidual algebras}\label{DuBi}

\begin{lemma}\label{L_7618_SecDual}
Let $p \in (1, \infty)$.
Let $A$ be an $L^p$-operator algebra
(resp.\  ${\mathrm{SQ}}_p$-operator algebra).
Then:
\begin{enumerate}
\item\label{L_7618_SecDual_AReg}
$A$ is Arens regular.
\item\label{L_7618_SecDual_WkSt}
Multiplication on $A^{**}$ is separately weak* continuous.
\item\label{L_7618_SecDual_Lp}
$A^{**}$ is an $L^p$-operator algebra
(resp.\  ${\mathrm{SQ}}_p$-operator algebra).
\end{enumerate}
\end{lemma}

\begin{proof}
We first recall (Theorem 3.3 (ii) of~\cite{He}, or \cite{LeM},
or the
remarks above Theorem 4.1 in \cite{Daws2})
that any ultrapower of $L^p$~spaces (resp.\  ${\mathrm{SQ}}_p$ spaces)
is again an $L^p$~space (resp.\  ${\mathrm{SQ}}_p$ space).
In the ${\mathrm{SQ}}_p$ space case this uses the well
known fact that ultrapowers behave well
with respect to subspaces and quotients
(this is obvious for subspaces,
for quotients see e.g.\  the proof of Proposition 6.5 in \cite{He}).
In particular,
such an ultrapower is reflexive,
so
every $L^p$~space (resp.\  ${\mathrm{SQ}}_p$ space) is superreflexive.
(See Proposition~1 of~\cite{Daws}.)

Now let $E$ be an $L^p$~space (resp.\  ${\mathrm{SQ}}_p$ space).
Theorem~1 of~\cite{Daws}
implies that $B ( E)$
is Arens regular.
The proof of Theorem~1 of~\cite{Daws}
embeds $B ( E )^{**}$ isometrically as a subalgebra
of $B (F)$
for a Banach space~$F$
obtained as an ultrapower of $l^r (E )$
for an arbitrarily chosen $r \in (1, \infty)$
(called $p$ in~\cite{Daws}).
We may choose $r = p$.
Then $l^r (E)$ is isometrically isomorphic
to an $L^p$~space (resp.\  ${\mathrm{SQ}}_p$ space).
Since ultrapowers of $L^p$~spaces
(resp.\  ${\mathrm{SQ}}_p$ spaces)
are $L^p$~spaces (resp.\  ${\mathrm{SQ}}_p$ spaces)
as we said at the start of this proof,
we have shown that $B ( E )^{**}$
is an $L^p$- (resp.\  ${\mathrm{SQ}}_p$-) operator algebra.

Now suppose that  $A \subseteq B (E )$
is a norm closed subalgebra.
Since $B (E )$ is Arens regular,
$A^{**}$ is a subalgebra of $B ( E  )^{**}$
and $A$ is Arens regular
by 2.5.2 in~\cite{BLM}.
It is now immediate that $A^{**}$ is an $L^p$-
(resp.\  ${\mathrm{SQ}}_p$-) operator algebra.
It also follows from 2.5.3 in~\cite{BLM}
that multiplication on $A^{**}$ is separately weak* continuous.
\end{proof}

It follows from \cite[Proposition 8]{Daws}
that $B ( L^1 (X, \mu) )$ is not Arens regular
unless $L^1 (X, \mu)$ is finite dimensional.

\begin{corollary}\label{C_7620_KLp}
Let $p \in (1, \infty)$
and let $(X, \mu)$ be a measure space.
Then multiplication on $B ( L^p (X, \mu) )$
is separately weak* continuous.
\end{corollary}

\begin{proof}
We have $\Kdb (L^p (X, \mu))^{**} \cong B (L^p (X, \mu))$
by Theorem \ref{T_7920_KStarStar}~(\ref{T_7920_KStarStar_2nd}).
\end{proof}

\begin{definition}\label{D_7620_DualLpOpAlg}
Let $p \in (1, \infty)$.
A {\emph{dual $L^p$-operator algebra}}
is a Banach algebra $A$ with a predual
such that there is a measure space $(X, \mu)$
and an isometric and weak* homeomorphic isomorphism
from $A$ to a weak* closed subalgebra of $B (L^p (X, \mu))$.
\end{definition}

By Corollary~\ref{C_7620_KLp}, the multiplication
on a dual $L^p$-operator algebra is separately weak* continuous.

\begin{corollary}\label{C_7620_2Dual}
Let $p \in (1, \infty)$
and let $A$ be an $L^p$-operator algebra.
Then $A^{**}$ is a dual $L^p$-operator algebra.
\end{corollary}

\begin{proof}
The embedding of $B ( L^p (X, \mu) )^{**}$
in Lemma~\ref{L_7618_SecDual} coming from the proof from \cite{Daws}
is easily checked
to be weak* continuous,
hence a weak* homeomorphism onto its range
by the Krein-Smulian theorem.
Hence $B ( L^p (X, \mu) )^{**}$ is a dual $L^p$-operator algebra.
It easily follows
that $A^{**}$ is too.
\end{proof}

\begin{lemma}\label{L_7620_SubOfDual}
Let $p \in (1, \infty)$
and let $A$ be a dual $L^p$-operator algebra.
Then:
\begin{enumerate}
\item\label{L_7620_SubOfDual_SbD}
The weak* closure of any subalgebra of~$A$
is a dual $L^p$-operator algebra.
\item\label{L_7620_SubOfDualUnital}
If $A$ is approximately unital
then $A$ is unital.
\end{enumerate}
\end{lemma}

\begin{proof}
The proofs are essentially the same as in the case $p = 2$,
as done in the proof of Proposition 2.7.4 in~\cite{BLM}.
\end{proof}

\subsection{States, hermitian elements,
 and real positivity}\label{StaRp}

We take states to be as at the beginning of Section~2 of~\cite{BOZ}.

\begin{definition}\label{Dstate}
If $A$ is a unital Banach algebra,
then a {\emph{state}} on~$A$
is a linear functional $\om \colon A \to \Cdb$
such that $\| \om \| = \om (1) = 1$.
If $A$ is an approximately unital Banach algebra,
we define a {\emph{state}} on~$A$
to be a linear functional $\om \colon A \to \Cdb$
such that $\| \om \| = 1$
and $\om$ is the restriction to~$A$ of a state on the
multiplier unitization~$A^{1}$
(Definition~\ref{D_7702_Unitization}).

We denote by $S (A)$ the set of all states on~$A$, and write
$Q (A)$ for the quasistate space
(that is, the set of  $\lambda \varphi$ for $\lambda \in [0, 1]$
and $\varphi \in S (A)$).

If ${\mathfrak{e}} = (e_t)_{t \in \Ld}$ is
a cai for $A$, define
\[
S_{\mathfrak{e}} (A)
 = \big\{ \om \in \Ball (A^*) \colon \om (e_t) \to 1 \big\}
\]
and define
\[
Q_{\mathfrak{e}} (A)
 = \big\{ \lambda \varphi \colon
    {\mbox{$\lambda \in [0, 1]$
      and $\varphi \in S_{\mathfrak{e}} (A)$}} \big\}.
\]
\end{definition}

If $A$ is a $C^*$-algebra (unital or not),
this definition gives the usual states and quasistates on~$A$.

The first part of the following definition
is Definition 2.6.1 of~\cite{Plm2}.

\begin{definition}\label{D_7702_NumRange}
Let $A$ be a unital Banach algebra,
and let $a \in A$.
We define the {\emph{numerical range}} of~$a$ to be
$\{ \varphi (a) \colon \varphi \in S (A) \}$.

If $E$ is a Banach space and $a \in B (E)$,
we define the {\emph{spatial numerical range of~$a$}} to be
\[
\big\{ \langle a \xi, \eta \rangle \colon
  {\mbox{$\xi \in  \Ball  (E)$ and $\eta \in \Ball  (E^*)$
    with $\langle  \xi, \eta \rangle = 1$}}
   \big\}.
\]
\end{definition}

There are other definitions of the numerical range.
For our purposes, only the convex hull is important,
and by Theorem~14 of~\cite{Lum}
the convex hulls of the
numerical range and the spatial numerical range of
an element in $B (E)$ are always the same.

\begin{definition}[see Definition 2.6.5 of~\cite{Plm2} and the
  preceding discussion]\label{DHerm}
Let $A$ be a unital Banach algebra,
and let $a \in A$.
We say that $a$ is {\emph{hermitian}}
if $\| \exp (i \ld a) \| = 1$ for all $\ld \in \Rdb$.

If $A$ is approximately unital we define the hermitian elements of~$A$
to be the elements in $A$ which are
hermitian in the multiplier unitization $A^{1}$
(Definition~\ref{D_7702_Unitization}).
\end{definition}

\begin{lemma}[see Theorem 2.6.7 of~\cite{Plm2}]\label{L_7702_SAH}
Let $A$ be a unital Banach algebra,
and let $a \in A$.
Then $a$ is hermitian if and only if
$\varphi (a) \in \Rdb$ for all states $\varphi$ of~$A$.
\end{lemma}

\begin{lemma}\label{L_7Z14_HerSubalg}
Let $A$ be an approximately unital Banach algebra,
and let $B \subseteq A$ be a closed subalgebra
which contains a cai for~$A$.
Let $a \in B$.
Then $a$ is hermitian as an element of~$B$
if and only if $a$ is hermitian as an element of~$A$.
\end{lemma}

\begin{proof}
By definition,
we work in the multiplier unitizations.
By Lemma~\ref{repuni},
$B^1$ is isometrically a unital subalgebra of~$A^1$.
The Hahn-Banach Theorem now shows that states on~$B^1$
are exactly the restrictions of states on~$A^1$.
So the conclusion follows from Lemma~\ref{L_7702_SAH}.
\end{proof}

\begin{definition}\label{D_locmeasdecomp}
Let $(X, \mu)$ be a measure space that is not $\sigma$-finite.
Recall that a function $f \colon X \to \Cdb$
is locally measurable if $f^{-1} (E) \cap F$ is measurable
for all Borel sets $E \subseteq \Cdb$ and
all subsets $F \subseteq X$ of finite measure.
Two such functions are ``locally a.e.\  equal''
if they agree a.e.\  on any such set $F$.
We interpret $L^\infty(X, \mu)$ as
$L^{\infty}_{\mathrm{loc}} (X, \mu)$,
the Banach space of
essentially bounded locally measurable scalar functions
mod local a.e.\  equality.

Further recall that a measure space $(X, \mu)$ is decomposable
if $X$ may be partitioned into sets $X_i$
of finite measure for $i \in I$
such that a set $F$ in $X$ is measurable if and only if $F \cap X_i$ is
measurable for every $i \in I$, and then
$\mu (F) = \sum_{i \in I} \, \mu (F \cap X_i)$.
\end{definition}

By e.g.\  the corollary on p.\  136 in \cite{Lacey},
any abstract $L^p$ space ``is'' decomposable, indeed it is isometric
to a direct sum of $L^p$ space of finite measures.
Thus, we may assume that all measure spaces $(X, \mu)$ are decomposable.

The following result is in the literature
with extra hypotheses such as if
$\mu$ is $\sigma$-finite \cite[Lemma 5.2]{G4}.
(See also e.g.\  Theorem~4 and the
remark following it in \cite{Tam},
when in addition  $p$ is not an even integer.)
We are not aware of a reference for the general case,
but it is probably folklore.

\begin{proposition}\label{P_7702_HermIsMult}
Let $p \in [1, \infty) \setminus \{ 2 \}$.
Let $(X, \mu)$
be a decomposable measure space,
and let $a \in B (L^p (X, \mu))$
be hermitian.
Then there is a real valued function
$f \in L^{\infty} (X, \mu)$
such that $a$ is multiplication by~$f$,
and such that $| f (x) | \leq \| a \|$ for all $x \in X$.
\end{proposition}

\begin{proof} 
Let $X = \coprod_{i \in I} X_i$
be a partition of $X$ into sets of finite measure
as in the discussion of decomposability above.
For $i \in I$ let $e_i \in B (L^p (X, \mu))$
be multiplication by $\ch_{X_i}$.
Since hermitian elements have numerical range contained in~$\Rdb$,
we can apply Theorem~6
of~\cite{Paya}
(see the beginning of~\cite{Paya} for the definitions and notation),
to see that $a$ commutes with $e_i$ for all $i \in I$.
One easily checks that $h = e_i a e_i$ is a hermitian element
of $B (L^p (X_i, \mu))$.
By the finite measure case of our result (\cite[Lemma 5.2]{G4}),
there is a real valued function
$f_i \in L^{\infty} (X_i, \mu)$
such that $h$ is multiplication by~$f_i$.

We can clearly assume that
$f_i$ is bounded by $\| e_i a e_i \| \leq \| a \|$.
Now define $f \colon X \to \Rdb$
by $f (x) = f_i (x)$ when $i \in I$ and $x \in X_i$.
Then $f$ is bounded by $\| a \|$,
and is measurable by the choice of the partition of~$X$.
For $i \in I$ and $\xi \in L^p (X_i, \mu)$
we clearly have $a \xi = f \xi$.
It follows from density of the linear span
of the subspaces $L^p (X_i, \mu)$ that $a$ is multiplication by~$f$.
\end{proof}

The $\sigma$-finite case
is deduced in \cite{G4}
from the finite measure case of Lamperti's Theorem
\cite[Theorem 3.2.5]{FJ}
by considering the invertible isometries $e^{i t h}$ for $t \in [0, 1]$.
We mention another approach
when $p$ is not an even integer.
It is known
(\cite[Corollary 1.8]{DlJrPl};
we thank Gideon Schechtman for this reference)
that $l^p$ doesn't
contain a two dimensional Hilbert space, and so 
Theorem~4 of~\cite{Tam}
implies our conclusion.
The same reference also
proves the result in the case that $\mu$ has no atomic part in~$X_i$.

\begin{definition}\label{D_7702_RealPos}
Let $A$ be a unital Banach algebra.
Let $a \in A$.
We say that $a$ is {\emph{accretive}}
or {\emph{real positive}}
if the numerical range of~$a$
is contained in the closed right half plane.
That is,
${\operatorname{Re}} (\varphi (a)) \geq 0$
for all states $\varphi$ of $A$.

If instead $A$ is approximately unital,
we define the real positive elements of~$A$
to be the elements in $A$ which are
real positive in the multiplier unitization $A^{1}$.

In both cases,
we denote the set of real positive elements of~$A$
by ${\mathfrak{r}}_A$.

Following p.\  8 of~\cite{BOZ},
we further define
\[
{\mathfrak{c}}_{A^*}
 = \big\{ \varphi \in A^* \colon
  {\mbox{${\operatorname{Re}} (\varphi (a)) \geq 0$
      for all $a \in {\mathfrak{r}}_A$}} \big\}.
\]
The elements of ${\mathfrak{c}}_{A^*}$
are called {\emph{real positive functionals}} on~$A$.
\end{definition}

For other equivalent conditions for real positivity,
see for example \cite[Lemma 2.4 and Proposition 6.6]{BSan}.

We warn the reader that ${\mathfrak{r}}_{A^{**}}$ is
defined after Lemma 2.5 of~\cite{BOZ} to be a proper subset
of the real positive elements in $A^{**}$,
the set of
elements of $A^{**}$ which are real positive
with respect to $(A^{1})^{**}$.
One should be careful with this ambiguity; fortunately it only
pertains to second duals
and seldom arises.
(Also see Proposition~\ref{scas}.)

\begin{lemma}\label{L_7Z14_RPosSubalg}
Let $A$ be an approximately unital Banach algebra,
and let $B \subseteq A$ be a closed subalgebra
which contains a cai for~$A$.
Let $a \in B$.
Then $a$ is real positive as an element of~$B$
if and only if $a$ is real positive as an element of~$A$.
\end{lemma}

\begin{proof}
The proof is the same as that of Lemma~\ref{L_7Z14_HerSubalg},
using Definition~\ref{D_7702_RealPos}
in place of Lemma~\ref{L_7702_SAH}.
\end{proof}

\begin{lemma}\label{L_7Z14_HerDiffRPos}
Let $p \in [1, \infty) \setminus \{ 2 \}$,
let $A$ be an approximately unital $L^p$~operator algebra,
and assume that the multiplier unitization $A^1$
is again an $L^p$~operator algebra.
Let $a \in A$ be hermitian.
Then there exist $b, c \in A$,
each of which is both hermitian and real positive,
such that
\begin{equation}\label{Eq_7Z14_HerDiffRPos}
a = b - c,
\qquad
b c = c b = 0,
\qquad
\| b \| \leq \| a \|,
\andeqn
\| c \| \leq \| a \|.
\end{equation}
\end{lemma}

By Lemma~\ref{mund} below,
the hypothesis that $A^1$ be an $L^p$~operator algebra
is automatic for $p \neq 1$.

It seems unlikely that Lemma~\ref{L_7Z14_HerDiffRPos}
holds for a general Banach algebra.

\begin{proof}[Proof of Lemma~{\rm{\ref{L_7Z14_HerDiffRPos}}}]
We may assume (using e.g.\  the corollary on p.~136 in \cite{Lacey})
that $(X, \mu)$ is a decomposable measure space
and $A^1$ is a unital subalgebra of $B (L^p (X, \mu))$.
Since $a$ is hermitian in~$A^1$,
Lemma~\ref{L_7Z14_HerSubalg} implies that
$a$ is hermitian in $B (L^p (X, \mu))$.
Proposition~\ref{P_7702_HermIsMult}
provides $f \in L^{\infty} (X, \mu)$
such that $a$ is multiplication by~$f$,
and such that $| f (x) | \leq \| a \|$ for all $x \in X$.

Choose a sequence $(r_n)_{n \in \Ndb}$
of polynomials with real coefficients
such that $r_n (\ld) \to \ld^{1/4}$
uniformly on $[ 0, \| a \|^2 ]$.
Adjusting by constants and scaling,
we may assume that $r_n (0) = 0$
and $| r_n (\ld) | \leq \| a \|^{1/2}$
for $\ld \in [ 0, \| a \|^2 ]$.
Set $s_n (\ld) = r_n (\ld^2)^2$
for $\ld \in [ - \| a \|, \, \| a \| ]$.
Then $(s_n)_{n \in \Ndb}$
is a sequence of polynomials with real coefficients
such that $r_n (\ld) \to |\ld|$
uniformly on $[ - \| a \|, \, \| a \| ]$.
Moreover, for all $n \in \Ndb$
we have $s_n (0) = 0$
and $0 \leq s_n (\ld) \leq \| a \|$
for all $\ld \in [ - \| a \|, \, \| a \| ]$.
In particular,
$s_n \circ f \to | f |$ uniformly on~$X$.

For $n \in \Ndb$,
define $d_n = s_n (a)$,
which is the multiplication operator by the function $s_n \circ f$,
and let $d$ be the multiplication operator by~$| f |$.
Then $d_n \in A$ for all $n \in \Ndb$ and $\| d_n - d \| \to 0$,
so $d \in A$
and $\| d \| \leq \| a \|$.
Therefore also
\[
b = \tfrac{1}{2} (d + a)
\andeqn
c = \tfrac{1}{2} (d - a)
\]
are in~$A$.
The conditions~(\ref{Eq_7Z14_HerDiffRPos})
are clearly satisfied.

The multiplication operator map
from $L^{\infty} (X, \mu)$ to $B (L^p (X, \mu))$
is an isometric unital homomorphism.
(Recall the convention that
we are using $L^{\infty}_{\mathrm{loc}} (X, \mu)$ here.)
The functions $\tfrac{1}{2} (| f | + f)$
and $\tfrac{1}{2} (| f | - f)$
are nonnegative,
hence both hermitian and real positive in $L^{\infty} (X, \mu)$
(because $L^{\infty} (X, \mu)$ is a $C^*$-algebra).
Lemma~\ref{L_7Z14_HerSubalg} and Lemma~\ref{L_7Z14_RPosSubalg}
therefore imply that their multiplication operators $b$ and~$c$
are both hermitian and real positive in $B (L^p (X, \mu))$.
A second application of these lemmas
shows that the same holds in~$A^1$.
By definition, this is also true in~$A$.
\end{proof}

\begin{definition}\label{D_7702_FA}
Let $A$ be a unital or approximately unital Banach algebra.
Taking $1$ to be the identity of $A^{1}$
in the approximately unital case,
we define
${\mathfrak{F}}_A = \{ a \in A \colon \| 1 - a \| \leq 1 \}$.
\end{definition}

\begin{proposition}[Proposition 3.5 of~\cite{BOZ}]\label{P_7702_rAAndFA}
Let $A$ be a unital or approximately unital Banach algebra.
Then, in the notation of Definition~{\rm{\ref{D_7702_RealPos}}}
and Definition~{\rm{\ref{D_7702_FA}}},
we have ${\mathfrak{r}}_A = {\overline{\Rdb_+ {\mathfrak{F}}_A}}$.
\end{proposition}

We recall some facts about roots of elements of ${\mathfrak{r}}_A$.

\begin{definition}\label{N_7919_nthroot}
Let $A$ be a unital or approximately unital Banach algebra,
let $b \in {\mathfrak{r}}_A$, and let $t \in (0, 1)$.
If $A$ is unital,
we denote by $b^t$
the element $b_t$ constructed in \cite[Theorem 1.2]{LRS}.
If $A$ is approximately unital,
let $A^1$ be the multiplier unitization,
recall that $b \in {\mathfrak{r}}_{A^1}$ by definition,
and define $b^t$ to be as above but evaluated in~$A^1$.
\end{definition}

The conditions required in \cite[Theorem 1.2]{LRS}
are weaker than here,
but this case is all we need.
Such noninteger powers,
for the special case $\| b - 1 \| < 1$ and when $A$ is commutative,
seem to have first appeared in Definition~2.3 of~ \cite{Estl}.
A discussion relating this definitions to others,
and giving a number of properties,
is contained in~\cite{BOZ},
from Proposition~3.3 through Lemma~3.8 there.
In particular, $(b^{1/n})^n = b$ and $t \mapsto b^t$ is continuous.
For later use, we recall several of these properties
and state a few other facts not given explicitly in~\cite{BOZ}.

\begin{proposition}\label{L_7Z16_RootProp}
Let $A$ be a unital or approximately unital Banach algebra,
and let $a \in {\mathfrak{r}}_A$.
\begin{enumerate}
\item\label{L_7Z16_RootProp_Series}
If $t \in (0, 1)$ and $\| b - 1 \| \leq 1$
(that is, $b \in {\mathfrak{F}}_A$),
then
\[
b^t = 1 + \sum_{k = 1}^{\infty}
   \frac{t (t - 1) (t - 2) \cdots (t - k + 1)}{k!} (-1)^k (1 - b)^k,
\]
with absolute convergence.
\item\label{L_7Z16_RootProp_Scalar}
If $t \in (0, 1)$ and $\ld \in (0, \infty)$
then $(\ld x)^t = \ld^t x^t$.
\item\label{L_7Z16_RootProp_bound}
For all $t \in (0, 1)$,
$\| a^t \| \leq 2 \| a \|^t / (1 - t)$.
\item\label{L_7Z16_RootProp_PApp}
For all $t \in (0, 1)$,
$a^{t}$ is a norm limit of polynomials in~$a$ with no constant term.
\item\label{L_7Z16_RootProp_Comm}
For all $t \in (0, 1)$,
$a^{t} a = a a^{t}$.
\item\label{L_7Z16_RootProp_ApprI}
$\lim_{n \to 0} \| a^{1/n} a - a \|
 = \lim_{n \to 0} \| a a^{1/n} - a \|
 = 0$.
\item\label{L_7Z16_RootProp_Norm}
If $a \in {\mathfrak{F}}_A$ and $t \in (0, 1)$,
then $\| 1 - a^t \| \leq 1$.
\end{enumerate}
\end{proposition}

\begin{proof}
For part~(\ref{L_7Z16_RootProp_Series}),
see the proof of \cite[Proposition 3.3]{BOZ}
and the discussion in and before
the Remark before \cite[Lemma 3.6]{BOZ}.

For~(\ref{L_7Z16_RootProp_Scalar}),
see the discussion after \cite[Proposition 3.5]{BOZ}.

Part~(\ref{L_7Z16_RootProp_bound})
is a slight weakening of the second estimate in
Lemma~3.6 of~ \cite{BOZ}.

Part~(\ref{L_7Z16_RootProp_PApp}) holds for $a \in {\mathfrak{F}}_A$
by the proof of Proposition~3.3 of~\cite{BOZ}.
By~(\ref{L_7Z16_RootProp_Scalar}),
it holds for $a \in \Rdb_+ {\mathfrak{F}}_A$.
By continuity (Corollary~1.3 of~\cite{LRS}),
it holds for $a \in {\overline{\Rdb_+ {\mathfrak{F}}_A}}$.
Apply Proposition~\ref{P_7702_rAAndFA}.

Part~(\ref{L_7Z16_RootProp_Comm})
is immediate from Part~(\ref{L_7Z16_RootProp_PApp}).
Part~(\ref{L_7Z16_RootProp_ApprI})
is Lemma~3.7 of~ \cite{BOZ}.

For~(\ref{L_7Z16_RootProp_Norm}),
use~(\ref{L_7Z16_RootProp_Series}),
together with
\[
\frac{t (t - 1) (t - 2) \cdots (t - k + 1)}{n!} (-1)^k < 0
\]
for $k = 1, 2, \ldots$
and
\[
\sum_{k = 1}^{\infty}
  \frac{t (t - 1) (t - 2) \cdots (t - k + 1)}{k!} (-1)^k = - 1.
\]
This completes the proof.
\end{proof}

\begin{lemma}\label{ssu}
Suppose that $A$ is a closed subalgebra of an approximately unital
Banach algebra $B$, and suppose that $A$ has a cai.
Then ${\mathfrak{F}}_B \cap A \subseteq {\mathfrak{F}}_A$
and ${\mathfrak{r}}_B \cap A \subset {\mathfrak{r}}_A$.
\end{lemma}

\begin{proof}
The first statement follows easily from Lemma~\ref{L_7702_Incl}.
The second follows from the first and the relations
${\mathfrak{r}}_A = \overline{\Rdb_+ {\mathfrak{F}}_A}$
and ${\mathfrak{r}}_B = \overline{\Rdb_+ {\mathfrak{F}}_B}$
(Proposition \ref{P_7702_rAAndFA}).
\end{proof}

\begin{proposition}\label{uninc} 
Let $B$ be a nonunital approximately unital Banach algebra,
and let $A \subseteq B$ be a
closed subalgebra
which contains a cai for~$B$.
Then:
\begin{enumerate}
\item\label{uninc_Incl}
$A^1 \subseteq B^1$ isometrically.
\item\label{uninc_RP}
${\mathfrak F}_A = {\mathfrak F}_B \cap A$
and ${\mathfrak{r}}_A = {\mathfrak{r}}_B \cap A$.
\item\label{uninc_ExtState}
Every state or quasistate
on $A$ may be extended to a state or quasistate on~$B$.
\end{enumerate}
\end{proposition}

\begin{proof}
Part~(\ref{uninc_Incl}) is Lemma~\ref{repuni}.
That ${\mathfrak F}_A = {\mathfrak F}_B \cap A$
is immediate from~(\ref{uninc_Incl}),
and now ${\mathfrak{r}}_A = {\mathfrak{r}}_B \cap A$
by e.g.\  Proposition \ref{P_7702_rAAndFA}.
Part~(\ref{uninc_ExtState}) is obvious from~(\ref{uninc_Incl}),
Definition~\ref{Dstate}, and the Hahn-Banach Theorem.
\end{proof}

\begin{lemma}\label{havec}
Suppose that an Arens regular
Banach algebra $A$ has a cai
and also has a real positive approximate identity.
Then $A$ has a cai in ${\mathfrak{F}}_A$.
If in addition $A$ has a countable bounded approximate identity,
then $A$ has a cai in ${\mathfrak{F}}_A$
which is a sequence.
\end{lemma}

\begin{proof}
Corollary 3.9 of~\cite{BOZ} implies that
$A$ has an approximate identity in ${\mathfrak{F}}_A$.
Since ${\mathfrak{F}}_A$ is bounded,
one may then use the argument in the second paragraph
of the proof of \cite[Proposition 6.13]{BSan}
to see that $A$ has a cai $(e_t)_{t \in \Ld}$ in ${\mathfrak{F}}_A$.
If in addition $A$ has a countable bounded approximate identity,
then one can use Corollary 32.24 of~\cite{HwRs2}
and its analog on the right
(see also Theorem 4.4 in \cite{BOZ}) to find
$x, y \in A$ with $A = \overline{xA} = \overline{Ay}$.
Choose $t_1, t_2, \ldots \in \Ld$
with $t_1 < t_2 < \cdots$ and
$\| f_{t_k} x - x \| + \| y f_{t_k} - y \| < 2^{-k}$; then
$(f_{t_k})$ is a countable cai in ${\mathfrak{F}}_A$.
\end{proof}

\begin{corollary}\label{bigmove}
Suppose that $A$ is an approximately unital Arens regular
Banach algebra. 
If  $1_{A^{**}}$ is a weak* limit
of 
a bounded net of real positive elements in $A$,
then $A$ has a real positive cai.
\end{corollary}

\begin{proof} 
By a standard convexity argument,
or e.g.\  \cite[Lemma 2.1]{BOZ},
$A$ has a real positive bounded approximate
identity.
It follows from Lemma \ref{havec}
that $A$ has a cai in ${\mathfrak{F}}_A$.
\end{proof}

The hypothesis in the last result about $1_{A^{**}}$
being a weak* limit holds
if $A$ has one of the Kaplansky density type properties,
e.g.\  properties  (\ref{Item_Kap_rposA})--(\ref{Item_KpD_FA})
in Proposition \ref{L_7618_KapScal}.
See also the proof of Proposition 6.4 in \cite{BOZ}.

\subsection{More on the multiplier unitization}\label{Sec_2m}

The multiplier unitization was defined in
Definition \ref{D_7702_Unitization}.

\begin{lemma}\label{mund}
Let $E$ be a Banach space.
Suppose that $A$ is
a nonunital closed approximately unital subalgebra of $B (E)$
which acts nondegenerately on $E$.
Then the multiplier unitization of
$A$ is isometrically isomorphic to $A + \Cdb 1_E$,
where $1_E$ is the identity operator on $E$.
\end{lemma}

\begin{proof}
For $a, c \in A$ and $\lambda \in \Cdb$,
we clearly have
\[
\| a c + \lambda c \|
  = \| (a + \lambda 1_E) c \|
  \leq \| a + \lambda 1_E \| \| c \|.
\]
So $\| a  + \lambda 1 \|_{A^1} \leq
\| a + \lambda  1_E \|$.
The reverse inequality
follows from the fact that if $(e_t)_{t \in \Ld}$
is a cai for~$A$,
then $a e_t + \lambda e_t \to a + \lambda 1_E$
in the strong operator topology on $B (E)$.
\end{proof}

\begin{lemma}\label{bicr}
Suppose that $A$ is an approximately unital
Arens regular Banach algebra,
and let ${\mathfrak{e}} = (e_t)_{t \in \Ld}$ be a cai for $A$.
Then:
\begin{enumerate}
\item\label{bicr_1}
The multiplier unitization of
$A$ is isometrically isomorphic to $A + \Cdb 1_{A^{**}}$ in $A^{**}$.
\item\label{bicr_2}
With $S_{\mathfrak{e}} (A)$ as defined
in Definition $\mathrm{\ref{Dstate}}$,
and identifying $A^*$
with the weak* continuous functionals on $A^{**}$,
we have
\[
S_{\mathfrak{e}} (A)
 = \big\{ \om \in S (A^{**}) \colon
  {\mbox{$\om$ is weak* continuous}} \big\}
\]
(the normal state space of $A^{**}$).
\item \label{bicr_2b}
$S_{\mathfrak{e}} (A)$ and $S (A)$ both span $A^*$,
and both separate the points of $A$.
\item\label{bicr_3}
In the notation found before Lemma ${\mathrm{2.6}}$
of \textrm{\cite{BOZ}}
and in Definition~{\rm{\ref{D_7702_RealPos}}},
we have
\[
{\mathfrak{r}}^{\mathfrak{e}}_A = {\mathfrak{r}}_A
\qquad \mbox{and} \qquad
{\mathfrak{c}}^{\mathfrak{e}}_{A^*} = {\mathfrak{c}}_{A^*}.
\]
\item\label{bicr_4}
If $A$ is also nonunital
then  $\{ \varphi |_A \colon \varphi \in S (A^1) \}$ is the
weak* closure in $A^*$ of any one of the following
sets in Definition $\mathrm{\ref{Dstate}}$:
$S (A)$, $S_{\mathfrak{e}} (A)$, $Q (A)$, and $Q_{\mathfrak{e}} (A)$.
\end{enumerate}
\end{lemma}

\begin{proof}
The proof of~(\ref{bicr_1})
is essentially the same as the proof of Lemma~\ref{mund}:
for $a, c \in A$ and $\lambda \in \Cdb$,
clearly
\[
\| ac + \lambda c  \| = \| (a + \lambda 1_{A^{**}}) c \|
\leq \| a + \lambda 1_{A^{**}} \| \| c \|.
\]
So $\| a  + \lambda 1 \|_{A^1} \leq \| a + \lambda 1_{A^{**}} \|$.
The reverse inequality
follows from the fact that
if $(e_t)_{t \in \Ld}$ is a cai,
then Lemma \ref{L_7917_AppIdConverge}
implies that $a e_t + \lambda e_t \to a + \lambda 1_{A^{**}}$ weak*.

For~(\ref{bicr_2}),
since $e_t \to 1$ weak* in $A^{**}$
by Lemma \ref{L_7917_AppIdConverge},
it is clear that weak* continuous states on $A^{**}$
restrict to elements of $S_{\mathfrak{e}} (A)$.
For the reverse inclusion,
let $\om \in S_{\mathfrak{e}} (A)$.
Then $\om^{**}$ is a weak* continuous functional on $A^{**}$
and $\| \om^{**} \| = 1$.
That $\om^{**} (1) = 1$
follows from weak* continuity of $\om^{**}$
and the weak* convergence $e_t \to 1$.

The assertion about $S_{\mathfrak{e}} (A)$ in
(\ref{bicr_2b})
follows from part~(\ref{bicr_2})
and Theorem~2.2 of~\cite{Mag},
according to which the normal state space of $A^{**}$
spans $A^*$
and separates the points of~$A$.
The second assertion in~(\ref{bicr_2b})
follows from the first assertion
and the inclusion $S_{\mathfrak{e}} (A) \subseteq S (A)$,
which is in Lemma~2.2 of~\cite{BOZ}.

We prove~(\ref{bicr_3}).
We need only prove
${\mathfrak{r}}^{\mathfrak{e}}_A \subseteq {\mathfrak{r}}_A$,
since the reverse inclusion holds by definition,
and equality implies
${\mathfrak{c}}_{A^*} = {\mathfrak{c}}^{\mathfrak{e}}_{A^*}$
by definition.
So let
$a \in {\mathfrak{r}}^{\mathfrak{e}}_A$
and let $\om \in S (A)$.
By definition,
$\om$ extends to a state $\om^1$ on~$A^1$.
By part~(\ref{bicr_1})
we have $A^1 \subseteq A^{**}$,
so the Hahn-Banach Theorem provides an extension of $\om^1$
to a state $\ph$ on $A^{**}$.
Use weak* density of the normal states in $S (A^{**})$
(which follows from Theorem~2.2 of~\cite{Mag})
to find a net $(\ph_t)_{t \in \Ld}$
in the normal state space of~$A^{**}$
which converges weak* to~$\ph$.
Now
$\Re ( \om (a) ) = \lim_t \Re ( \ph_t (a) ) \geq 0$.
So $a \in {\mathfrak{r}}_A$.

Finally, we prove~(\ref{bicr_4}).

It follows from  \cite[Lemma 2.6]{BOZ} that, with overlines
denoting weak* closures, we have
\[
\overline{S (A)} = \overline{Q (A)}
 \subseteq \{ \varphi |_A \colon \varphi \in S (A^1) \}.
\]
Also, $\{ \varphi |_A \colon \varphi \in S (A^1) \}$
is shown to be weak* closed in the proof of that
lemma.

Now suppose that $\varphi \in S (A^1)$
and set $\psi = \varphi |_A$.
Use the Hahn-Banach Theorem to extend $\varphi$
to a state $\rho$ on $A^{**}$.
Use again weak* density of the normal states in $S (A^{**})$
to find a net $(\psi_t)_{t \in \Ld}$
in the normal state space of~$A^{**}$
which converges weak* to~$\rho$.
Set $\varphi_t = \psi_t |_A$ for $t \in \Ld$.
For $a \in A$ we then have
\[
\varphi_t (a) = \psi_t (a) \to \psi(a) = \varphi (a).
\]
By part~(\ref{bicr_2}),
this shows that $\psi$ is in the weak* closure
of $S_{\mathfrak{e}} (A)$.
Since $S_{\mathfrak{e}} (A) \subseteq S (A) \subseteq Q (A)$
and $S_{\mathfrak{e}} (A) \subseteq Q_{\mathfrak{e}} (A)$,
the assertion follows.
\end{proof}

The set ${\mathfrak{r}}_{A^{**}}$, as
defined on p.\  11 of \textrm{\cite{BOZ}},
may be a proper subset
of the accretive elements in $A^{**}$, even for
approximately unital $L^p$-operator algebras.
In fact,
the identity $e$ of $A^{**}$ is certainly accretive in $A^{**}$,
but need not be accretive in $(A^1)^{**}$.
(Equivalently,
by Lemma \ref{L_7719_BicId}~(\ref{L_7719_BicId_RealPos}),
we need not have $\| 1 - e \| \leq 1$.)
This happens for $A = \Kdb (L^p ([0, 1]))$,
by Proposition \ref{knsca}. 
However, it follows from the later result Proposition~\ref{scas}
(and Proposition~\ref{2extension} (\ref{2extension_2}))
that ${\mathfrak{r}}_{A^{**}}$, as
defined on p.\  11 of \textrm{\cite{BOZ}},
equals the accretive elements in $A^{**}$
if $A$ is a scaled approximately unital $L^p$-operator algebra.

\begin{remark}\label{sconv}
The sets  $S_{\mathfrak{e}} (A)$ and $Q_{\mathfrak{e}} (A)$
are easily seen to be convex in $A^*$.
We do not know whether $S (A)$ and $Q (A)$ are necessarily convex
if $A$ is a general approximately unital Arens regular Banach algebra,
since convex combinations
of norm $1$ functionals may have norm strictly less than~$1$.
However they are convex if $A$ is
an approximately unital $L^p$-operator algebra,
since Corollary \ref{HBsm}~(\ref{HBsm_EqState}) below
implies convexity of $S (A)$, and this implies convexity of $Q (A)$.
\end{remark}

\begin{proposition}\label{muau}
Let $p \in (1, \infty)$.
The multiplier unitization of an approximately unital
$L^p$-operator algebra
is an $L^p$-operator algebra.
\end{proposition}

\begin{proof}
This follows from Lemma \ref{bicr}~(\ref{bicr_1}) and the fact
(Lemma \ref{L_7618_SecDual}~(\ref{L_7618_SecDual_Lp}))
that
biduals of $L^p$-operator algebras are $L^p$-operator algebras
(or from Lemmas \ref{mund} and \ref{2deg}).
\end{proof}

Similarly, for $p \in (1, \infty)$
the multiplier unitization of an approximately unital
${\mathrm{SQ}}_p$-operator algebra
is an ${\mathrm{SQ}}_p$-operator algebra.

The multiplier algebra ${\operatorname{M}} (A)$,
and the left and right multiplier algebras
${\operatorname{LM}} (A)$ and ${\operatorname{RM}} (A)$,
of an approximately $L^p$-operator algebra
may be defined to be subsets of $A^{**}$
just as in the operator algebra case.
Then the multiplier unitization $A^1$
is contained in ${\operatorname{M}} (A)$
isometrically and unitally.
If $A$ is represented isometrically and nondegenerately on $L^p (X)$
then, just as in the operator algebra case,
${\operatorname{M}} (A)$,  ${\operatorname{LM}} (A)$,
and ${\operatorname{RM}} (A)$ may be identified isometrically
as Banach algebras with the usual subalgebras of $B (L^p (X))$.
See Theorem 3.19 in \cite{G4}, and the discussion in that paper.
One can also, for example,
copy the proof of Theorem 2.6.2 of~\cite{BLM}
for ${\operatorname{LM}} (A)$,
and later results in Section~2.6 of~\cite{BLM}
for ${\operatorname{RM}} (A)$ and ${\operatorname{M}} (A)$.

In particular,
${\operatorname{M}} (A)$,  ${\operatorname{LM}} (A)$,
and ${\operatorname{RM}} (A)$
are all unital  $L^p$-operator algebras.
Similarly, ${\operatorname{LM}} (A)$
can be identified with the
algebra of bounded right $A$-module endomorphisms of~$A$,
as usual.
One may also check that
the useful principle in \cite[Proposition 2.6.12]{BLM}
holds for approximately $L^p$-operator algebras,
with the same proof.
(Also see Theorem 3.17 in \cite{G4}.)

\subsection{Idempotents}\label{Idems}

\begin{definition}\label{D_bicontrinvisom}
We recall that if $A$ is a unital Banach algebra,
then an idempotent $e \in A$
is called {\emph{bicontractive}}
if $\| e \| \leq 1$ and  $\| 1 - e \| \leq 1$.
We collect some standard facts related to bicontractive idempotents.
We say that an element $s$ of a unital Banach algebra~$A$
is an {\emph{invertible isometry}}
if $s$ is invertible, $\| s \| = 1$, and $\| s^{-1} \| = 1$.
\end{definition}

\begin{lemma}\label{L_7719_BicId}
\begin{enumerate}
\item\label{L_7719_BicId_HToIso}
Let $A$ be a unital Banach algebra
and let $e \in A$ be a hermitian idempotent.
Then $1 - 2 e$ is an invertible isometry of order~$2$.
\item\label{L_7719_BicId_HToBicr}
Let $A$ be a unital Banach algebra.
Then every hermitian idempotent in $A$ is bicontractive.
\item\label{L_7719_BicId_BicrOnLp}
Let $p \in [1, \In)$,
let $(X, \mu)$ be a measure space,
and let $e \in B (L^p (X, \mu))$
be an idempotent.
Then $e$ is bicontractive if and only if
$1 - 2 e$ is an invertible isometry.
\item\label{L_7719_BicId_RealPos}
Let $A$ be a unital Banach algebra
and let $e \in A$ be an idempotent.
Then $e$ is real positive if and only if $1 - e$
is contractive ($\| 1 - e \| \leq 1$).
\end{enumerate}
\end{lemma}

The converse of (\ref{L_7719_BicId_HToBicr})
is false,
even in $L^p$~operator algebras.
See Lemma 6.11 of~\cite{PV},
which is just the idempotent $e_2$ of Example~\ref{mp}
for $p \neq 2$.

Part~(\ref{L_7719_BicId_BicrOnLp})
fails in general unital Banach algebras.
This failure is well known,
and our Example~\ref{mp2}
contains an explicit counterexample.

\begin{proof}[Proof of Lemma~{\rm{\ref{L_7719_BicId}}}]
For~(\ref{L_7719_BicId_HToIso}),
by definition we have
\[
\big\| 1 + [\exp (i \lambda) - 1] e \big\|
 = \| \exp (i \lambda e) \|
 \leq 1
\]
for all $\lambda \in \Rdb$.
Setting
$\lambda = \pi$ gives $\| 1 - 2 e \| \leq 1$.
One checks immediately that $(1 - 2 e)^2 = 1$,
so in fact $\| 1 - 2 e \| = 1$.
The rest of~(\ref{L_7719_BicId_HToIso}) follows easily.

Part~(\ref{L_7719_BicId_HToBicr})
follows from Lemma 6.6 of~\cite{PV}.

We prove~(\ref{L_7719_BicId_BicrOnLp}).
The forward direction
follows from \cite[Theorem 2.1]{BLac}
(or, when $\mu (X) = 1$,
from the corollary on page~11 of~\cite{ByrneS}).
Conversely,
if $\| 1 - 2 e \| \leq 1$
then
\[
\| e \|
 = \big\| \tfrac{1}{2} [ 1 - (1 - 2 e) ] \big\|
 \leq \tfrac{1}{2} ( \| 1 \| + \| 1 - 2 e \| )
 \leq 1,
\]
and the proof that $\| 1 - e \| \leq 1$ is similar.

Part~(\ref{L_7719_BicId_RealPos})
is \cite[Lemma 3.12]{BOZ}.
\end{proof}

\begin{definition}\label{Dorder}
We define two order relations on idempotents $e, f$
in a Banach algebra~$A$.
We write $e \leq_{\mathrm{r}} f$ if $f e = e$
and $e \leq f$ if $e f = f e = e$.
\end{definition}

If $A$ is a subalgebra of $B (E)$ then,
viewing these idempotents as operators on $E$,
then $e \leq_{\mathrm{r}} f$ simply says that
$\Ran (e) \subseteq \Ran (f)$.
The second relation is the ordering
considered in e.g.\  \cite[Section 6]{PV}.

Clearly $e \leq f$ and $f \leq e$ imply $e = f$.
This isn't true for the relation $\leq_{\mathrm{r}}$.

\begin{lemma}\label{orderidemp}
Let $p \in (1, \infty)$,
and let $A$ be an
approximately unital $L^p$-operator algebra.
Let $e, f \in A$ be idempotents.
Assume that $e$ and $f$ are both contractive
or both real positive.
Then:
\begin{enumerate}
\item\label{orderidemp_RevProd}
$f e = e$ if and only if $e f = e$.
\item\label{orderidemp_RtoNone}
$e \leq_{\mathrm{r}} f$ if and only if $e \leq f$.
\end{enumerate}
\end{lemma}

\begin{proof}
Part (\ref{orderidemp_RtoNone})
is immediate from part~(\ref{orderidemp_RevProd}),
so we just prove part~(\ref{orderidemp_RevProd}).

By definition
(see Definition~\ref{D_7702_RealPos}),
we may work in the multiplier unitization $A^{1}$,
which is a unital $L^p$-operator algebra
by Proposition~\ref{muau}.
So we can assume that there is a measure space $(X, \mu)$
such that $A$ is a unital subalgebra of $B (L^p (X, \mu))$.

First suppose that $e$ and $f$ are contractive.
Assume that $f e = e$.
Then $e f$ is necessarily an idempotent,
and is clearly contractive.
Clearly $\Ran (e f) \subseteq \Ran (e)$.
Since $e f e = e^2 = e$,
we have $\Ran (e) \subseteq \Ran (e f)$.
By \cite[Theorem 6]{CS},
the range of
a contractive idempotent on a smooth space determines
the idempotent.
So $e f = e$, as desired.

Next assume that $e f = e$.
Let $q \in (1, \In)$
satisfy $\frac{1}{p} + \frac{1}{q} = 1$.
Then $e^*, f^* \in B (L^q (X, \mu))$ are contractive idempotents
such that $f^* e^* = e^*$.
The case already considered implies $e^* f^* = e^*$,
whence $f e = e$.

Now suppose that $e$ and $f$ are real positive.
Then $1 - e$ and $1 - f$ are contractive idempotents
by Lemma \ref{L_7719_BicId}~(\ref{L_7719_BicId_RealPos}).
So $(1 - e) (1 - f)  = 1 - f$ if and only if
$(1 - f) (1 - e)  = 1 - f$
by the contractive case.
Expanding and rearranging,
we get $f e = e$ if and only if $e f = e$.
\end{proof}

\subsection{Representations}\label{Reps}
We say a few words on representations.

\begin{lemma}\label{L_7702_WkStarExt}
Let $p \in (1, \infty)$,
let $A$ be an $L^p$-operator algebra,
let $X$ be a measure space,
let $M$
be a weak* closed subalgebra of $B (L^p (X))$,
and let $\pi \colon A \to M$
be a contractive homomorphism.
Then there exists a
unique weak* continuous contractive homomorphism
$\widetilde{\pi} \colon A^{\ast\ast} \to M$
which extends~$\pi$.
\end{lemma}

\begin{proof}
The proof is the same as for the operator algebra
case (2.5.5 in \cite{BLM}, but without the matrix norms)
and using
Lemma \ref{L_7618_SecDual}.
\end{proof}

Let $\pi \colon A \to B (L^p (X))$
be a contractive representation
of an approximately unital $L^p$-operator algebra.
Then $E = {\overline{\spn}} \big( \pi (A) (L^p (X)) \big)$
may not be an $L^p$-space
on a subset of~$X$.
Indeed, in Example \ref{mp} below,
$\Ran (e_2)$ is not an $L^p$-space on a subset.
However it is isometric to an $L^p$ space,
as we will see next.

Some of the following
follows from \cite[Proposition 1.8]{Joh}
(we thank Eusebio Gardella for this reference) and
\cite[Theorem 3.12, Corollary 3.13]{G4}
(see also \cite[Section 2]{PV}),
but for completeness we give a self-contained proof.

\begin{lemma}\label{2deg}
Let $p \in (1, \In)$,
let $A$ be an approximately unital Banach algebra,
and let  $\pi \colon A \to B (L^p (X))$
be a contractive representation.
Set
$E = {\overline{\spn}} \big( \pi (A) (L^p (X)) \big)$.
Then there exists a unique contractive idempotent $f \in B (L^p (X))$
whose range is~$E$.
Moreover, $E$ and $f$ have the following properties.
\begin{enumerate}
\item\label{2deg_AI}
For every cai $(e_t)_{t \in \Ld}$ for $A$,
the net $(\pi (e_t))_{t \in \Ld}$ converges to $f$
in both the weak* topology and the strong operator topology
on $B (L^p (X))$.
\item\label{2deg_faf}
For all $a \in A$ we have $\pi (a) = f \pi (a) f$.
\item\label{2deg_Compr}
The compression of $\pi$ to $E$
is a contractive representation,
which is isometric if $\pi$ is isometric.
\item\label{2deg_Nondeg}
The compression of $\pi$ to $E$
is nondegenerate.
\item\label{2deg_Lp}
$E$ is linearly isometric to an $L^p$ space.
\end{enumerate}
\end{lemma}

\begin{proof}
Let $q \in (1, \In)$
satisfy $\frac{1}{p} + \frac{1}{q} = 1$.

We claim that if $(e_t)_{t \in \Ld}$ is a cai in~$A$
such that $( \pi (e_{t}))_{t \in \Ld}$
converges weak* to some $f \in B (L^p (X))$,
then $f$ is a contractive idempotent whose range is~$E$.
Assume the claim has been proved.
Since $L^p (X)$ is a smooth space,
such an idempotent is unique by \cite[Theorem 6]{CS}.
The argument of Lemma \ref{L_7917_AppIdConverge},
with this uniqueness statement in place of
uniqueness of the identity in $A^{**}$,
shows that such an idempotent $f$ exists
and that for any cai $(e_t)_{t \in \Ld}$ in~$A$,
we have $\pi (e_t) \to f$ weak*.

We prove the claim.
We have $\| f \| \leq 1$ and
$\langle f \pi (a) \xi, \, \eta \rangle
 = \langle \pi (a) \xi, \, \eta \rangle$
for all $a \in A$, $\xi \in L^p (X)$,
and $\et \in L^q (X)$.
It follows that $f \xi = \xi$ for all $\xi \in E$.
So $E \subseteq {\operatorname{Ran}} (f)$.
Also,
if $\eta \in E^{\perp} \subseteq L^q (X)$,
then
$\langle f \xi, \eta \rangle
 = \lim_{t} \langle \pi (e_{t}) \xi, \, \eta \rangle
 = 0$.
Thus $E^{\perp} \subseteq {\operatorname{Ran}} (f)^{\perp}$,
whence ${\operatorname{Ran}} (f) \subseteq E$.
The claim is proved.
We now have the main statement,
and weak* convergence in~(\ref{2deg_AI}).

Part~(\ref{2deg_Lp}) follows from the fact
(Theorem~3 in Section~17 of~\cite{Lacey};
see also Theorem~4 of \cite{Ando})
that the range of a contractive idempotent on an $L^p$~space
is isometrically isomorphic to an $L^p$ space.

We prove~(\ref{2deg_faf}).
We know that $f \pi (a) = \pi (a)$
for all $a \in A$,
so we prove that $\pi (a) f = \pi (a)$.
For $\xi \in L^p (X)$ and $\eta \in L^q (X)$,
we have
\[
\langle \pi (a) f  \xi, \eta \rangle
 = \langle f  \xi, \, \pi (a)^* \eta \rangle
 = \lim_t \langle  \pi (e_{t} ) \xi, \, \pi (a)^* \eta \rangle
 = \lim_t \langle \pi (a e_{t} ) \xi, \, \eta \rangle
 = \langle \pi (a)  \xi, \eta \rangle.
\]
Thus $\pi (a) f = \pi (a)$.

Part~(\ref{2deg_Compr}) is now immediate,
as is~(\ref{2deg_Nondeg}) since
$\pi(e_t) \pi(a) \xi \to \pi(a) \xi$ for $a \in A, \xi \in L^p (X)$.

We prove strong operator convergence in~(\ref{2deg_AI}).
It suffices to prove that
$\pi (e_t) \xi \to f \xi$ for $\xi \in f L^p (X)$
and for $\xi \in (1 - f) L^p (X)$.
The first of these follows from (\ref{2deg_Nondeg}).
The second case is trivial:
$\pi (e_t) \xi = 0$ by~(\ref{2deg_faf}),
and $f \xi = 0$.
\end{proof}

\begin{remark}\label{sqprep}
The last result also holds with $L^p$-spaces
replaced by the ${\mathrm{SQ}}_p$ spaces mentioned in the introduction,
although (\ref{2deg_Lp}) would then say that
$E$ is an ${\mathrm{SQ}}_p$ space.
The proof is essentially the same, except that
(\ref{2deg_Lp}) becomes trivial.
We also need to use the fact
that ${\mathrm{SQ}}_p$ spaces are smooth for $p \in (1, \infty)$.
In fact, they are also strictly convex.
To see this,
first observe that reflexivity of $L^p$~spaces
implies reflexivity of ${\mathrm{SQ}}_p$~spaces.
Next,
$L^p$~spaces are both smooth and strictly convex,
so their subspaces are as well.
So the duals of subspaces are both strictly convex and smooth.
By reflexivity,
the quotient of a subspace is the dual of a subspace of the dual,
so both smooth and strictly convex.
\end{remark}

If $A$ is unital as a Banach algebra and
also is an $L^p$-operator algebra
then it follows that $A$ may be viewed as a subalgebra of $B (L^p (X))$
containing the identity operator on $L^p (X)$, for some
measure space $X$.
This was proved first in Section 2 of~\cite{PV}.

\begin{corollary}\label{C_7702_DualRepn}
Let $p \in (1, \In)$.
Let $A$ be a dual unital $L^p$~operator algebra
(Definition $\mathrm{\ref{D_7620_DualLpOpAlg}}$).
Then $A$ has an isometric
unital representation on an $L^p$~space
which is a weak* homeomorphism onto its range.
\end{corollary}

\begin{proof}
Let $\pi \colon A \to B (L^p (X))$ be an isometric representation
which is a weak* homeomorphism onto its range.
As in Lemma~\ref{2deg},
let $E = {\overline{\spn}} (\pi (A) (L^p (X) )$,
and let $f$ be as there.
Clearly $f = \pi (1_A)$.
Define $\sm \colon A \to B (E) = f B (L^p (X)) f$
by $\sm (a) = f \pi (a) f$ for $a \in A$.
Lemma~\ref{2deg} implies that
$\sm$ is an isometric unital representation on an $L^p$~space.
In light of the Krein-Smulian theorem, all we need to show is that
the weak* topology on $B (E)$
is the same as the restriction to $f B (L^p (X)) f$
of the weak* topology on $B (L^p (X))$.
The inclusion of $E$ in $L^p (X)$
as a complemented subspace
gives an inclusion of $\Kdb (E)$ in $\Kdb (L^p (X))$,
and by Theorem \ref{T_7920_KStarStar}~(\ref{T_7920_KStarStar_2nd})
the second dual of this inclusion
is $B (E) \hookrightarrow B (L^p (X))$,
which is therefore a weak* homeomorphism onto its image.
\end{proof}

In particular, applying this principle
to the bidual of an approximately unital $L^p$-operator algebra~$A$,
we obtain
a faithful normal isometric representation of $A^{**}$
that can to some extent play the role of the enveloping von
Neumann algebra
coming from the universal representation of a $C^*$-algebra.

\begin{corollary}\label{ddag}
Let $p \in (1, \infty)$,
and let $A$ be an
approximately unital $L^p$-operator algebra.
Then
there exists a measure space $(X, \mu)$
and a unital isometric representation
$\theta \colon A^{**} \to B (L^p (X, \mu) )$
such that:
\begin{enumerate}
\item\label{ddag_WeakStarHomeo}
$\theta$ is a weak* homeomorphism onto its range.
\item\label{ddag_Nondeg}
$\theta |_A$ acts nondegenerately on $L^p (X, \mu)$.
\item\label{ddag_cai}
For any cai $(e_t)_{t \in \Ld}$ in $A$,
we have
$\theta (e_t) \to 1$ in the strong operator topology
on $B (L^p (X, \mu))$.
\end{enumerate}
\end{corollary}

\begin{proof}
This is clear from Corollary~\ref{C_7702_DualRepn}
and Lemma \ref{2deg}.
\end{proof}

\section{Examples}\label{Sec_3}

As we mentioned in the introduction,
so far the study of $L^p$-operator algebra
has been very largely example driven.
Thus there is a wealth of examples in the literature,
or in preprint form.
(See the works of the second author,
Viola, Gardella and Thiel, and others referred to earlier.)
In this section we discuss the main examples which we have used,
or which seem useful but are not in the literature.
We recall again that, as always,
in this section $p \in (1, \infty) \setminus \{ 2 \}$
unless stated to the contrary.

\begin{notation}\label{N_7719_Mnp}
As in for example~\cite[Lemma 6.11]{PV},
for $n \in \Ndb$ and $p \in [1, \In]$
we write $l^p_n$ for $L^p$ of an $n$~point space with counting measure,
and define $M_n^p = B (l^p_n)$.
\end{notation}

\begin{example}\label{mp}
Let $p \in [1, \In)$.
Let $e_n \in M_n^p$ be the $n \times n$ matrix whose
entries are all~$\frac{1}{n}$.
We will use $e_n$ several times in this paper
and so the calculations that follow
will be important for us.
If $p = 2$ then $e_n$ is a rank one projection.
For the rest of this example, assume $p \neq 2$,
and let $q \in (1, \infty]$
satisfy $\frac{1}{p} + \frac{1}{q} = 1$.

Suppose $n = 2$.
We have
\[
1 - 2 e_2
 = \left[
  \begin{array}{ccl} 0  &  -1  \\ -1 &  0 \end{array} \right],
\]
which is an invertible isometry.
So $\| e_2 \| = \| 1 - e_2 \| = 1$
by Lemma \ref{L_7719_BicId}~(\ref{L_7719_BicId_BicrOnLp}),
and $e_2$ is real positive
by Lemma \ref{L_7719_BicId}~(\ref{L_7719_BicId_RealPos}).
However,
$e_2$ is not hermitian, by Proposition~\ref{P_7702_HermIsMult},
or by Lemma 6.11 of~\cite{PV}.

For the rest of this example, assume $n \geq 3$
(as well as $p \neq 2$).
We claim that $\| e_n \| = 1$
but $\| 1 - e_n \| > 1$,
so that $e_n$ is not bicontractive.
Then Lemma \ref{L_7719_BicId}~(\ref{L_7719_BicId_RealPos})
implies that $e_n$ is not real positive.

To see that $e_n$ is contractive,
set
\[
\et = ( 1, 1, \ldots, 1 ) \in l^p_n
\andeqn
\mu = \tfrac{1}{n} ( 1, 1, \ldots, 1 ) \in l^q_n.
\]
Then one easily checks that for all $\xi \in l^p_n$
we have $e_n \xi = \langle \mu, \xi \rangle \et$,
so $\| e_n \| \leq \| \mu \|_q \| \et \|_p = 1$.

To show that $\| 1 - e_n \| > 1$,
by Lemma \ref{L_7719_BicId}~(\ref{L_7719_BicId_BicrOnLp})
it is enough to prove that $1 - 2 e_n$ is not isometric.
As pointed out to us by Eusebio Gardella,
Lamperti's Theorem \cite{FJ} 
implies that that the only
matrices which are isometries in the $L^p$ operator norm are the
complex permutation matrices,
and clearly $1-2 e_n$ is not
of this form.
However, we can give a direct proof.

Define $g \colon [1, \In) \to [0, \In)$
by
\[
g (p)
 = \big\| (1 - 2 e_n) (1, 0, 0, \ldots, 0 ) \big\|_p^p.
\]
We have
\[
(1 - 2 e_n) (1, 0, 0, \ldots, 0 )
 = \left( 1 - \frac{2}{n}, \, - \frac{2}{n}, \, - \frac{2}{n},
    \, \ldots, \, - \frac{2}{n} \right),
\]
so
\[
g (p)
  = \left( 1 - \frac{2}{n} \right)^p
     + (n - 1) \left( \frac{2}{n} \right)^p
\]
for $p \in [1, \In)$.
One further has
$g (2) = 1$ and
\[
g' (p) = \left( 1 - \frac{2}{n} \right)^p
         \log \left( 1 - \frac{2}{n} \right)
     + (n - 1) \left( \frac{2}{n} \right)^p
          \log \left( \frac{2}{n} \right)
\]
for all $p \in [1, \In)$.
Both the logarithm terms are strictly negative,
so $g' (p) < 0$.
Therefore
$\big\| (1 - 2 e_n) (1, 0, 0, \ldots, 0 ) \big\|_p \neq 1$
for all $p \in [1, \In) \setminus 2$.
Thus $\| 1 - e_n \| > 1$.

One can see easily that $\| 1 - e_n \| < 2$
(this follows for example from a lemma in the sequel paper),
but we will not use this here.

Lemma \ref{L_7719_BicId}~(\ref{L_7719_BicId_RealPos})
implies that $1 - e_n$ is real positive.
It follows also that
the ``support idempotent'' $s (1 - e_n)$ of $1 - e_n$
(see Definition~\ref{D_7Z17_SuppId})
is not contractive,
unlike support idempotents for real positive Hilbert space operators
(see e.g.\ Corollary 3.4 in \cite{BRII}).
In turn this shows that,
unlike the Hilbert space operator case,
the limit
$\lim_{m \to \In} \| x^{1 / m} \|$ need not equal $1$
for real positive elements in an $L^p$ operator algebra  $A$
(or even for elements of ${\mathfrak{F}}_{A}$).
We are using the $m$-th root in
Definition~\ref{N_7919_nthroot} and the discussions after it.
We also see that,
unlike the Hilbert space operator case in Proposition 2.3 of \cite{BRI},
$\frac{1}{2} {\mathfrak{F}}_{A}$ is not closed under $n$-th roots.
Indeed,
\[
\tfrac{1}{2} (1 - e_n) \in \tfrac{1}{2} {\mathfrak{F}}_{A}
 \subseteq \Ball (A)
\]
but
\[
\lim_{m \to \In} \big( \tfrac{1}{2} (1 - e_n) \big)^{1 / m}
 = s (1 - e_n)
 = 1 - e_n
 \notin \Ball (A).
\]
Nonetheless it is true that
${\mathfrak{F}}_{A}$ is closed under $n$-th roots,
by Proposition \ref{L_7Z16_RootProp}~(\ref{L_7Z16_RootProp_Norm}).
\end{example}

Another example of bicontractive idempotents,
related to the case of $M_2^p$ discussed above,
appears in the
group $L^p$ operator algebra of a discrete group containing elements
of order~2.
(See e.g.\  \cite{P2,G2,G3}.)
These elements give projections in the group $C^*$-algebra,
which are actually in the purely algebraic group algebra.
The
corresponding idempotents in the group $L^p$ operator algebra are
bicontractive, and ``look like'' the bicontractive
idempotents in $M_2^p$.
Since we make little use of group $L^p$ operator algebras
in this paper, we omit the details.
As described below,
however, they motivate Example~\ref{E_7721_eAndfAndef}.

Let $E$ be a Banach space of the form $L^p (X, \mu)$
for some measure space $(X, \mu)$ and some $p \in (1, \In)$.
Let $e, f \in B (E)$ be commuting contractive idempotents.
It is very tempting to conjecture that,
as in the Hilbert space operator case,
$e + f - e f$,
which is an idempotent with range
${\operatorname{Ran}} (e) + {\operatorname{Ran}} (f)$,
is also contractive.
This conjecture is false,
as we will see in Example~\ref{E_7721_eAndfAndef} below,
even if $e$ and $f$ are bicontractive.
Thus, the lattice theoretic properties
of (even commuting) bicontractive idempotents
on $L^p$ spaces are deficient.
Indeed we shall see that there is
a disappointing comparison between the structure of the
lattice of idempotents in $B (L^p(X))$
and the beautiful and fundamental behavior
of projections in von Neumann algebras.
Our example also does two other things.
It shows that the product
of two commuting real positive idempotents need not be real positive.
And it shows that on $L^p$,
there are commuting accretive operators
whose geometric mean exists but is not accretive.
This shows that \cite[Lemma 5.8]{BlcWng} fails
with Hilbert spaces replaced by $L^p$ spaces.

The construction of the example
is motivated as follows.
Fix $p \in (1, \In) \SM \{ 2 \}$.
By Lemma \ref{L_7719_BicId}~(\ref{L_7719_BicId_BicrOnLp}),
commuting pairs of bicontractive idempotents
in $B (L^p (X, \mu))$
are in one to one correspondence with pairs of
commuting invertible isometries of order~$2$ in $B (L^p (X, \mu))$,
and therefore with representations
of $( \Zdb / 2 \Zdb)^2$ on $L^p (X, \mu)$ via isometries.
In particular,
the conjecture in the previous paragraph holds
for all $(X, \mu)$ (for our given value of~$p$)
\ifo{} it holds for the pair of bicontractive idempotents
coming from the universal isometric $L^p$~representation
of $( \Zdb / 2 \Zdb)^2$.
Since $( \Zdb / 2 \Zdb)^2$ is amenable,
this will be true \ifo{}
it holds for the left regular representation
of $( \Zdb / 2 \Zdb)^2$ on $l^p ( ( \Zdb / 2 \Zdb)^2 ) \cong l^p_4$.

\begin{example}\label{E_7721_eAndfAndef}
Fix $p \in (1, \In) \SM \{ 2 \}$.
There is a finite dimensional unital $L^p$~operator algebra
(specifically $M_4^p$) which contains the following:
\begin{enumerate}
\item\label{E_7721_eAndfAndef_Bicrt}
Two commuting bicontractive idempotents
whose product is not even contractive.
\item\label{E_7721_eAndfAndef_RPos}
Two commuting real positive idempotents
whose product is not real positive.
\item\label{E_7721_eAndfAndef_Accr}
Two commuting accretive opertors
whose geometric mean exists but is not accretive.
\end{enumerate}

We work throughout in $M_4^p$.
Define
\[
s =  \left[
  \begin{array}{ccccl}
  0     &  1     &  0     &  0        \\
  1     &  0     &  0     &  0        \\
  0     &  0     &  0     &  1        \\
  0     &  0     &  1     &  0
\end{array} \right]
\in M^p_4
\andeqn
t = \left[ \begin{array}{ccccl}
  0     &  0     &  1     &  0        \\
  0     &  0     &  0     &  1        \\
  1     &  0     &  0     &  0        \\
  0     &  1     &  0     &  0
\end{array} \right]
\in M^p_4.
\]
One checks that these are commuting isometries of order~$2$.
Next, define
\[
e = \tfrac{1}{2} (1 + s)
\andeqn
f = \tfrac{1}{2} (1 + t).
\]
These are commuting idempotents,
and they are bicontractive
by Lemma \ref{L_7719_BicId}~(\ref{L_7719_BicId_BicrOnLp}).
Then one checks that $e f$
is the idempotent $e_4$ of Example~\ref{mp},
and that
$e + f - e f$ is an idempotent.
We claim that it is not contractive.
First,
we look at $1 - (e + f - e f)$,
getting
\[
1 - (e + f - e f)
 = \frac{1}{4}  \left[ \begin{array}{ccccl}
  1     & -1     & -1     &  1        \\
 -1     &  1     &  1     & -1        \\
 -1     &  1     &  1     & -1        \\
  1     & -1     & -1     &  1
\end{array} \right]  .
\]
Define $w = \diag (1, \, -1, \, -1, \, 1)$,
which is an invertible isometry in $M_4^p$.
Then one checks that
$w [1 - (e + f - e f)] w^{-1} = e_4$
in the language of Example~\ref{mp}.
In that example
we showed that this idempotent is contractive,
and also showed
that $1 - w [1 - (e + f - e f)] w^{-1}$ is not contractive.
Therefore also
\[
e + f - e f
 = w^{-1} \big( 1 - w [1 - (e + f - e f)] w^{-1} \big) w
\]
is not contractive.
This is~(\ref{E_7721_eAndfAndef_Bicrt}).

Now define $e_0 = 1 - e$ and $f_0 = 1 - f$.
We have seen that $e$ and $f$ are contractive,
so $e_0$ and $f_0$ are real positive
by Lemma \ref{L_7719_BicId}~(\ref{L_7719_BicId_RealPos}).
However, $1 - e_0 f_0 = e + f - e f$
is not contractive,
so $e_0 f_0$ is not real positive,
again by Lemma \ref{L_7719_BicId}~(\ref{L_7719_BicId_RealPos}).
This is~(\ref{E_7721_eAndfAndef_RPos}).

We turn to~(\ref{E_7721_eAndfAndef_Accr}).
We want invertible elements.
Neither $e$ nor $f$ is invertible, but this is
easily fixed by adding $\varepsilon 1$ to each of them, which does
not change the fact that they commute.
We recall the well known
Ando et al list of properties that a ``good'' geometric
mean should possess (see e.g.\  p.\ 306 of~\cite{ALM}).
One of these is that the geometric mean of $a$ and $b$
should be
$a^{1 / 2} \, b^{1 / 2}$
(as in Definition~\ref{N_7919_nthroot})
whenever $a$ and $b$ commute.
One also needs to assume in our
case that these principal square roots exist.

Suppose that
$(\varepsilon 1 + e)^{1 / 2} (\varepsilon 1 + f)^{1 / 2}$
is accretive for all $\varepsilon > 0$.
Using the Macaev-Palant formula
$\big\| a^{1 / 2} - b^{1 / 2} \big\| \leq K \| a - b \|^{1 / 2}$
(see Lemma~2.4 of~\cite{BlcWng},
the preceding discussion, and the reference given there),
letting $\varepsilon \to 0$ implies that
$e^{1/2} f^{1/2}$ is accretive.
We have $e^{1/2} = e$ and $f^{1/2} = f$
by e.g.\  
Proposition \ref{L_7Z16_RootProp}~(\ref{L_7Z16_RootProp_Series}).
So $e f$ is accretive, a contradiction.
\end{example}

\begin{example}\label{uta}
Let $p \in [1, \In)$.
Given a closed linear subspace $E \subseteq B (L^p (X))$,
define
${\mathcal{U}} (E) \subseteq B \big( L^p (X) \oplus^p L^p (X) \big)$
to be the set of operators which have the $2 \times 2$ matrix form
\begin{equation}\label{Eq_7721_UE}
\left[
   \begin{array}{ccl} \lambda  &  x  \\ 0 &  \mu \end{array} \right]
\end{equation}
with $\lambda, \mu \in \Cdb$ and $x \in E$.
Then ${\mathcal{U}} (E)$ is a unital $L^p$-operator algebra.
Moreover,
if $F \subseteq L^p (Y)$ and $u \colon E \to F$ is linear,
then the map
${\mathcal{U}} (u) \colon {\mathcal{U}} (E) \to {\mathcal{U}} (F)$,
defined by
\[
{\mathcal{U}} (u) \left( \left[
    \begin{array}{ccl} \lambda  &  x  \\ 0 &  \mu \end{array} \right]
            \right)
   =
    \left[ \begin{array}{ccl} \lambda  &  u (x)  \\ 0 &  \mu \end{array}
      \right]
\]
for $\lambda, \mu \in \Cdb$ and $x \in E$,
is a unital homomorphism.
We will show that if $u$ is contractive or isometric,
then so is ${\mathcal{U}} (u)$.

To begin,
we claim that if $\lambda, \mu \in \Cdb$ and $x \in B (L^p (X))$,
then
\begin{equation}\label{Eq_7721_NormComp}
\left\| \left[
   \begin{array}{ccl} \lambda  &  x  \\ 0 &  \mu \end{array}
          \right] \right\|
  = \left\| \left[
  \begin{array}{ccl} |\lambda|  &  \| x \| \\ 0 &  | \mu | \end{array}
      \right] \right\|,
\end{equation}
with the norm on the right hand side being taken in~$M_2^p$.
Hence the norm on ${\mathcal{U}} (E)$ only depends on the norms
of elements in~$E$,
not the elements themselves.

We prove the claim.
Let $\lambda, \mu \in \Cdb$ and let $x \in B (L^p (X))$.
Define
\[
a = \left[
   \begin{array}{ccl} \lambda  &  x  \\ 0 &  \mu \end{array} \right]
  \in B \big( L^p (X) \oplus^p L^p (X) \big)
\andeqn
c = \left[
  \begin{array}{ccl} |\lambda|  &  \| x \| \\ 0 &  | \mu | \end{array}
      \right]
  \in M_2^p.
\]
We have
\begin{equation}\label{Eq_7802_NormFormula}
\| a \|
  = \sup \Big( \Big\{ \big(
  \| \lambda \eta + x \xi \|^p_p + \| \mu \xi \|_p^p \big)^{1 / p}
       \colon {\mbox{$\eta, \xi \in L^p (X)$ satisfy
       $\| \eta \|_p^p + \| \xi \|_p^p \leq 1$}}  \Big\} \Big).
\end{equation}
The quantity inside the supremum is dominated by
\[
\big[ \big( |\lambda| \| \eta \|_p  + \| x \| \| \xi \|_p \big)^p
      + \big( |\mu| \|  \xi \|_p \big)^p \big]^{1 / p}
  = \big\| c ( \| \et \|_p, \, \| \xi \|_p ) \big\|_p
  \leq \| c \|.
\]
So $\| a \| \leq \| c \|$.
To see the other direction we may assume that $x \neq 0$.
Choose scalars $\alpha, \beta$ with
$| \alpha |^p + | \beta |^p \leq 1$
such that the norm of $c$ is achieved at $(\alpha, \beta)$.
Multiplying $\alpha$ and $\beta$
by a complex number of absolute value~$1$,
we may assume that $\beta \geq 0$.
Since
$c \big( \af, \bt) = ( \af | \ld | + \bt \| x \|, \, \bt | \mu | \big)$,
we see that
$\| c \big( \af, \bt) \|_p \leq \| c \big( |\af|, \bt) \|_p$,
so we may also assume that $\af \geq 0$.
If $\bt = 0$ then
\[
\| c \|
 = \| c ( \af, \bt) \|_p
 = | \af \ld |
 \leq | \ld |
 \leq \| a \|.
\]
Otherwise,
let $\ep > 0$.
Choose $\dt > 0$ such that
\[
\dt < \bt \| x \|
\andeqn
\big( \big| | \lambda | \alpha + \bt \| x \| \big| - \dt \big)^p
  > \big| | \lambda | \alpha + \bt \| x \| \big|^p - \ep.
\]
Choose $\xi \in L^p (X)$ of
norm $\beta$ so that $\big| \| x \xi \|_p - \bt \| x \| \big| < \dt$.
Then $x \xi \neq 0$.
Choose $\zt \in \Cdb$ such that $| \zt | = 1$ and $\zt \ld = | \ld |$.
Define $\eta = \zt \alpha \| x \xi \|_p^{-1} x \xi \in L^p (X)$.
Then $\eta$ has
norm $\alpha$,
so that $\| \eta \|_p^p + \| \xi \|_p^p \leq 1$.
Now
\begin{align*}
\| a (\eta, \xi) \|^p
 & = \| \lambda \eta + x \xi \|^p_p + \| \mu  \xi \|_p^p
   = \left( \left| \frac{\lambda \zt \alpha}{\| x \xi \|_p} + 1 \right|
       \| x \xi \|_p \right)^p
     + |\mu \beta|^p
\\
  & = \big| | \lambda | \alpha + \| x \xi \|_p \big|^p + |\mu \beta|^p
    > \big( \big| | \lambda | \alpha + \bt \| x \| \big| - \dt \big)^p
      + |\mu \beta|^p
\\
  & > \big| | \lambda | \alpha + \bt \| x \| \big|^p - \ep
          + |\mu \beta|^p
    = \| c (\af, \bt) \|_p^p - \ep
    = \| c \|^p - \ep.
\end{align*}
Since $\ep > 0$ is arbitrary,
the claim follows.

It follows that if $u \colon E \to F$ as above
is isometric, then so is ${\mathcal{U}} (u)$.

We claim that
if $u \colon E \to F$ is a linear contraction,
then ${\mathcal{U}} (u)$ is also contractive.
By the previous claim,
it suffices to prove that if $\ld, \mu, \rh, \sm \in [0, \infty)$
and $\rh \leq \sm$,
then
\begin{equation}\label{Eq_7802_MatNorm}
\left\| \left[
   \begin{array}{ccl} \lambda  &  \rh \\ 0 &  \mu \end{array}
    \right] \right\|
  \leq \left\|  \left[
   \begin{array}{ccl} \lambda  &  \sm \\ 0 &  \mu \end{array}
     \right] \right\|.
\end{equation}
We apply~(\ref{Eq_7802_NormFormula})
to these matrices.
For $\af, \bt \in \Cdb$
we have $\| (| \af |, | \bt | ) \|_p = \| (\af, \bt ) \|_p$.
Since $\ld, \rh \geq 0$,
the expression $| \ld \af + \rh \bt |^p + | \mu \bt |^p$
becomes no smaller if $\af$ and $\bt$
are replaced by $| \af |$ and $| \bt |$,
and similarly with $\sm$ in place of~$\rh$.
Therefore the norms of the matrices in~(\ref{Eq_7802_MatNorm})
are $N (\rh)$ and $N (\sm)$, with $N$ given by
\[
N (\ta) = \sup \Big( \Big\{ \big(
( \lambda \af + \ta \bt )^p + (\mu \bt )^p \big)^{1 / p}
       \colon {\mbox{$\af, \bt \in [0, \infty)$ satisfy
       $\af^p + \bt^p \leq 1$}}  \Big\} \Big)
\]
for $\ta \in [0, \infty)$.
Since all the variables are nonnegative,
clearly $\rh \leq \sm$ imples $N (\rh) \leq N (\sm)$.
This yields~(\ref{Eq_7802_MatNorm}).
The claim is proved.
\end{example}

Example~\ref{uta} is often useful for counterexamples
because it can convert a
bad linear subspace of $B (L^p (X))$
into a suitably badly behaved  $L^p$-operator algebra.
Note that if $E$ is weak* closed in $B (L^p (X))$
then ${\mathcal{U}} (E)$ is a dual $L^p$-operator algebra
in the sense of Definition \ref{D_7620_DualLpOpAlg}.
This follows just as in Lemma 2.7.7 (1) in \cite{BLM},
but using the characterization
of weak* convergent nets in $B (L^p (X))$ given after
Corollary \ref{C_7620_KLp}.

\begin{example}\label{cone}
Let $p \in [1, \In)$.
The set of continuous functions $f \colon [0, 1] \to M_2^p$
is a unital $L^p$-operator algebra.
We may view
this as the canonical copy of $C( [0, 1]) \otimes M_2^p$
inside the bounded operators
on
\[
L^p ( [0, 1]) \otimes l^p_2
 \cong l^p_2 (L^p ( [0, 1]))
 \cong L^p ( [0, 1]) \oplus^p L^p ([0, 1]).
\]
The subalgebra consisting of functions with
$f (1)$ diagonal is also a unital $L^p$-operator algebra.
The subalgebras consisting of functions $f$ with
$f (0) = 0$, or with $f (0) = 0$ and $f (1)$ diagonal,
are approximately unital $L^p$-operator algebras.
Indeed,
if $(e_n)_{n \in \Ndb}$ is a cai for $C_0 ( (0, 1])$,
then, using tensor notation,
$(e_n \otimes 1_2)_{n \in \Ndb}$ is a cai for these algebras.
\end{example}

\begin{example}\label{Sepran}
Let $p \in (1, \In)$.
Let $(X, \mu)$
be a measure space,
and, to avoid trivialities,
assume that $L^p (X, \mu)$ is not separable.
Let $A$ in $B (L^p (X, \mu))$ be the ideal of operators on $L^p (X)$
with separable
range, which is known to be a closed ideal.
We claim that $A$ is an $L^p$-operator algebra with a cai,
and, if $X$ is
a discrete space with counting measure,
even a cai consisting of hermitian and real positive idempotents.

We prove the first part of the claim.
If $E \subseteq L^p (X)$ is any separable subspace,
it follows by Theorem 6 in Section 16 on p.~146 of~\cite{Lacey}
and Lemma 2 in Section 17 on p.~153 of \cite{Lacey} (see also
Proposition 1.25 in \cite{P2}),
that $E$ is contained in the range of a contractive idempotent
with separable range.
(Spaces are assumed to be real in~\cite[Section 16]{Lacey},
however the complex case is no doubt well known to Banach space experts.
Indeed by the just
cited results or their proofs a separable subspace of $L^p(X)$ is
contained in a separable closed sublattice $F$.
Since the norm on $F$ is $p$-additive, $F$ is an abstract $L_p$ space
(see p.\  131 of~\cite{Lacey} for definitions of these terms),
so by Theorem 3 in both
Sections 15 and 17 of \cite{Lacey}, $F$ is contractively complemented.)

Also, it is well known and an exercise
that an operator $x$ on a reflexive space
has separable range if and only if $x^*$ has separable range.
Taking $q \in (1, \In)$
to satisfy $\frac{1}{p} + \frac{1}{q} = 1$,
we see that
$A^*$ is the collection of operators on $L^q (X, \mu)$
with separable range.
For any $x_1, x_2, \ldots, x_n  \in A$, the closure of the
linear span of their ranges is separable by standard arguments.
It follows that there exist contractive idempotents $e$ and $f$
with separable ranges such that
$x_k = e x_k = x_k f$ for $k = 1, 2, \ldots, n$. 
Thus $A$ has a cai $(e_t)_{t \in \Ld}$,
indeed a cai consisting of contractive idempotents
and such that for any finite set $F \subseteq A$
there is $t \in \Ld$
such that $e_t x = x e_t = x$ for all $x \in F$.
Indeed take $\Ld$ to be the collection
of such finite subsets of $A$.

Now take $X$ to be a set~$I$ with counting measure,
so $L^p (X) = l^p (I)$.
For any $J \subseteq I$ let $e_J$ be the
canonical hermitian (diagonal) projection $e_J$
onto the image of $l^p (J)$ in
$l^p (I)$.
Suppose $x_1, x_2, \ldots, x_n \in B (l^p (I))$
have separable ranges.
Then, as above, the closure $E$ of the
joint span of their ranges is separable.
So there exists a countable subset $J$ of $I$
(the union of the supports of elements
in a countable dense set in $E$) such that
all elements of $E$ are
supported on $J$.
As in the last paragraph, the net $(e_J)$, indexed by the
countable subsets $J$ of $I$ ordered by inclusion,
is a real positive hermitian cai
consisting of bicontractive idempotents
(since $1 - e_J = e_{I \setminus J}$ is contractive).
\end{example}

\begin{example}\label{E_7802_L1G}
Let $G$ be a locally compact group which is not discrete,
with Haar measure~$\mu$.
Then $L^1 (G)$ is approximately unital,
and by Wendel's theorem
its multiplier algebra is $M (G)$,
the measure algebra on $G$.
In particular,
$M (G)$ in an $L^1$~operator algebra.
The identity of $M (G)$ is $\delta_1$,
the Dirac measure at $1_G$.
Hence the multiplier unitization of
$L^1 (G)$ is
$L^1 (G) + \Cdb \delta_1 \subseteq M (G) \subseteq B (L^1 (G))$.
If $f \in L^1 (G)$ and $\lambda \in \Cdb$ then
\[
\| f + \lambda \delta_1 \|
=  \sup \left(
    \left\{ \left| \int_G \, f g \, d \mu + \lambda g (1) \right|
           \colon g \in \Ball (C_0 (G)) \right\} \right).
\]

We claim that the multiplier unitization of
$L^1 (G)$ is $L^1 (G) \oplus^1 \Cdb$.
Fix $f \in L^1 (G)$ and $\lambda \in \Cdb$;
it is enough to prove that
$\| f + \lambda \dt_1 \|_{M(G)} \geq \| f \|_1 + | \ld |$.
Given $\varepsilon > 0$, choose $h \in \Ball  (C_0 (G))$
with $\left| \int_G \, f h \, d \mu \right| > \| f \|_1 - \varepsilon$.
Replacing $h$ by $e^{i \bt} h$ for suitable $\bt \in \Rdb$,
we may assume that $\int_G \, f h \, d \mu \geq 0$.
We have $\mu (\{ 1 \}) = 0$ since $G$ is not discrete.
Choose by regularity a neighborhood
$U$ of $1$ such that
$\int_U \, |f | \, d \mu < \varepsilon$.
By Urysohn's lemma
there is a continuous function $k_1 \colon G \to [0, 1]$
with compact support $K$ contained in $U$
and taking the value $1$ at~$1_G$.
There is also a continuous function $k_2 \colon G \to [0,1]$
which is $1$ on $G \SM U$ and is zero on $K$.
Choose $\theta \in \Rdb$ such that $e^{i \theta} \lambda = |\lambda|$,
and let $g = h k_2 + e^{i \theta} k_1$.
Thus we have $g \in \Ball  (C_0 (G))$
with $\lambda g (1) = |\lambda|$, and $g = h$ on $G \SM U$.
Thus
\[
\left| \int_G \, f g \, d \mu - \int_G \, f h \, d \mu \right|
 \leq 2 \int_U \, |f| \, d \mu < 2 \varepsilon.
\]
Using $\int_G \, f h \, d \mu \geq 0$
and $\lambda g (1) = |\lambda| \geq 0$,
we have
\begin{align*}
\| f + \lambda \dt_1 \|
& \geq \left| \int_G \, f g \, d \mu + \lambda g (1) \right|
  > \left| \int_G \, f h \, d \mu + \lambda g (1) \right|
         - 2 \varepsilon
\\
& = \int_G \, f h \, d \mu + |\lambda| - 2 \varepsilon
  > \| f \|_1 + |\lambda|  - 3 \varepsilon.
\end{align*}
Since $\varepsilon > 0$ is arbitrary,
the claim is proved.

It follows (see Definition~\ref{D_7702_FA}
and Proposition~\ref{P_7702_rAAndFA})
that
${\mathfrak{F}}_{L^1 (G)} = {\mathfrak{r}}_{L^1 (G)} = \{ 0 \}$.
By Lemma~\ref{L_7Z14_HerDiffRPos},
$L^1 (G)$ also has no nonzero hermitian elements.
In particular, $L^1 (G)$ has no
hermitian or real positive cai.
\end{example}

\begin{example}\label{disk0}
A good example of an $L^p$-operator algebra with a real positive
cai but no hermitian cai is the set $A$
of functions in the disk algebra vanishing at~$1$, represented
on $L^p$ of the circle as multiplication operators.
The disk algebra contains no nontrivial
hermitian elements, since the latter would be real valued functions.
However, $A$ is approximately unital.
One way to see this is to combine
Example I.1.4 (b) of~\cite{MIBS}
(after Lemma I.1.5 there)
with Theorem 4.8.5 (1) of~\cite{BLM},
realizing the disk algebra as an operator
algebra by representing it
on $L^2$ of the circle
(instead of $L^p$) as multiplication operators.
\end{example}

\begin{example}\label{kl01}
Let $p \in [1, \In) \SM \{ 2 \}$.
We consider the algebras $\Kdb (L^p (X, \mu))$
for $X = \Ndb$ with counting measure and
$X = [0, 1]$ with Lebesgue measure.
The first has a cai consisting of real positive,
in fact, hermitian, idempotents.
The second has a cai,
but contains no nonzero real positive elements,
and in particular no nonzero hermitian elements.

A hermitian element in $B (L^p (X, \mu))$ is
``multiplication by an essentially bounded real valued
locally measurable function'' (Proposition \ref{P_7702_HermIsMult}).
Thus the hermitian elements in $B (l^p)$
are the infinite diagonal matrices
with bounded real entries.
Therefore the canonical approximate identity in
$\Kdb (l^p)$ is a cai consisting
of real positive and hermitian elements.
(Also see the discussion in \cite[Section 6]{PV}.)

Abbreviate $A = \Kdb (L^p ([0, 1]))$.
This algebra is approximately unital
by e.g.\ Theorem~2 of~\cite{Palmer}.
We can in fact give a formula for cai $(e_n)_{n = 0, 1, \ldots}$
consisting of contractive finite rank idempotents
which is increasing in the order $\leq$
in Definition~\ref{Dorder}.
For $n = 0, 1, \ldots$,
for
\[
\xi \in L^p ([0, 1]),
\qquad
k = 1, 2, \ldots, 2^n,
\andeqn
x \in \left[ \frac{k - 1}{2^n}, \,  \frac{k}{2^n} \right),
\]
define
\[
(e_n \xi) (x)
 = 2^n \int_{(k - 1) / 2^n}^{k/ 2^n} \xi (t) \, d t.
\]
One easily checks that
$(e_n)_{n = 1, 2, \ldots}$
has the properties claimed for it.

Assume now that $p \in (1, \In) \SM \{ 2 \}$.
It is known (see Theorem~2 of~\cite{BenL})
there is no nonzero $a \in A$ with
$\| 1 - a \| \leq 1$.
It follows from Proposition \ref{P_7702_rAAndFA}
that ${\mathfrak{r}}_A = \{ 0 \}$.
That is, for $p \in (1, \In) \SM \{ 2 \}$,
there are no nonzero real positive elements
in $A$ in the main sense of \cite{BOZ}.
Hence by Lemma \ref{bicr}~(\ref{bicr_3})
and Lemma \ref{L_7618_SecDual}~(\ref{L_7618_SecDual_AReg}),
for every cai~${\mathfrak{e}}$, we have
${\mathfrak{r}}^{\mathfrak{e}}_A = \{ 0 \}$.
(This set was defined before Lemma 2.6 in \cite{BOZ}.
In our present case,
by Lemma \ref{bicr}~(\ref{bicr_3}) and Definition~\ref{D_7702_RealPos},
${\mathfrak{r}}^{\mathfrak{e}}_A$ is the set of elements $x \in A$
with $\Re (\varphi (x)) \geq 0$ for all $\varphi \in S (A)$.)
In particular, for $p \in (1, \In) \SM \{ 2 \}$,
$A$ has no real positive cai.
So, by Lemma~\ref{L_7Z14_HerDiffRPos}
and Proposition~\ref{muau},
$A$ has no hermitian cai.
\end{example}

It is easy to see directly that
$\Kdb (L^p ([0, 1]))$ has no nonzero hermitian elements.
Indeed, Proposition \ref{P_7702_HermIsMult} implies that a
hermitian element in $B (L^p ([0, 1]))$
is the multiplication operator $M_f$
by a bounded measurable real valued function $f$.
If the range of such an operator $M_f$
is nonzero then it contains $L^p (E)$
for some non-null $E \subseteq [0, 1]$. 
Indeed there is $\varepsilon > 0$ such that
$E = \big\{ x \in [0, 1] \colon | f (x) | > \varepsilon \big\}$
has strictly positive measure.
So  $L^p (E)$ is contained in the
range of multiplication by $f$.
Since the measure has no atoms, $L^p (E)$ is infinite dimensional.
This cannot be if $M_f$ is compact,
since in that case its restriction to $L^p(E)$ is
compact and bounded below.

\begin{proposition}\label{knsca}
Let $p \in (1, \In) \SM \{ 2 \}$.
Set $A = \Kdb (L^p ([0, 1]))$.
If $e$ is the identity of $A^{**}$, viewed
as an element of $(A^1)^{**}$, then $\| 1 - e \| > 1$.
\end{proposition}

\begin{proof}
Suppose that $\| 1 - e \| \leq 1$.
Then by Goldstine's Theorem there are nets
$(a_t)_{t \in \Ld}$ in~$A$ and $(\lambda_t)_{t \in \Ld}$ in~$\Cdb$
such that $\| \lambda_t 1 + a_t \| \leq 1$ for all $t \in \Ld$
and
$\lambda_t 1 + a_t \to 1 - e$ weak*.
Applying the character annihilating $A$ we see that $\lambda_t \to 1$.
Hence $a_t \to - e$ weak*.
Theorem~2 of~\cite{BenL} provides $\dt > 0$
such that whenever $b \in A$ satisfies $\| b \| \geq \frac{1}{2}$
then $\| 1 - b \| > 1 + \dt$.
Choose $t_0 \in \Ld$ such that $| 1 - \ld_t | < \frac{\dt}{2}$
for $t \in \Ld$ with $t \geq t_0$.
There is $t_1 \in \Ld$ such that $t_1 \geq t_0$ and
$\| a_{t_1} \| > \| - e \| - \frac{1}{2}$
(for otherwise $\| a_t \| \leq  \| - e \| - \frac{1}{2}$
for $t  \geq t_0$,
giving the contradiction $\| -e \| \leq  \| - e \| - \frac{1}{2}$).
Clearly $\| - e \| \geq 1$.
So $\| a_{t_1} \| > \frac{1}{2}$,
whence $\| 1 + a_{t_1} \| > 1 + \dt$.
But
\[
\| 1 + a_{t_1} \|
 \leq | 1 - \ld_{t_1} | + \| \ld_{t_1} + a_{t_1} \|
 < \frac{\dt}{2} + 1.
\]
This contradiction shows that $\| 1 - e \| \leq 1$ is impossible.
\end{proof}

\section{Miscellaneous results on $L^p$-operator algebras}\label{Sec_4}

\subsection{Quotients and bi-approximately unital algebras}\label{quot}

\begin{definition}\label{biau_Lp}
Let $A$ be an $L^p$-operator algebra.
and let $J \subseteq A$ be a closed ideal.
We say that $J$ is a {\emph{bi-approximately unital ideal}} in $A$
(or is {\emph{bi-approximately unital in~$A$}})
if $J$ is approximately unital
and if there is an $L^p$~operator unitization $B$ of~$A$
(as in Definition~\ref{D_7720_Unitization})
such that identity $e$ of the bidual $J^{**}$
is a bicontractive idempotent in $B^{**}$.
\end{definition}

\begin{definition}\label{biau_AReg}
Let $A$ be an approximately unital Arens regular Banach algebra.
We say that $A$ is
{\emph{bi-approximately unital}} if
in the bidual $(A^1)^{**}$
of its multiplier unitization~$A^1$
the identity $e$ of $A^{**}$
is a bicontractive idempotent.
\end{definition}

The next lemma shows that the terminology is consistent.

\begin{lemma}\label{L_7718_Bicontr}
Let $A$ be an approximately unital $L^p$~operator algebra.
Then $A$ is bi-approximately unital
in the sense of Definition~ $\mathrm{\ref{biau_AReg}}$
if and only if
$A$ is bi-approximately unital as an ideal in itself
in the sense of Definition~ $\mathrm{\ref{biau_Lp}}$.
\end{lemma}

Recall from Lemma~\ref{L_7618_SecDual}~(\ref{L_7618_SecDual_AReg})
that $L^p$~operator algebras
are automatically Arens regular.

\begin{proof}[Proof of Lemma~{\rm{\ref{L_7718_Bicontr}}}]
If $A$ is bi-approximately unital
in the sense of Definition~\ref{biau_AReg},
we can take the $L^p$~operator unitization
required in Definition~\ref{biau_Lp}
to be $A^1$,
recalling from Proposition~\ref{muau} that $A^1$
is an $L^p$~operator algebra.
If $A$ is bi-approximately unital as an ideal in itself,
let $B$ be an $L^p$~operator unitization
as required in Definition~\ref{biau_Lp},
and let $e$ be as there.
The obvious homomorphism
$\ph \colon B \to A^1$
is contractive,
by Remark~\ref{R_7702_UnitizationFacts}
(\ref{R_7702_UnitizationFacts_ToMult}),
so $\ph^{**} \colon B^{**} \to (A^1)^{**}$
is contractive.
Thus $\| \ph^{**} (e) \| \leq \| e \| \leq 1$
and $\| 1 - \ph^{**} (e) \| \leq \| 1 - e \| \leq 1$.
Since $\ph^{**} (e)$ is the
identity of $A^{**}$ as in Definition~\ref{biau_AReg},
we have shown that $A$ is bi-approximately unital.
\end{proof}

The algebra $\Kdb (L^p ([0, 1]))$
is an approximately unital $L^p$~operator algebra
which is not bi-approximately unital.
See Example~\ref{kl01} and Proposition \ref{knsca}.

By Lemma \ref{L_7719_BicId}~(\ref{L_7719_BicId_BicrOnLp}),
$A$ is bi-approximately unital
if and only if $A$ is a $u$-ideal in $A^1$
as defined at the beginning of Section~3 of \cite{GKS},
that is, that $\| 1 - 2 e \| \leq 1$
where $e$ is the identity of $A^{\perp \perp}$ in $(A^1)^{**}$.

\begin{lemma}\label{havec2}
Let $A$ be an approximately unital Arens regular Banach algebra.
If $A$ has a real positive bounded approximate identity,
then $A$ is bi-approximately unital in the sense of Definition
$\mathrm{\ref{biau_AReg}}$.
\end{lemma}

\begin{proof}
Lemma \ref{havec} implies that
$A$ has a cai in ${\mathfrak{F}}_A$.
This cai converges weak* to the identity $e$ of $A^{**}$
by Lemma \ref{L_7917_AppIdConverge}.
Since norm closed balls are weak* closed,
we get $\| e \| \leq 1$ and $\| 1 - e \| \leq 1$.
Hence $e$ is bicontractive.
\end{proof}

We conjecture that the converse
of Lemma \ref{havec2}
is always true for $L^p$-operator algebras,
namely that a bi-approximately unital $L^p$-operator algebra $A$
has a real positive cai.
Corollary \ref{bigmove} may be helpful
for this question.

In \cite{GT} it is shown that
quotients of $L^p$-operator algebras by closed ideals need not
be $L^p$-operator algebras,
giving a negative solution to Problem 3.8 in \cite{LeM}.
Things are better if the
ideal is approximately unital.

\begin{lemma}\label{quo}
Let  $p \in (1, \infty)$,
let $A$ be an $L^p$~operator algebra,
and let $J \subseteq A$ be a closed ideal.
\begin{enumerate}
\item\label{quo_biappu} 
If $J$ is a bi-approximately unital ideal in $A$
then $A / J$ is an $L^p$~operator algebra.
\item\label{quo_appu}
If $J$ is approximately unital
then there is a continuous bijective homomorphism
from $A / J$ to an $L^p$~operator algebra
whose inverse is also continuous.
\end{enumerate}
\end{lemma}

\begin{proof}
We may suppose that
$A$ is unital with identity~$1$.
Recall from Lemma \ref{L_7618_SecDual}~(\ref{L_7618_SecDual_WkSt})
that multiplication on $A^{**}$ is separately weak* continuous.
Also,
the weak* closure of $J$ in $A^{**}$ is $J^{\perp \perp}$.

Let $(e_t)_{t \in \Ld}$ be a cai for~$J$.
Since $J$ is Arens regular
(Lemma \ref{L_7618_SecDual}~(\ref{L_7618_SecDual_WkSt})),
Lemma~\ref{L_7917_AppIdConverge}
shows that there is $e \in J^{**}$ which is an identity for $J^{**}$
and such that $(e_t)_{t \in \Ld}$ converges weak* to~$e$.
Clearly $\| e \| \leq 1$.

We claim that $e A^{**} = J^{\perp \perp}$
and $A^{**} e = J^{\perp \perp}$.
The proofs are the same,
so we do only the first.
We have $J^{\perp \perp} \subseteq e A^{**}$
since $e$ is an identity for $J^{\perp \perp}$.
Also, if $a \in A$ then $e_t a \in J$ for all $t \in \Ld$,
and $e_t a \to e a$ weak*, so $e a$ is in the weak* closure
of $J$ in $A^{**}$, which is $J^{\perp \perp}$.
Thus $e A \subseteq J^{\perp \perp}$.
Since multiplication on $A^{**}$ is separately weak* continuous,
it follows that $e A^{**} \subseteq J^{\perp \perp}$.
The claim is proved.

For $a \in A^{**}$, since $a e, e a \in J^{\perp \perp}$
and $e$ is an identity for $J^{\perp \perp}$,
we get $(e a) e = e a$ and $e (a e) = a e$.
So $e$ is central in $A^{**}$.

Since $e$ is an idempotent,
we have an algebra homeomorphism
(not necessarily isometric)
$A^{**} / e A^{**} \cong (1 - e) A^{**}$.
If $J$ is bi-approximately unital,
then $\| 1 - e \| = 1$,
and this isomorphism is isometric.
Therefore we have algebras homomorphisms
\[
A / J \hookrightarrow A^{**} / J^{\perp \perp}
      = A^{**} / e A^{**}
      \rightarrow (1 - e) A^{**}
      \hookrightarrow A^{**}.
\]
All maps are isometric except possibly the third,
which is a homeomorphism in general
and is isometric if $J$ is bi-approximately unital.
Since $A^{**}$ is an $L^p$~operator algebra
by Lemma \ref{L_7618_SecDual}~(\ref{L_7618_SecDual_WkSt}),
we are done.
\end{proof}

\begin{remark}\label{qreadetc}
\begin{enumerate}
\item\label{qreadetc_Read}
Using an ultrapower argument,
Charles Read showed in an unfinished personal communication
that the quotient  $B (l^p)/\Kdb (l^p)$
is isometrically an $L^p$-operator algebra.
This fact is also contained in Theorem~2.1 and Remark~2
of~\cite{BdjJhn},
combined with the fact
(Theorem 3.3 (ii) of~\cite{He})
that ultrapowers of $L^p$~spaces are $L^p$~spaces.
This result
also follows from Lemma \ref{quo}~(\ref{quo_biappu}),
since the canonical cai for $\Kdb (l^p)$ is bicontractive
and hence so is its weak* limit.

Read was also working on whether
$B (L^p) / \Kdb (L^p)$ is an $L^p$-operator algebra.
The results of~\cite{BdjJhn} quoted above
show that it is at least isomorphic to one,
a fact which also follows from Lemma \ref{quo}~(\ref{quo_appu}).
We are studying Read's unfinished proof of the latter
in hopes of answering this question.

\item\label{qreadetc_Toeplitz}
We remind the reader of an example from \cite{GT}:
the $p$ variant of the Toeplitz algebra quotiented
by $\Kdb (l^p)$ is isomorphic to $F^p (\Zdb)$,
the norm closed subalgebra of $B (l^l (\Zdb))$ generated by
the bilateral shift and its inverse.
(This is the full group $L^p$-operator algebra
of the two element group as
defined in \cite{P2};
see Definition~3.3 and the discussion before Proposition 3.14 there.)
In particular, it is not $C (\Tdb)$.
\end{enumerate}
\end{remark}

\begin{example}\label{mp2}
We exhibit $p \in (1, \infty) \setminus \{ 2 \}$
and an $L^p$~operator algebra $A$
with a closed approximately unital ideal $J$
such that $A / J$ is not isometrically isomorphic to
an $L^p$~operator algebra.
This shows that Lemma \ref{quo}~(\ref{quo_appu})
can't be improved.
In our example,
$A$ is commutative and three dimensional,
and $J$ has an identity $e$ which is central in $A$
and with $\| e \| = 1$
(but $\| 1 - e \| > 1$).

Fix $n \in \{ 2, 3, \ldots \}$.
(We will later take $n = 3$.)
Let $e_n$ be as in Example~\ref{mp}.
Define $\zt = e^{ 2 \pi i / n }$
and $s = \diag \big( 1, \zt, \zt^{2}, \ldots, \zt^{n - 1} \big)$.
For $k = 0, 1, \ldots, n - 1$ set
$f_k = s^k e_n s^{-k}$.
We claim that:
\begin{enumerate}
\item\label{I_7z11_mp2_fk}
$f_k$ is contractive idempotent for $k = 0, 1, \ldots, n - 1$.
\item\label{I_7z11_mp2_orth}
$f_0, f_1, \ldots, f_{n - 1}$ are orthogonal,
that is, $f_j f_k = 0$ if $j \neq k$.
\item\label{I_7z11_mp2_Sum1}
$\sum_{k = 0}^{n - 1} f_k = 1_{M_n}$, the $n \times n$ identity matrix.
\end{enumerate}
For~(\ref{I_7z11_mp2_fk}),
recall from Example~\ref{mp} that $\| e_n \| = 1$,
and use $\| f_k \| \leq \| s \|^k \| e_n \| \| s^{-1} \|^k$.
For (\ref{I_7z11_mp2_orth}) and~(\ref{I_7z11_mp2_Sum1}),
let $u \in M_n$ be the matrix whose $k$-th column
(starting the count at $0$ instead of~$1$) is
\[
\frac{1}{\sqrt{n}} s^k \big( 1, 1, \ldots, 1 \big)
 = \frac{1}{\sqrt{n}}
    \big( 1, \zt^k, \zt^{2k}, \ldots, \zt^{(n - 1)k} \big).
\]
Computations show that $u$ is unitary (in the $p = 2$ sense),
and that
\[
u^* e_n u = \diag \big( 1, 0, 0, \ldots, 0 \big)
\andeqn
u^* s u
 = \left[
   \begin{array}{ccccl}
  0     &  0     & \cdots &  0     &  1      \\
  1     &  0     & \cdots & \cdots & 0         \\
  0     &  1     & \ddots & \ddots & \vdots    \\
 \vdots & \ddots & \ddots & \ddots & \vdots    \\
  0     & \cdots &  0     &  1     &  0
\end{array}
    \right].
\]
For $k = 0, 1, \ldots, n - 1$,
it follows that
$u^* f_k u = (u^* s u) (u^* e_n u) (u^* s u)^{-k}$
is the orthogonal projection
(in the $p = 2$ sense)
to the span of the $k$-th standard basis vector
(starting the count at $0$ instead of~$1$).
Parts (\ref{I_7z11_mp2_orth}) and~(\ref{I_7z11_mp2_Sum1})
of the claim now follow immediately.

Set $n = 3$ and let $A$ be the subalgebra of $M_3^p$
spanned by $f_0$, $f_1$, and $f_2$.
This contains $1_{M_3}$.
Let $J = \Cdb e_3$, an ideal in $A$ with an identity of norm $1$.
We claim that
if
\begin{equation}\label{Eq_7Z11_pLarge}
p > \frac{\log (4)}{\log (4) - \log (3)}
\end{equation}
then
$A / J$ is not isometric to an $L^p$-operator algebra.
(This is presumably true
for all $p \in [1, \infty) \setminus \{ 2 \}$.)
Indeed, the image  $f$ of $f_1$ in $A / J$
is a contractive idempotent.
It is actually bicontractive since
\[
\| 1 - f \|
 = \inf \big\{ \| 1 - f_1 + \lambda f_0 \|
       \colon \lambda \in \Cdb \big\}
 \leq \| 1 - f_1 - f_0  \| = \| f_2  \| \leq 1 .
\]
We claim that if $p$ is as in~(\ref{Eq_7Z11_pLarge})
then $\| 1 - 2 f \| > 1$.
If we can prove this claim then
$A / J$ cannot be an $L^p$-operator algebra
by Lemma \ref{L_7719_BicId}~(\ref{L_7719_BicId_BicrOnLp}).

To prove the last claim note first that
since
\[
1 - 2 f_1 + \lambda f_0
 = s_1 \big( 1 - 2 e_3 + \lambda s_1^{-1} f_0 s_1 \big) s_1^{-1},
\]
we have
\[
\| 1 - 2 f \|
 = \inf \big\{ \| 1 - 2 f_1 + \lambda f_0 \| \colon
    \lambda \in \Cdb \big\}
 = \inf \big\{ \| 1 - 2 e_3
   + \lambda s_1^{-1} f_0 s_1 \| \colon \lambda \in \Cdb \big\}.
\]
With $\frac{1}{p} + \frac{1}{q} = 1$,
the norm of $1 - 2 e_3 + \lambda s_1^{-1} f_0 s_1$
dominates the $q$-norm of
the first row of $1 - 2 e_3 + \lambda s_1^{-1} f_0 s_1$.
This first row is
\begin{equation}\label{Eq_7Z11_qNormRow}
\left( \frac{1}{3} \,, \, - \frac{2}{3} \,, \, - \frac{2}{3} \right)
 - \frac{1}{3} \lambda \, \big( 1 \,, \,  \zt \,, \,  \zt^2 \big)
 =
\frac{1}{3} \, \big( 1 - \lambda \,, \, - 2 - \lambda \zt \,,
     \, - 2 - \lambda \zt^2 \big).
\end{equation}

We estimate the minimum of
\[
| 1 - \lambda |^q + | 2 + \lambda \zt |^q + | 2 + \lambda \zt^2 |^q
 = | 1 - \lambda |^q + \big| 2 {\overline{\zt}} + \lambda \big|^q
     + \big| 2 {\overline{\zt}}^2 + \lambda \big|^q.
\]
Write $\lambda = x + i y$ for real $x$ and~$y$.
Then
\[
2 {\overline{\zt}} + \lambda = - 1 + x + i \big( y - \sqrt{3} \big)
\andeqn
2 {\overline{\zt}}^2 + \lambda = - 1 + x + i \big( y + \sqrt{3} \big).
\]
Thus we are minimizing
\[
G (x, y)
  = \big( (1 - x)^2 + y^2 \big)^{q / 2}
     + \big( (1 - x)^2 + \big( y - \sqrt{3} \big)^2 \big)^{q / 2}
     + \big( (1 - x)^2 + \big( y + \sqrt{3} \big)^2 \big)^{q / 2}.
\]
Clearly $G (x, y) \geq G (1, y)$
for all $x, y \in \Rdb$.
So we must minimize the function
\[
g_c (y) = | y |^q + | y - c |^q + | y + c |^q
\]
for $c = \sqrt{3}$.
For any $c > 0$,
this function is continuous, even, and clearly strictly increasing
on $[c, \infty)$.
For $y \in (0, c)$ we have
\[
g_c' (y) = q \big( y^{q - 1} + (c + y)^{q - 1} - (c - y)^{q - 1} \big).
\]
Since $q - 1 \geq 0$ and $c + y > c - y > 0$,
it follows that $g_c' (y) > 0$.
By symmetry,
the minimum value of $g_c$ occurs at $y = 0$.
So,
for all $x, y \in \Rdb$,
we have $G (x, y) \geq G (1, 0) = 2 \cdot 3^{q / 2}$.

Applying this estimate to the $q$-norm of the right hand side
of~(\ref{Eq_7Z11_qNormRow}),
we get
\[
\| 1 - 2 f \|^q
 \geq \frac{2 \cdot 3^{q / 2}}{3^q}
 = 2 \cdot 3^{- q / 2}.
\]
If $q < 2 \log (2) / \log (3)$,
this quantity is greater than~$1$,
and this happens exactly when (\ref{Eq_7Z11_pLarge}) holds.
The claim is proved.
\end{example}

\subsection{Unitization of nonunital $L^p$-operator algebras}
Unfortunately Meyer's beautiful unitization theorem
(see \cite[Corollary 2.1.15]{BLM})
for operator algebras on Hilbert spaces
fails badly for $L^p$-operator algebras.
That is, unitizations of nonunital $L^p$-operator algebras
are not unique isometrically
(Example~\ref{E_7Z11_UnitzeC}
and Example~\ref{E_7Z11_Unitzec0PlusC} below).
However if two approximately unital $L^p$-operator algebras
$A_1$ and $A_2$
are isometrically isomorphic
and they each act nondegenerately
on the $L^p$ spaces on which they act,
then $A_1 + \Cdb 1$ is isometrically isomorphic to $A_2 + \Cdb 1$.
Indeed, for $j = 1, 2$, the algebra $A_j + \Cdb 1$
is isometrically isomorphic to the
multiplier unitization of $A_j$ by Lemma \ref{mund}.

We now illustrate the failure of Meyer's theorem,
even in the case of approximately unital $L^p$-operator algebras.
We give two versions.
In the first, the algebras are finite dimensional and already unital,
but degenerately represented.
In the second,
the algebras are genuinely nonunital.

\begin{example}\label{E_7Z11_UnitzeC}
Let $M_2^p = B (l^p_2)$ be as in Notation~\ref{N_7719_Mnp}.
Let $e = e_2$ be as in Example~\ref{mp},
and let $f = e_{1, 1}$,
the $(1, 1)$ standard matrix unit.
Let $1 = 1_{M_2}$ be the $2 \times 2$ identity matrix.
Then $\Cdb e \cong \Cdb f$ isometrically.
We claim that $\Cdb e + \Cdb 1$ is
not isometric to $\Cdb f + \Cdb 1$,
so that Meyer's unitization theorem fails.
The idempotents in $\Cdb f + \Cdb 1$ are $0$, $f$, $1 - f$, and $1$,
all of which are clearly hermitian.
By Example~\ref{mp}, however,
$e$ is a non-hermitian idempotent in $\Cdb e + \Cdb 1$.
The claim follows.
\end{example}

\begin{example}\label{E_7Z11_Unitzec0PlusC}
We continue with the notation in Example~\ref{E_7Z11_UnitzeC},
to produce
a nonunital example where Meyer's unitization theorem fails.
Set $A = c_0 \oplus \Cdb e$
and $B = c_0 \oplus \Cdb f$,
both viewed as subalgebras of $B \big(l ^p (\Ndb) \oplus^p l^p_2 \big)$.
These are isometrically isomorphic $L^p$-operator algebras,
which are approximately unital.
Indeed,
they have obvious increasing approximate identities consisting of
hermitian idempotents.
Write $1$ for the identity of $B \big(l ^p (\Ndb) \oplus^p l^p_2 \big)$.
We claim that
$A + \Cdb 1$ is not isometrically isomorphic to $B + \Cdb 1$.
To see this,
first observe that all elements of $B + \Cdb 1$
are multiplication operators on
$l^p (\Ndb) \oplus^p l^p_2 = l^p (\Ndb \amalg \{0, 1 \})$.
It is immediate that all idempotents in this algebra are hermitian.
On the other hand,
there is a canonical restriction homomorphism
$\rho \colon A + \Cdb 1 \to B (l^p_2)$,
which is a unital contractive surjection to $\Cdb e + \Cdb 1_{M_2}$,
namely
``compression" to the subspace $l^p_2$ of $l ^p (\Ndb) \oplus^p l^p_2$.
As we said in Example~\ref{E_7Z11_UnitzeC},
$e \in \Cdb e + \Cdb 1_{M_2}$
is a non-hermitian idempotent.
However, $g = (0, e) \in A \subseteq A + \Cdb 1$
is an idempotent such that $\rh (g) = e$.
If $g$ were hermitian,
then $e$ would be too, by Lemma 6.7 of~\cite{PV}.
So $A + \Cdb 1$ contains a non-hermitian idempotent.
\end{example}

In Example~\ref{E_7Z11_Unitzec0PlusC},
one can show that the algebra $B + \Cdb 1$
is a spatial $L^p$~AF algebra in the sense of
Definition~9.1 of~\cite{PV},
while $A + \Cdb 1$ isn't.

We remark that \cite[Proposition 9.9]{PV} gives conditions
which force uniqueness of the unitization
of an $L^p$-operator algebra.
The fact that Meyer's theorem fails for $\Cdb e$ and $\Cdb f$
in Example~\ref{E_7Z11_UnitzeC} shows,
by Meyer's proof (see 2.1.14 in \cite{BLM}), that,
even in $M_2^p = B (l^p_2)$,
some of the basic properties
of the Cayley transform for Hilbert space operators
must fail for $p \neq 2$.
We turn to this next.

\subsection{The Cayley and ${\mathfrak{F}}$ transforms}\label{Cayl}
The Cayley transform $\kappa (x) = (x - 1) (x + 1)^{-1}$
is an important tool for operator algebras on a Hilbert space,
as is the fact that in that setting $\kappa(x)$
is a contraction on accretive~$x$.
In \cite{BRII, BRord} the associated transform
\[
{\mathfrak{F}} (x) = x (x + 1)^{-1} = \frac{1}{2} (1 + \kappa (x))
\]
is used.
For $L^2$-operator algebras it takes ${\mathfrak{r}}_A$
onto the strict contractions in $\frac{1}{2} {\mathfrak{F}}_A$.
This all fails in full generality for $L^p$-operator algebras,
which means that many of the general results in \cite{BOZ}
do not improve for $L^p$-operator algebras.

Here are two things which do work.
First,
if $A$ is an approximately unital $L^p$-operator algebra
then the ${\mathfrak{F}}$ transform
does map  ${\mathfrak{r}}_A$ into  ${\mathfrak{F}}_A$.
(By Lemma~3.4 of~\cite{BOZ},
this is true for arbitrary approximately unital Banach algebras.)
Second, if $A$ is any unital Banach algebra
and $x \in {\mathfrak{F}}_A$,
then $\| \kappa (x) \| = \| 1 - 2 {\mathfrak{F}} (x) \| \leq 1$.
Indeed, with $y = x - 1$,
we have $\| y \| \leq 1$, so that
\[
\| \kappa (x) \|
  = \big\| \left( 1 + \tfrac{1}{2} y \right)^{-1}
                  \left( \tfrac{1}{2} y \right)  \big\|
  \leq \left\| \tfrac{1}{2} y  \right\| \sum_{k = 0}^\infty
            \left\| \tfrac{1}{2} y  \right\|^k
  \leq 1 .
\]

\begin{example}\label{E_7Z21_NotContr}
We prove the existence of $\dt > 0$
such that for all $p \in [1, 1+\dt)$
there is a unital finite dimensional $L^p$-operator algebra
containing a real positive element $x$
for which $\| \kp (x) \| > 1$.
Presumably this happens for all $p \in [1, \infty) \setminus \{ 2 \}$,
but proving this may require more work.

Indeed in $M^p_2$  (Notation~\ref{N_7719_Mnp}) consider
\[ x = \left[
   \begin{array}{ccl} 1 - i  &  1  \\ 1 &  1 - i \end{array} \right]
\andeqn
\kp (x) = \frac{1}{5} \left[
   \begin{array}{ccl} 1 - 3 i  &  1 + 2 i
     \\ 1 + 2 i &  1 - 3 i \end{array} \right].
\]
Since $x = 2 e_2 - i 1_{M_2}$ in the notation of Example \ref{mp},
it follows from considerations
in that example that $x$ is real positive in $M^p_2$.
However $\kappa (x)$ applied to the unit vector $(1, 0)$
has $p$-norm $\frac{1}{5} (10^{p / 2} + 5^{p / 2})^{1 / p}$, which
exceeds $1$ for $p \in [1, \dt)$, for some fixed $\dt > 0$.

One may also arrive at this same example
by modifying the $L^1$-operator algebra
example given in Example 3.14 in \cite{BOZ}.
It was stated there that the
convolution algebra
$l^1 (\Zdb_2)$ contains real positive elements $x$
for which $\| \kp (x) \| > 1$.
An explicit example of such an element
was not given there though.
Let $F^p_{\mathrm{r}} (\Zdb_2)$ be
the reduced group $L^p$-operator algebra of the two element group
(as defined in \cite{P2}).
This is isometric,
via the regular representation of $\Zdb_2$ on $l^p (\Zdb_2)$,
to the unital
subalgebra of $M_2^p$
generated by the idempotent
\[
e = \frac{1}{2} \left[
  \begin{array}{ccl} 1  &  1  \\ 1 &  1 \end{array} \right]
\]
(called $e_2$ in Example~\ref{mp}).
This latter algebra contains our element $x$ above.
The regular representation of $\Zdb_2$ on $l^p (\Zdb_2)$
sends the nontrivial group element to
\[
s = \left[
  \begin{array}{ccl} 0  &  1  \\ 1 &  0 \end{array} \right],
\]
and we have the relations $e = \frac{1}{2} (s + 1)$
and $s = 2 e - 1$.

Moreover, $F^1_{\mathrm{r}} (\Zdb_2) \cong l^1 (\Zdb_2)$
isometrically.
Via these considerations, a real positive element $w$ in
Example 3.14 in \cite{BOZ}
corresponds to a real positive element $a$
in $F^1_{\mathrm{r}} (\Zdb_2)$
and a real positive matrix $x$ in $M^1_2$.
Moreover, $\| \kp (w) \| > 1$ if and only if $\| \kp (a) \| > 1$,
in turn if and only if $\| \kp (x) \| > 1$.
Since
the map $F^1_{\mathrm{r}} (\Zdb_2) \to F^p_{\mathrm{r}} (\Zdb_2)$
is unital and contractive for $p \in [1, \infty)$,
it follows easily that $a$ (resp.~$x$) is also real positive in
$F^p_{\mathrm{r}} (\Zdb_2)$  (resp.\  $M^p_2$).
By ``continuity in $p$'', the Cayley transform of $x$ in $M^p_2$ is not
contractive for $p$ close to $1$.
A specific example of this of course is the matrix $x$
in the second paragraph of the present example.
\end{example}

\subsection{Support idempotents}\label{SupP}
There is some improvement over~\cite{BOZ}
in the theory of support idempotents.

\begin{proposition}\label{L_7Z17_ExistSupp}
Let $A$ be an approximately unital Arens regular Banach algebra,
and let $x \in {\mathfrak{r}}_A$.
Then, using the notation of Definition~{\rm{\ref{N_7919_nthroot}}},
the sequence $(x^{1/n})_{n \in \Ndb}$
has a weak* limit $s (x) \in A^{**}$.
Moreover:
\begin{enumerate}
\item\label{L_7Z17_ExistSupp_Idmpt}
$s (x)$ is an idempotent.
\item\label{L_7Z17_ExistSupp_Idnty}
$s (x)$ is an identity for the second dual
of the closed subalgebra of~$A$ generated by~$x$.
\item\label{L_7Z17_ExistSupp_IdOnx}
$s (x) x = x s (x) = x$.
\item\label{L_7Z17_ExistSupp_F}
With ${\mathfrak{F}}$ as in Subsection {\rm{\ref{Cayl}}},
we have
$s ({\mathfrak{F}} (x)) = s (x)$.
\item\label{L_7Z17_ExistSupp_Norm}
$\| 1 - s (x) \| \leq 1$.
\item\label{L_7Z17_ExistSupp_RPos}
$s (x)$ is a real positive idempotent in $A^{**}$.
\end{enumerate}
\end{proposition}

\begin{definition}\label{D_7Z17_SuppId}
Let $A$ and $x \in A$ be as in Proposition~\ref{L_7Z17_ExistSupp}.
We call $s (x)$ the {\emph{support idempotent}} of~$x$.
\end{definition}

Proposition~\ref{L_7Z17_ExistSupp} is proved in the discussion
after
Proposition 3.17 of~\cite{BOZ}
(see also the discussion
after Corollary 6.20 in \cite{BSan}).
Our advantage here over the situation in
those papers is that the weak* limit of $x^{1 / n}$ exists
(it equals the identity for the second dual
in (\ref{L_7Z17_ExistSupp_Idnty}) above),
and so the support idempotent $s (x)$ is unique.

The support idempotent of~$x$
is minimal in several senses
related to the orderings in Definition~\ref{Dorder}.

\begin{corollary}\label{L_7Z17_MoreSupp}
Under the hypotheses of
Proposition~{\rm{\ref{L_7Z17_ExistSupp}}},
we furthermore have:
\begin{enumerate}
\item\label{L_7Z17_ExistSupp_MinL}
If $f \in A^{**}$ is any idempotent with $f x = x$,
then $f s (x) = s (x)$,
that is, $s (x) \leq_{\mathrm{r}} f$
in the sense of Definition~{\rm{\ref{Dorder}}}.
\item\label{L_7Z17_ExistSupp_RMin}
If $f \in A^{**}$ is any idempotent with $x f = x$,
then $s (x) f = s (x)$.
\item\label{L_7Z17_ExistSupp_RPosMin}
If $f \in A^{**}$ is any 
idempotent
with $f x = x$ 
and $x f = x$,
then $s (x) \leq f$
in the sense of Definition~{\rm{\ref{Dorder}}}.
\end{enumerate}
\end{corollary}

\begin{proof}
By Proposition \ref{L_7Z16_RootProp}~(\ref{L_7Z16_RootProp_PApp}), in
part (\ref{L_7Z17_ExistSupp_MinL}) we have $f x^{1/n} = x^{1/n}$.
Hence  (\ref{L_7Z17_ExistSupp_MinL})
follows from $x^{1/n} \to s (x)$ weak*
and separate weak* continuity of multiplication (\cite[2.5.3]{BLM}).
Similarly
for (\ref{L_7Z17_ExistSupp_RMin}).
Part~(\ref{L_7Z17_ExistSupp_RPosMin}) is now obvious.
\end{proof}

Thus $s(x)$ is the
smallest idempotent in $A^{**}$ with $f x = x$,
in the ordering $\leq_{\mathrm{r}}$ (or with
$f x = x$ and $x f = x$, in the ordering $\leq$).
Recall from Corollary~\ref{C_7620_2Dual}
that if $A$ is an $L^p$-operator algebra
then so is $A^{**}$,
and so by Lemma \ref{orderidemp}~(\ref{orderidemp_RtoNone})
we see that  $\leq$ coincides with $\leq_{\mathrm{r}}$
on real positive idempotents in $A^{**}$.
Hence in this case $s(x)$ is the
smallest idempotent in $A^{**}$ with $f x = x$
(or with $xf = x$), in the ordering $\leq$.

In the case of a subalgebra of $B (L^p (X))$,
we also get a support idempotent acting on $L^p (X)$.

\begin{proposition}\label{L_7Z17_SuppOnSpace}
Let $p \in (1, \infty)$,
let $A \subseteq B (L^p (X))$ be an approximately unital
closed subalgebra,
and let $x \in {\mathfrak{r}}_A$.
Let $s (x)$ be as in Proposition~{\rm{\ref{L_7Z17_ExistSupp}}}.
Let $\ph \colon A^{**} \to B (L^p (X))$ be the contractive
homomorphism obtained from the identity representation of~$A$
as in Lemma~{\rm{\ref{L_7702_WkStarExt}}},
and set $e = \varphi (s (x))$.
Then:
\begin{enumerate}
\item\label{L_7Z17_SuppOnSpace_Idem}
$e$ is an idempotent with range ${\overline{x L^p (X)}}$,
and $e$ is real positive if $A$ is nondegenerate.
\item\label{L_7Z17_SuppOnSpace_Idty}
$e x = x e = x$.
\item\label{L_7Z17_SuppOnSpace_Comp}
${\overline{x L^p (X)}}$ is a complemented subspace
of $L^p (X)$.
\item\label{L_7Z17_SuppOnSpace_SOT}
Using the notation of Definition~{\rm{\ref{N_7919_nthroot}}},
$x^{1/n} \to e$ in the strong operator topology on $B (L^p (X))$.
\item\label{L_7Z17_SuppOnSpace_Min}
If $A$ is nondegenerate
and $f \in B (L^p (X))$ is any real positive idempotent
with $f x = x$ or $x f = x$,
then $e \leq f$
in the sense of Definition~{\rm{\ref{Dorder}}}.
\end{enumerate}
\end{proposition}

Nondegeneracy is probably needed for real positivity
in~(\ref{L_7Z17_SuppOnSpace_Idem})
and for~(\ref{L_7Z17_SuppOnSpace_Min}).
Otherwise,
letting $f$ be as in Lemma~\ref{2deg},
our proof below only yields a real positive idempotent
in $B (f L^p (X) )$.

\begin{proof}[Proof of Proposition~{\rm{\ref{L_7Z17_SuppOnSpace}}}]
Let $E \subseteq L^p (X)$ and the idempotent $f \in B (L^p (X) )$
be as in Lemma~\ref{2deg}.
We first claim that $\varphi (1) = f$.
(Indeed this is always the case in the situation of
Lemma~\ref{2deg} provided that $A$ is Arens regular,
by the following simple argument.)
Let $(e_t)_{t \in \Ld}$ be a cai for~$A$.
Then $e_t \to 1$ weak* in $A^{**}$
by Lemma~\ref{L_7917_AppIdConverge}.
Therefore $e_t \to \varphi (1)$ weak* in $B (L^p (X))$
by weak* to weak* continuity of~$\varphi$.
Also $e_t \to f$ weak* in $B (L^p (X))$
by Lemma \ref{2deg}~(\ref{2deg_AI}).
The claim is proved.

We have $x^{1/n} \to e$ weak* in $B (L^p (X))$.
Since $\varphi$ is a homomorphism, $e$ is
an idempotent satisfying $e x = x e = x$,
which is~(\ref{L_7Z17_SuppOnSpace_Idty}).
Using Proposition \ref{L_7Z17_ExistSupp}~(\ref{L_7Z17_ExistSupp_Norm})
and $\varphi (1) = f$,
we get $\| f - e \| \leq \| \varphi \| \| f - s (x) \| \leq 1$.
If $A$ is nondegerate, then $f = 1$,
so $e$ is real positive
by Lemma \ref{L_7719_BicId}~(\ref{L_7719_BicId_RealPos}).
Since $e x = x$,
we clearly have $x L^p (X) \subseteq e L^p (X)$.
So ${\overline{x L^p (X)}} \subseteq e L^p (X)$.
Since $x^{1 / n} \eta \to e \eta$ weakly for $\eta \in L^p (X)$,
it follows that $x L^p (X)$ is weakly, hence norm, dense in
$e L^p (X)$. 
Thus
$e L^p (X) = {\overline{x L^p (X)}}$ and we now
have all of~(\ref{L_7Z17_SuppOnSpace_Idem}),
as well as~(\ref{L_7Z17_SuppOnSpace_Comp}).

For $\et \in L^p (X)$
we have $x^{1/n} x \et \to x \et$ in norm. 
Since $(x^{1/n})_{n \in \Ndb}$
is a bounded sequence
(use Proposition \ref{L_7Z16_RootProp}~(\ref{L_7Z16_RootProp_bound})), 
it follows that $x^{1/n} e \xi \to e \xi$
for all $\xi \in L^p (X)$.
Clearly $x^{1/n} (1-e) \xi = 0 \to e (1-e) \xi$ also,
using (\ref{L_7Z17_SuppOnSpace_Idty})
and Proposition \ref{L_7Z16_RootProp}~(\ref{L_7Z16_RootProp_PApp}).
Thus we have~(\ref{L_7Z17_SuppOnSpace_SOT}).

For (5), note that  $f x = x$ if and only if $f e = e$
as in the proof of Corollary~\ref{L_7Z17_MoreSupp},
and similarly $xf = x$ if and only if $ef = e$.
Since $f e = e$ if and only if $ef = e$
by Lemma \ref{orderidemp}~(\ref{orderidemp_RevProd}),
the proof of~(\ref{L_7Z17_SuppOnSpace_Min})
is clear (as in the proof
of Corollary~\ref{L_7Z17_MoreSupp}).
\end{proof}

It is shown in
\cite[Corollary 3.19]{BOZ} that
if $x, y \in {\mathfrak{r}}_A$
then $\overline{x A} \subseteq \overline{y A}$
if and only if $s (y) s (x) = s (x)$.

\begin{lemma}\label{fref}
Let  $p \in (1, \infty)$.
Let $A$ be an approximately
unital $L^p$-operator algebra,
and let $x, y \in {\mathfrak{r}}_A$.
Then ${\overline{x A}} = {\overline{y A}}$
if and only if $s (x) = s (y)$.
\end{lemma}

\begin{proof}
If $s (x) = s (y)$ then  $\overline{xA}  = \overline{yA}$
by \cite[Corollary 3.18]{BOZ}.
Conversely,
if $\overline{xA}  = \overline{yA}$
then by
\cite[Corollary 3.18]{BOZ} we have $s (x) A^{**} = s (y) A^{**}$.
It follows that $s (x) s (y) = s (y)$ and $s (y) s (x) = s (x)$.
By Proposition \ref{L_7Z17_ExistSupp}~(\ref{L_7Z17_ExistSupp_RPos})
and Lemma \ref{orderidemp}~(\ref{orderidemp_RtoNone}),
the second equation implies $s (x) s (y) = s (x)$.
So $s (x) = s (y)$.
\end{proof}

Unlike the $L^2$-operator algebra case
(see e.g.\ \cite[Lemma 2.5]{BRI}),
if $x \in \frac{1}{2} {\mathfrak{F}}_A$
(that is, if $\| 1 - 2 x \| \leq 1$),
then $s (x)$ need not be contractive.
An example is
$x = \frac{1}{2} (1 - e_3)$, for $e_3$ as in Example \ref{mp}.

\subsection{Some consequences of strict convexity of
 $L^p$~spaces}\label{SomeC}

\begin{lemma}\label{scp}
Let $E$ be a strictly convex Banach space,
and let $f \in B (E)$ be a contractive idempotent.
Let $\xi \in E$ satisfy $\| f \xi \| = \| \xi \|$.
Then $f \xi = \xi$.
\end{lemma}

\begin{proof}
This is well known.
Suppose that $\xi \neq f \xi$.
Set $\et = \frac{1}{2} (\xi + f \xi)$.
Then $\| \et \| < \| f \xi \|$, giving the contradiction
$\| f \xi \| = \| f \et \| \leq \| \et \| < \| f \xi \|$.
\end{proof}

\begin{lemma}\label{da}
Let $p \in (1, \infty)$,
let $E$ and $F$ be Banach spaces,
and let $S \subseteq B (E, F)$ be a linear subspace.
Define matrix norms on $B (E, F)$
by interpreting elements of $M_n (B (E, F))$
as linear maps from the $l^p$~direct sum of $n$ copies of~$E$
to the $l^p$~direct sum of $n$ copies of~$F$.
Then any $\varphi \in \Ball (S^*)$
is $p$-completely contractive in the sense of \cite{PisCb}.
\end{lemma}

\begin{proof}
This follows by essentially the argument
in the $L^2$-operator space case,
and no doubt this is well known.
By the usual argument
(see e.g.\  the proof of \cite[Lemma 4.2]{Daws2}),
we have to show that
\[
\Bigg\| \sum_{j, k = 1}^n \, \beta_j x_{j, k} \alpha_k \Bigg\| \leq
\sup \Bigg( \Bigg\{ \Bigg( \sum_{j = 1}^n
      \Bigg\| \sum_{k = 1}^n x_{j, k} \xi_k \Bigg\|^p \Bigg)^{1/p}
       \colon \sum_{k = 1}^n \, \| \xi_k \|^p \leq 1 \Bigg\} \Bigg),
\]
where $n \in \Ndb$,
$\xi_1, \xi_2, \ldots, \xi_n \in E$,
$\beta = (\beta_1, \beta_2, \ldots, \beta_n) \in \Ball (l^q_n)$,
$\alpha  = (\alpha_1, \alpha_2, \ldots, \alpha_n) \in \Ball  (l^p_n)$,
and $x_{j, k} \in S$ for $j, k = 1, 2, \ldots, n$.
However the latter supremum may be written as
\[
\sup \Bigg( \Bigg\{ \Bigg| \sum_{j = 1}^n \,
   \psi_j \Bigg( \sum_{k = 1}^n \, x_{j, k} \xi_k \Bigg)  \Bigg| \colon
\sum_{k} \, \| \xi_k \|^p \leq 1 , \,\,
\sum_{j = 1}^n \| \psi_j \|^q \leq 1 \Bigg\} \Bigg),
\]
where $\xi_1, \xi_2, \ldots , \xi_n \in E$
and $\psi_1, \psi_2, \ldots, \psi_n \in F^*$.
This supremum clearly dominates
\[
\sup \Bigg( \Bigg\{ \Bigg|
   \sum_{j, k = 1}^n \,
      \beta_j \psi \big(x_{j, k} \alpha_k \xi \big) \Bigg|
\colon \psi  \in \Ball  (F^*), \,\, \xi \in \Ball  (E) \Bigg\} \Bigg),
\]
since $\sum_{j = 1}^n \, \| \beta_j \psi \|^q \leq 1$
and $\sum_{k = 1}^n \, \| \alpha_k \xi \|^p \leq 1$.
This last supremum is equal to
$\big\| \sum_{j, k = 1}^n \, \beta_j x_{j, k} \alpha_k \big\|$.
\end{proof}

Both the following lemmas apply in particular to hermitian idempotents,
by parts (\ref{L_7719_BicId_HToBicr})
and~(\ref{L_7719_BicId_RealPos})
of Lemma \ref{L_7719_BicId}.

\begin{lemma}\label{L_7Z18_BEVanishingState}
Let $E$ be a Banach space,
let $\omega \in \Ball  (E^*)$, and let $\xi \in \Ball  (E)$.
Let $\varphi$ be the vector state on $B (E)$
given by $\varphi (a) = \langle \omega, a \xi \rangle$
for all $a \in B (E)$.
Let $e \in B (E)$ be a real positive idempotent,
and suppose $\varphi (e) = 0$.
\begin{enumerate}
\item\label{L_7Z18_BEVanishingState_E}
If $E$ is strictly convex then
$\varphi (a e) = 0$ for all $a \in A$.
\item\label{L_7Z18_BEVanishingState_EStar}
If $E^*$ is strictly convex then
$\varphi (e a) = 0$ for all $a \in A$.
\end{enumerate}
\end{lemma}

\begin{proof}
{}From $\ph (e) = 0$ we get $\ph (1 - e) = 1$.
Also $\| 1 - e \| \leq 1$
by Lemma \ref{L_7719_BicId}~(\ref{L_7719_BicId_RealPos}).

Suppose $E$ is strictly convex.
We have
\[
\| (1 - e) \xi \| \geq | \varphi (1 - e) | = 1 = \| \xi \|.
\]
So $\xi = (1 - e) \xi$ by Lemma \ref{scp}.
For $a \in B (E)$
we then have $\ph (a e) = \langle \omega, a e \xi \rangle = 0$.

Now suppose $E^*$ is strictly convex.
We have $\| (1 - e)^* \| = \| 1 - e \| \leq 1$,
and $\ph (a) = \langle a^* \omega, \xi \rangle$
for all $a \in B (E)$,
so
\[
\| (1 - e)^* \om \| \geq | \varphi (1 - e) | = 1 = \| \om \|.
\]
So $\om = (1 - e)^* \om$ by Lemma \ref{scp}.
For $a \in B (E)$
we then have $\ph (e a) = \langle e^* \omega, a \xi \rangle = 0$.
\end{proof}

\begin{lemma}\label{L_7Z18_LpVanishingState}
Let  $p \in (1, \infty)$,
and let $A$ be a unital $L^p$-operator algebra.
Let $\varphi$ be a state on $A$
and let $e \in A$ be a real positive idempotent.
If $\varphi (e) = 0$ then
$\varphi (a e) = \varphi (e a) = 0$ for all $a \in A$.
\end{lemma}

\begin{proof}
We may assume that $A$ is a unital subalgebra of $B (L^p (X))$
for some~$X$.
By Lemma \ref{da},
$\varphi$ is $p$-completely contractive in the
sense of \cite{PisCb}.
So by Theorem 2.1 of that paper and the remark
after it,
and using the fact that ultraproducts of $L^p$ spaces are $L^p$ spaces
(Theorem 3.3 (ii) of~\cite{He}),
there exist an ${\mathrm{SQ}}_p$-space $E$,
$\xi \in \Ball (E)$,
$\omega \in \Ball (E^*)$,
and a $p$-completely contractive map $\pi \colon A \to B (E)$
such that $\varphi (a) = \langle \omega, \pi (a) \xi \rangle$
for all $a \in A$.
It is easy to see
and no doubt well known that $\pi$ may be taken to be a
unital homomorphism.
Then $\pi (e)$ is an
idempotent, and $\| 1 - \pi (e) \| \leq 1$.
As explained in Remark~\ref{sqprep},
${\mathrm{SQ}}_p$-spaces are both smooth and strictly convex.
So their duals are also strictly convex.
We may therefore apply Lemma~\ref{L_7Z18_BEVanishingState}
to the vector state $\langle \omega, \, \cdot \, \xi \rangle$
on $B (E)$.
Thus for all $a \in E$ we have
\[
\varphi (a e) = \langle \omega, \, \pi (a) \pi (e) \xi \rangle = 0
\andeqn
\varphi (e a) = \langle \omega, \, \pi (e) \pi (a) \xi \rangle = 0.
\]
This completes the proof.
\end{proof}

\begin{remark}\label{tsres}

\begin{enumerate}
\item\label{tsres_1}
Lemma~\ref{L_7Z18_LpVanishingState} holds if $A$ is a unital
${\mathrm{SQ}}_p$-operator algebra.
The proof is the same too,
but with $L^p$ replaced by ${\mathrm{SQ}}_p$ throughout.
For
further details on the construction of $\pi$ in this
case see e.g.\  \cite[Theorem 4.1]{Daws2} and references therein.
\item\label{tsres_2}
Lemma~\ref{L_7Z18_LpVanishingState} holds
for an approximately unital
$L^p$- (or ${\mathrm{SQ}}_p$-) operator algebra $A$,
and indeed holds for restrictions
of states on any $L^p$- (or ${\mathrm{SQ}}_p$-) operator algebra
unitization of $A$.
This follows by applying the unital case to the
extending state on the unitization of $A$.
\end{enumerate}
\end{remark}

\begin{corollary}\label{supp1}
Let $p \in (1, \infty)$.
Let $A$ be an approximately unital $L^p$-operator algebra.
If $x \in {\mathfrak{r}}_A$
and $\varphi \in S (A)$ with $\varphi (s (x)) = 0$,
then $\varphi (x) = 0$.
Conversely, if further $x \in {\mathfrak{F}}_A$
and $\varphi (x) = 0$ then $\varphi (s (x)) = 0$.
\end{corollary}

\begin{proof}
We may work in~$A^1$ by extending $\varphi$ to a state there,
and we may thus assume that $A$ is unital.

The idempotent $s (x)$ is real positive
by Proposition \ref{L_7Z17_ExistSupp}~(\ref{L_7Z17_ExistSupp_RPos}).
Using
Proposition \ref{L_7Z17_ExistSupp}~(\ref{L_7Z17_ExistSupp_IdOnx}),
Lemma~\ref{L_7Z18_LpVanishingState},
and $\ph (s (x)) = 0$, we get $\ph (x) = 0$.

On the other hand,
if $x \in {\mathfrak{F}}_A$
and $\varphi (x) = 0$ then $\varphi (1 - x) = 1$.
As in the proof of
Lemma~\ref{L_7Z18_LpVanishingState},
there are an ${\mathrm{SQ}}_p$-space $F$,
a contractive unital homomorphism $\pi \colon A^1 \to B (F)$,
$\xi \in \Ball (F)$, and  $\eta \in \Ball  (F^*)$,
such that $\varphi = \langle \pi (\cdot) \xi, \eta \rangle$
for all $a \in A^1$.
Then
\[
1 = \varphi (1 - x)
 = \langle \omega , (1 - \pi (x)) \xi \rangle
 \leq \| \pi (1 - x) \xi  \|
 \leq 1.
\]
Therefore,
with $\frac{1}{p} + \frac{1}{q} = 1$,
both $\xi$ and $(1 - \pi (x)) \xi$
define norm one linear functionals on $L^q (X)$
which take $\om$ to~$1$.
Strict convexity of $L^q (X)$
implies $(1 - \pi (x)) \xi = \xi$.
So $\pi (x) \xi = 0$.
Proposition \ref{L_7Z16_RootProp}~(\ref{L_7Z16_RootProp_PApp})
now implies that $\pi (x^{1 / n}) \xi = 0$
for all $n \in \Ndb$.
Hence  $\varphi (x^{1 / n}) = 0$.
In the limit $\varphi (s (x)) = 0$.
\end{proof}

The next lemma is a generalization of
part of \cite[Lemma 2.10]{BRI},
with a similar proof
but using Corollary \ref{supp1}.

\begin{lemma}\label{infa}
Let  $p \in (1, \infty)$.
Let $A$ be an approximately unital
$L^p$-operator algebra, and let $x \in {\mathfrak{F}}_A$.
The following are equivalent:
\begin{enumerate}
\item\label{Linfa1}
$s (x) = 1$.
\item\label{Linfa2}
$\varphi (x) \neq 0$ for all $\varphi \in S (A)$.
\item\label{Linfa3}
$\Re ( \varphi (x) ) > 0$ for all $\varphi \in S (A)$.
\end{enumerate}
If $x \in {\mathfrak{r}}_A$ then {\rm{(\ref{Linfa3})}}
implies~{\rm{(\ref{Linfa2})}}
and  {\rm{(\ref{Linfa2})}} implies~ {\rm{(\ref{Linfa1})}}.
\end{lemma}

\begin{proof}
Let $x \in {\mathfrak{r}}_A$.
Then (\ref{Linfa3}) implies~(\ref{Linfa2}) trivially.
To show that (\ref{Linfa2}) implies~(\ref{Linfa1}),
suppose (\ref{Linfa1}) fails.
Represent $A^{**}$ 
as a unital subalgebra of $B (L^p (X))$
for some~$X$ by Corollary~\ref{ddag}.
Choose $\xi \in \Ball (L^p (X))$ in the range of the
idempotent $1 - s (x)$,
and choose $\eta \in \Ball (L^p (X)^*)$ with
$\langle \xi, \eta \rangle = 1$.
Then $\varphi (x) = \langle x \xi, \eta \rangle$
defines a state on $A$ with $\varphi (1 - s (x)) = 1$.
Since $\varphi (s (x)) = 0$,
Corollary \ref{supp1} implies
$\varphi (x) = 0$.

If $x \in {\mathfrak{F}}_A$ then
(\ref{Linfa1}) implies~(\ref{Linfa2}) by Corollary \ref{supp1}.
For (\ref{Linfa2}) implies~(\ref{Linfa3}),
follow part of the proof of \cite[Lemma 2.10]{BRI}:
$| 1 - \ph (x) | \leq 1$
is not compatible with both $\ph (x) \neq 0$
and $\Re ( \varphi (x) ) \leq 0$.
\end{proof}

\subsection{Hahn-Banach smoothness of $L^p$-operator algebras}

\begin{definition}\label{D_7Z20_HBS}
Let $E$ be a Banach space and let $M \subseteq E$
be a closed subspace.
We say that $M$ is {\emph{Hahn-Banach smooth}} in~$E$
if for every $\om_0 \in M^*$
there is a unique $\om \in E^*$ with $\| \om \| = \| \om_0 \|$
and $\om |_M = \om_0$,
\end{definition}

Existence of $\om$ is just the Hahn-Banach Theorem.
When verifying this property,
we need only consider the case $\| \om_0 \| = 1$.

Proposition 2.1.18 in \cite{BLM} works for $L^p$-operator algebras.

\begin{proposition}\label{2extension}
Let $p \in (1, \infty)$.
Let $A$ be an approximately unital
$L^p$-operator algebra
and denote the identity
of $A^1$ by~$1$.
\begin{enumerate}
\item\label{2extension_1}
Let $(e_t)_{t \in \Lambda}$ be a cai in~$A$.
If $\psi \colon A^1 \to \Cdb$ is a functional on $A^1$,
then $\lim_t \psi (e_t) = \psi (1)$
if and only if $\| \psi \| = \| \psi |_A \|$.
\item\label{2extension_2}
$A$ is Hahn-Banach smooth in~$A^1$
(Definition~{\rm{\ref{D_7Z20_HBS}}}).
\end{enumerate}
\end{proposition}

\begin{proof}
We may assume that $A$ is nonunital (the case
of unital algebras being easy).

The forward direction of (\ref{2extension_1}) is
just as in the proof of \cite[Proposition 2.1.18]{BLM}.

For the other direction
suppose that  $\psi \colon A^1 \to \Cdb$
with $\| \psi \| = \| \psi |_A \| = 1$.
As in the proof of
Lemma~\ref{L_7Z18_LpVanishingState},
there are an ${\mathrm{SQ}}_p$-space $F$,
a contractive unital homomorphism $\pi \colon A^1 \to B (F)$,
$\xi \in \Ball (F)$, and  $\eta \in \Ball  (F^*)$,
such that $\psi = \langle \pi (\cdot) \xi, \eta \rangle$
for all $a \in A^1$.

Apply the extension of Lemma~\ref{2deg}
given in Remark~\ref{sqprep} to the representation $\pi |_A$.
Let $E \subseteq F$ and the idempotent $f \in B (F)$
be as there.
The extensions of parts (\ref{2deg_AI}) and~(\ref{2deg_Compr})
of Lemma \ref{2deg}
imply that
$\pi (e_{t}) \to f$ weak* in $B (F)$
and $\pi (a) = \pi (a) f$ for all $a \in A$.
Thus, for all $a \in A$,
\[
\bigl| \langle \pi (a) \xi, \eta \rangle \bigr|
 = \bigl| \langle \pi (a) f \xi,  \eta  \rangle \bigr|
 \leq \| a \| \| f \xi \|.
\]

This shows that $\| \psi |_{A} \| \leq \| f \xi \|$.
Hence, by hypothesis, $\| f \xi \| = \| \xi \| = 1$.
Since $E$ is strictly convex
(see Remark~\ref{sqprep}),
Lemma \ref{scp} implies $f \xi = \xi$,
that is, $\xi \in {\overline{\spn}} (\pi (A) E)$.
Now,
since $\pi (e_t) \to f$ weak* we have
\[
\langle \pi (e_t) \xi, \eta \rangle
 \longrightarrow \langle f \xi, \eta \rangle
 = \langle \xi, \eta \rangle,
\]
which says that
$\psi (e_t) \to \psi (1)$.

For the deduction of (\ref{2extension_2})
from (\ref{2extension_1}),
let $\ph \in A^*$ satisfy $\| \ph \| = 1$.
Proceed as in
the proof of \cite[Proposition 2.1.18]{BLM},
but beginning by writing $\varphi \in A^*$ as
$\varphi = \langle \pi (\cdot) \zeta, \eta \rangle$ for $E$ as above,
and for a contractive homomorphism $\pi \colon A \to B (E)$,
$\zeta \in \Ball  (E)$, and  $\eta \in \Ball  (E^*)$.
This may be done for example by considering a Hahn-Banach extension
of $\varphi$ to $A^1$ and using the unital case above.
\end{proof}

\begin{corollary}\label{HBsm}
Let  $p \in (1, \infty)$,
and let $A$ be a nonunital approximately unital
$L^p$-operator algebra.
\begin{enumerate}
\item\label{HBsm_EqState}
Let $\ph \in A^*$ satisfy $\| \ph \| = 1$.
Then the following are equivalent:
\begin{enumerate}
\item\label{HBsm_EqState_State}
$\ph$ is a state on~$A$,
that is (see Definition~{\rm{\ref{Dstate}}}),
$\ph$ extends to a state on~$A^1$.
\item\label{HBsm_EqState_All}
$\ph (e_t) \to 1$ for every cai $(e_t)_{t \in \Lambda}$ for~$A$.
\item\label{HBsm_EqState_Some}
$\ph (e_t) \to 1$ for some cai $(e_t)_{t \in \Lambda}$ for~$A$.
\item\label{HBsm_EqState_1}
$\ph (1_{A^{**}}) = 1$.
\end{enumerate}
\item\label{HBsm_Uniq}
Every state on $A$ has a unique
extension to a state on~$A^1$.
\end{enumerate}
\end{corollary}

\begin{proof}
Everything is immediate from Proposition~\ref{2extension}.
\end{proof}

Part~(\ref{HBsm_EqState})
says that states on such algebras may be defined
by any one of the equivalent conditions
in Lemma~2.2 of \cite{BOZ}.
The change in the statement of the last condition
is justified by Lemma~\ref{L_7917_AppIdConverge}.

In the notation of Definition~\ref{Dstate}
(taken from~\cite{BOZ}),
for any cai ${\mathfrak{e}} = (e_t)_{t \in \Lambda}$ of $A$
we have $S_{\mathfrak{e}} (A) = S (A)$.
That is, states on an approximately unital $L^p$-operator algebra
are the contractive functionals $\varphi$ with
$\varphi (e_t) \to 1$,
or equivalently have norm $1$ and extend
to a state on $A^1$ (or on $A^{**}$).

We remark that the last several results hold (beginning with
Lemma \ref{L_7Z18_LpVanishingState}) if $A$ is an
approximately unital ${\mathrm{SQ}}_p$-operator algebra.
The proofs are almost
identical, but with the kinds of emendations
prescribed in the proof of Lemma \ref{L_7618_SecDual}
for ${\mathrm{SQ}}_p$-spaces,
Remark \ref{tsres}, and Remark~\ref{sqprep}.

The definition of a scaled Banach algebra,
used in the next proposition,
is stated in the Introduction
(see also the beginning of Section~\ref{Sec_6} below).

\begin{proposition}\label{scas}
Suppose that $A$ is an approximately unital scaled
Banach algebra,
that $A$ is Hahn-Banach smooth in~$A^1$
(Definition~{\rm{\ref{D_7Z20_HBS}}}), and that $A^{**}$ is unital.
Then  ${\mathfrak{r}}_{A^{**}}$ as defined in~\cite{BOZ}
(after Lemma~{\rm{2.5}} there)
agrees with the set of accretive elements
of the unital Banach algebra $A^{**}$.
\end{proposition}

We are ignoring the statement in~\cite{BOZ}
that the definition there is only to be applied when $A^{**}$
is not unital.

To be explicit,
let $R_0$ be the set of accretive elements of $A^{**}$,
where $A^{**}$ is
thought of as a unital Banach algebra in its own right,
and let $R_1$ be the analogous subset of $(A^1)^{**}$.
Then the assertion of the proposition is that
$R_1 \cap A^{**} = R_0$.

\begin{proof}[Proof of Proposition~{\rm{\ref{scas}}}]
We may assume that $A$ is nonunital (the case
of unital algebras being easy).
To avoid confusion,
we use the notation $R_0$ and $R_1$ above.

We show $R_1 \cap A^{**} \subseteq R_0$.
Proposition 2.11 of~\cite{BOZ}
(which works also when $A^{**}$ is unital)
implies that the weak* closure of ${\mathfrak{r}}_A$
is~$R_1 \cap A^{**}$.
So we need to show that ${\mathfrak{r}}_{A} \subseteq R_0$
and that $R_0$ is weak* closed.
The second part
is e.g.\ Theorem~2.2 of~\cite{Mag};
the set $D_{A^{**}}$
(following the notation there) is $\{ - a \colon a \in R_0 \}$.
One way to see the first part is that
part of the proof of
Lemma~\ref{L_7917_AppIdConverge}
shows that every cai for~$A$ converges weak* to~$1_{A^{**}}$.
Given this,
the argument for
Lemma \ref{bicr}~(\ref{bicr_1})
shows that the subalgebra
$A + \Cdb \cdot 1_{A^{**}} \subseteq A^{**}$
is isometrically isomorphic to~$A^1$.
Thus, if $a \in {\mathfrak{r}}_A$,
then $a \in {\mathfrak{r}}_{A^1}$
by Definition~\ref{D_7702_RealPos},
so $a \in R_0$
by Lemma~\ref{L_7Z14_RPosSubalg}.

It remains to show that $R_0 \subseteq R_1$.
Let $a \in R_0$,
and let $\varphi$ be a state on $(A^1)^{**}$.
By weak* density of the normal states in $S ((A^1)^{**})$
(which follows from Theorem~2.2 of~\cite{Mag})
there is a net $(\psi_t)_{t \in \Ld}$ in $S (A^1)$
such that $\psi_t \to \ph$ weak*.
For $t \in \Ld$,
since $A$ is scaled
there are $\ld \in [0, 1]$ and $\om \in S (A)$
such that $\psi_t = \ld \om$.
Since $A$ is Hahn-Banach smooth in~$A^1$,
\cite[Lemma 2.2]{BOZ} implies that
the canonical weak* continuous extension of $\om$
is a state on $A^{**}$.
So $\Re (\om (a)) \geq 0$,
whence $\Re (\psi_t (a)) \geq 0$.
Then $\Re ( \varphi (a) ) = \lim_t  \Re ( \psi_t (a) ) \geq 0$.
So $a \in R_1$.
\end{proof}

\section{$M$-ideals}\label{Sec_5}

We recall the definitions of $M$-ideals and $M$-summands,
together with some elementary facts.
See, for example, Definition I.1.1 of~\cite{MIBS}
and the discussion afterwards.
If $E$ is a Banach space and $P \in B (E)$
is an idempotent,
then $P$ is called an {\emph{$L$-projection}}
if $\| \xi \| = \| P \xi \| + \| (1 - P) \xi \|$
for all $\xi \in E$,
and an {\emph{$M$-projection}}
if $\| \xi \| = \max (\| P \xi \|, \, \| (1 - P) \xi \| )$
for all $\xi \in E$.
The ranges of $L$-projections and $M$-projections
are called {\emph{$L$-summands}} and {\emph{$M$-summands}}.
The idempotent $P$ is an $M$-projection
if and only if $P^*$ is an $L$-projection,
and is an $L$-projection
if and only if $P^*$ is an $M$-projection.
Finally,
a subspace $J \subseteq E$ is an {\emph{$M$-ideal}}
if $J^{\perp}$ is an $L$-summand in~$E^*$,
equivalently
(using \cite[Theorem I.1.9]{MIBS}),
$J^{\perp \perp}$ is an $M$-summand in~$E^{**}$.
By \cite[Proposition I.1.2 (a)]{MIBS},
if $J$ is an $M$-summand,
then there is exactly one contractive idempotent with range~$J$,
namely the $M$-projection used in the definition.

Smith and Ward show in~\cite{SmithWard}
that the $M$-ideals in a $C^*$-algebra are exactly
the closed ideals in the usual sense (Theorem~5.3),
that an $M$-ideal in a unital Banach algebra
must be a subalgebra (Theorem~3.6),
and that $M$-ideals in Banach algebras
are often ideals (see, for example, Theorem~3.8).
Example~4.1 of~\cite{SmithWard}
shows that there are $M$-ideals in $B (l^1_2)$
which are subalgebras but not ideals and do not have cais.

The following definition is from the introduction to~\cite{BOZ}.

\begin{definition}\label{D_7X14_MAppUnital}
Let $A$ be a Banach algebra.
We say that $A$ is {\emph{$M$-approximately unital}}
if $A$ is an $M$-ideal in the multiplier unitization~$A^1$.
\end{definition}

As in the introduction to~\cite{BOZ},
an $M$-approximately unital Banach algebra is approximately unital.
The papers \cite{BOZ, BSan} give a number of
properties of $M$-approximately unital Banach algebras.
For example,
an $M$-approximately unital Banach algebra
having a real positive cai $(e_t)_{t \in \Ld}$
satisfying $\| 1 - 2 e_t \| \leq 1$ for all $t \in \Ld$
(\cite[Theorem 5.2]{BOZ}),
is Hahn-Banach smooth in its multiplier unitization
(Proposition I.1.12 of  \cite{MIBS}),
and has the Kaplansky density properties given
in \cite[Theorem 5.2]{BOZ}
and \cite[Proposition 6.4]{BOZ}.

\begin{proposition}\label{P_7X14_KEMAppU}
Let $p \in (1, \infty) \setminus \{ 2 \}$
and let $(X, \mu)$ be a measure space.
Then $\Kdb (L^p (X, \mu))$ is $M$-approximately unital
if and only if $\mu$ is purely atomic.
\end{proposition}

\begin{proof}
Theorem~11 of~\cite{Lima} states that
$\Kdb (L^p (X, \mu))$ is an
$M$-ideal in $B (L^p (X, \mu))$
if and only if $\mu$ is purely atomic.
By Theorem VI.4.17 in \cite{MIBS},
$\Kdb (L^p (X, \mu))$ is an
$M$-ideal in $B (L^p (X, \mu))$
if and only if it is an
$M$-ideal in $\Kdb (L^p (X, \mu)) + \Cdb 1$,
where $1$ is the identity operator on $L^p (X, \mu)$.
By Lemma \ref{mund}, $\Kdb (L^p (X, \mu)) + \Cdb 1$ is
the multiplier unitization of $\Kdb(L^p (X, \mu))$.
\end{proof}

\begin{lemma}\label{mlem}
Let $A$ be an approximately unital
Arens regular Banach algebra,
and let $J \subseteq A$ be an $M$-ideal in $A$,
with associated $M$-projection $P \colon A^{**} \to J^{\perp \perp}$.
Then $J$ is an approximately unital closed ideal
if and only if $P (1)$ is central in $A^{**}$.
\end{lemma}

\begin{proof}
By the discussion before Proposition 8.1 of~\cite{BOZ},
centrality of $P (1)$
implies that $J$ is an approximately unital closed ideal.

If $J$ is an approximately unital closed ideal
then, as in the proof of Lemma \ref{quo},
there is a central idempotent $e$
such that
$J^{\perp \perp} = e A^{**} = A^{**} e$.
The uniqueness of projections onto an $M$-summand
implies that $P$ is multiplication by~$e$.
So $P (1) = e$ is central.
\end{proof}

\begin{theorem}\label{mth}
Let  $p \in (1, \infty)$
and let $A$ be an $L^p$-operator algebra.
\begin{enumerate}
\item\label{mth_AppU}
Suppose that $A$ is approximately unital.
Then every $M$-ideal in~$A$ is an approximately unital closed ideal.
\item\label{mth_U}
Suppose that $A$ is unital.
Then $J \subseteq A$ is an $M$-summand
if and only if there is a central hermitian idempotent $z \in A$
such that $J = A z$.
In this case,
multiplication by $z$ is an $M$-projection with range~$J$.
\item\label{mth_Mapp}
Suppose that $A$ is $M$-approximately unital
(Definition~{\rm{\ref{D_7X14_MAppUnital}}}).
Then:
\begin{enumerate}
\item\label{mth_Mapp_Down}
Every $M$-ideal in~$A$ is $M$-approximately unital.
\item\label{mth_Mapp_Inter}
The intersection of finitely many $M$-ideals in~$A$
is an $M$-ideal in~$A$.
\item\label{mth_Mapp_Sum}
The closed ideal generated by any collection of $M$-ideals in~$A$
is an $M$-ideal in~$A$.
\end{enumerate}
\end{enumerate}
\end{theorem}

\begin{proof}
We prove~(\ref{mth_U}).
We may assume
(by the discussion above Proposition~\ref{P_7702_HermIsMult},
or the corollary on
p.\  136 in \cite{Lacey})
that there is a decomposable measure space $X$ such that
$A$ is a unital subalgebra of $B (L^p (X, \mu))$.

Suppose that $z \in A$ is a central hermitian idempotent.
Then $z$ is a hermitian idempotent in $B (L^p (X, \mu))$.
It follows from Proposition \ref{P_7702_HermIsMult}
that $z$ is multiplication by
the characteristic function of a locally measurable subset $E$ of $X$.
Thus for $x, y \in A$
(with suitable interpretation of the integrals below
if $E$ is only locally measurable),
\begin{align*}
& \| z x z + (1 - z) y (1 - z) \|^p
\\
& \hspace*{3em} {\mbox{}}
 = \sup \bigg(  \bigg\{ \int_E \, |x z \xi|^p \, d \mu
           + \int_{X \SM E} |y (1 - z) \xi|^p \, d \mu \colon
      \xi \in \Ball (L^p (X)) \bigg\} \bigg)
\\
& \hspace*{3em} {\mbox{}}
  \leq \max ( \| x \|, \| y \| )^p \,
   \sup \bigg(  \bigg\{ \int_E \, |\xi|^p \, d \mu
          + \int_{X \SM E} | \xi|^p \, d \mu \colon
     \xi \in \Ball (L^p (X))  \bigg\} \bigg)
\\
& \hspace*{3em} {\mbox{}}
 = \max ( \| x \|, \| y \| )^p.
\end{align*}
So multiplication by $z$ is an $M$-projection on $A$
and $z A$ is an $M$-summand.

Conversely,
let $P$ be an $M$-projection on~$A$,
and let $z = P (1)$.
By \cite[Proposition 3.1]{SmithWard},
$z$ is a hermitian idempotent.
Also, $P^*$ is an $L$-projection, so for any state $\varphi$ on $A$,
if $P^* (\varphi) \neq 0$ then
$\psi = \| P^* (\varphi) \|^{-1} P^* (\varphi)$
is a state with $\psi (z) = 1$
as in the proof of 4.8.5 in \cite{BLM}.
It follows from Lemma~\ref{L_7Z18_LpVanishingState}
that
\[
\psi ( (1 - z) A) = \psi (A (1 - z)) = 0.
\]
So
\[
\varphi (P ( (1 - z) A)) = \varphi (P (A (1 - z))) = 0
\]
for any state $\varphi$ on $A$.
Thus
\[
P ( (1 - z) A) = P (A (1 - z)) = 0
\]
by Lemma \ref{bicr}~(\ref{bicr_2b}).
A similar argument applied to $1 - P$
shows that
\[
(1 - P) (z A) = (1 - P) (A z) = 0.
\]
So $z A + A z \subseteq P (A)$.
Thus
\[
P (a) = P (z a + (1 - z) a) = P (z a) = z a,
\]
for all $a \in A$, and similarly $P (a) = a z$.
So $z$ is central and $P (A) = A z$.
This completes the proof of~(\ref{mth_U}).

We prove~(\ref{mth_AppU}).
Let $J \subseteq A$ be an $M$-ideal.
Since $J^{\perp \perp}$ is an $M$-ideal in $A^{**}$,
since $A$ is Arens regular
(Lemma~\ref{L_7618_SecDual}~(\ref{L_7618_SecDual_AReg})),
and since $A^{**}$ is an $L^p$-operator algebra
(Lemma~\ref{L_7618_SecDual}~(\ref{L_7618_SecDual_Lp})),
we can apply part~(\ref{mth_U})
and Lemma~\ref{mlem}.

Part~(\ref{mth_Mapp_Down})
follows from \cite[Proposition 3.2 (3)]{BOZ}
and part~(\ref{mth_AppU}),
and (\ref{mth_Mapp_Inter}) and (\ref{mth_Mapp_Sum})
now follow from \cite[Theorem 8.3]{BOZ}.
\end{proof}

\begin{remark}\label{wssi}
Let $A$ be an approximately unital $L^p$-operator algebra.
The proof of Theorem~\ref{mth} shows that the $h$-ideals,
as defined at the beginning of Section~3 of \cite{GKS},
are exactly the $M$-ideals.
One may ask if these are also the $u$-ideals
as defined before our Lemma \ref{havec2}.
This is not true: the idempotent $e_2$ in Example  \ref{mp}
gives a $u$-projection which is not an $M$-projection,
since as we said there $e_2$ is not hermitian.
Suppose that $A$ is a $u$-ideal
in its multiplier unitization $A^1$,
equivalently, as pointed out before Lemma~\ref{havec2},
that $A$ is bi-approximately unital.
One may ask whether it follows that $A$ is an $M$-ideal
in $A^1$.
As we will see in Corollary~\ref{mlem2},
the latter is equivalent to being scaled.
Recall from Lemma \ref{havec2}
that an approximately unital
$L^p$-operator algebra $A$ with a real positive bounded approximate
identity is bi-approximately unital.
(We conjectured after Lemma  \ref{havec2}
that the converse is true.)
More drastically, one may ask
if an approximately unital
$L^p$-operator algebra
with a real positive bounded approximate
identity is necessarily scaled.
We believe that this is unlikely to be true.
\end{remark}

\section{Scaled $L^p$-operator algebras}\label{Sec_6}

In the Introduction we said that an approximately unital
Banach algebra $A$ is {\emph{scaled}} if the set of restrictions to $A$
of states on $A^1$ equals the quasistate space $Q (A)$ of $A$.
Equivalently (see \cite{BOZ}, before Lemma 2.7 there)
an approximately unital Banach algebra is scaled
if every real positive functional (see Definition~\ref{D_7702_RealPos})
is a nonnegative multiple of a state.
That is, in the notation of
Definition~\ref{D_7702_RealPos}
and Definition~\ref{Dstate},
we have ${\mathfrak{c}}_{A^*} = \Rdb_{+} S (A)$, or
equivalently, ${\mathfrak{c}}_{A^*} \cap \Ball (A^*) = Q (A)$. 

Unital Banach algebras are scaled
(this is a special case of \cite[Proposition 6.2]{BOZ}),
and all $C^*$-algebras are well known to be scaled.

If $A$ is a nonunital
approximately unital Arens regular Banach algebra,
then the support idempotent of $A$ in $(A^1)^{**}$
is the weak* limit in $(A^1)^{**}$ of any cai in~$A$.
This exists and is an identity for $A^{**}$
by the argument of Lemma~\ref{L_7917_AppIdConverge}.
Clearly it is central in $(A^1)^{**}$.

\begin{lemma}\label{mlem1}
Suppose that $A$ is a nonunital
scaled approximately unital Arens regular Banach algebra.
Then the support idempotent of $A$ in $(A^1)^{**}$ is hermitian.
\end{lemma}

\begin{proof}
Suppose that $A$ is scaled and $(e_t)_{t \in \Ld}$ is a cai for $A$.
Then, as above, $(e_t)_{t \in \Ld}$ converges weak* to
a central idempotent $e \in (A^1)^{**}$
which is an identity for $A^{**}$.

If $\varphi$ is a state
on $A^1$ then $\varphi |_A$ is a nonnegative multiple, $r$ say,
of a state
on $A$, so that $\varphi (e) = \lim_t \varphi (e_t) = r \geq 0$.
So every
weak* continuous state on $(A^1)^{**}$ is nonnegative on $e$.
Since the weak* continuous states on a dual Banach algebra
are weak* dense in the states by \cite[Theorem 2.2]{Mag},
it follows from Lemma~\ref{L_7702_SAH} that
$e$ is hermitian.
\end{proof}

The last result says that
scaled
approximately unital Arens regular
Banach algebra are $h$-ideals
in their multiplier unitizations
as defined at the beginning of Section~3 of \cite{GKS}.

\begin{corollary}\label{mlem2}
Suppose that $A$ is an
approximately unital Arens regular
Banach algebra with the property that
whenever $e \in (A^1)^{**}$
is a hermitian idempotent
and $x, y \in A$,
then
$\| e x e + (1 - e) y (1 - e) \| \leq \max ( \|x \|, \| y \| )$.
Then $A$ is scaled if and only if $A$ is $M$-approximately unital.
\end{corollary}

\begin{proof}
Since unital algebras
are both scaled and approximately unital,
we may assume that $A$ is nonunital.
If $A$ is $M$-approximately unital then $A$ is scaled
by \cite[Proposition 6.2]{BOZ}.
For the other direction, by Lemma \ref{mlem1} and the hypothesis,
the support idempotent $e$ of $A$ in $(A^1)^{**}$ satisfies
$\| e x + (1 - e) y  \| \leq \max ( \|x \|, \| y \| )$ for $x, y \in A$.
By Goldstine's Theorem
and separate weak* continuity of multiplication
(\cite[2.5.3]{BLM}),
this inequality holds for all $x, y \in A^{**}$.
So multiplication by $e$ is an $M$-projection on $(A^1)^{**}$.
Therefore $A$ is an $M$-ideal in~$A^1$.
\end{proof}

\begin{corollary}\label{C_7X15_mlem2_Lp}
Let $A$ be an
approximately unital  $L^p$-operator algebra.
Then $A$ is scaled if and only if $A$ is $M$-approximately unital.
\end{corollary}

\begin{proof}
Use Corollary~\ref{mlem2}
and a computation in
the proof of Theorem \ref{mth}~(\ref{mth_U}).
\end{proof}

\begin{corollary}\label{scif}
Let  $p \in (1, \infty)$
and let $A$ be a nonunital approximately unital $L^p$-operator algebra.
Then the following are equivalent:
\begin{enumerate}
\item\label{7X15_scif_Scaled}
$A$ is scaled.
\item\label{7X15_scif_Herm}
The support idempotent $e$ of $A$ in $(A^1)^{**}$
(as defined at the beginning of the section) is hermitian.
\item\label{7X15_scif_Cpt}
The quasistate space
$Q (A)$ is weak* compact.
\end{enumerate}
\end{corollary}

If these hold then,
by Corollary~\ref{C_7X15_mlem2_Lp},
$A$ has all the properties of $M$-approximately unital algebras
described after Definition~\ref{D_7X14_MAppUnital}.

\begin{proof}[Proof of Corollary~{\rm{\ref{scif}}}]
The implication from (\ref{7X15_scif_Scaled})
to~(\ref{7X15_scif_Herm}) is
Lemma \ref{mlem1}.
For the reverse, if $e$ is hermitian then
multiplication by $e$ is an $M$-projection from $(A^1)^{**}$
to $A^{**}$ by Theorem \ref{mth} (\ref{mth_U}),
so $A$ is $M$-approximately unital.
Apply Corollary~\ref{C_7X15_mlem2_Lp}.

For the equivalence with~(\ref{7X15_scif_Cpt}),
first, $Q (A)$ is weak* compact if and only if it is weak* closed.
Also, Corollary \ref{HBsm}~(\ref{HBsm_EqState})
implies convexity of $Q (A)$,
as explained in Remark~\ref{sconv}.
Apply Lemma 2.7~(2) in \cite{BOZ}.
\end{proof}

The following answers the open question from \cite{BOZ}
as to whether all approximately unital Banach algebras are scaled.

\begin{corollary}\label{nots}
If $p \in (1, \infty) \setminus \{ 2 \}$
then $\Kdb (L^p ([0, 1]))$
is an approximately unital $L^p$-operator algebra which is
not scaled.
\end{corollary}

\begin{proof}
That $\Kdb (L^p ([0, 1]))$ has a cai is observed
in Example~\ref{kl01}.
This algebra is not $M$-approximately unital
by Proposition~\ref{P_7X14_KEMAppU},
and hence not scaled by Corollary~\ref{C_7X15_mlem2_Lp}.
\end{proof}

\begin{corollary}\label{C_7X15_KlpScaled}
The algebra $\Kdb (l^p)$ is scaled.
\end{corollary}

\begin{proof}
This algebra is $M$-approximately unital
by Proposition~\ref{P_7X14_KEMAppU},
hence scaled by Corollary~\ref{C_7X15_mlem2_Lp}.
\end{proof}

The last result can also be deduced from Proposition~\ref{hau}.

\begin{proposition}\label{hau}
Let  $p \in (1, \infty)$.
Suppose that an $L^p$-operator algebra $A$ has a cai $(e_t)_{t \in \Ld}$
consisting of hermitian elements of $A^1$.
Then $A$ is scaled.
\end{proposition}

\begin{proof}
Since unital algebras are scaled,
we may assume that $A$ is nonunital.
With $e_t \to e$ as usual, it follows as in the
proof of Lemma~\ref{bicr}~(\ref{bicr_3})
(using the fact that
normal states are weak* dense)
that $e$ is hermitian and central in $(A^1)^{**}$.
It follows from Theorem \ref{mth} or Corollary \ref{scif} that
$A$ is $M$-approximately unital and scaled.
\end{proof}

Proposition~\ref{hau} may suggest that
one requirement for a nonunital $L^p$-operator algebra
to be ``$C^*$-like'' is that it have
a hermitian cai.
The canonical cai for $\Kdb (l^p)$ is a real positive hermitian cai
as we said in Example \ref{kl01}.
On the other hand the cai for $\Kdb (L^p ([0, 1]))$
in Example~\ref{kl01}
seems perhaps surprisingly to have no good ``positivity'' properties.
Indeed as we said in Example \ref{kl01},
$A = \Kdb (L^p ([0, 1]))$ has no real positive cai.
Of course for any approximately unital $L^p$-operator algebra
the identity $e$ of $A^{**}$ is real positive in $A^{**}$.
However $e$ need not be real positive (accretive) in $(A^1)^{**}$,
and certainly is not hermitian,
as we said after the proof of Lemma~\ref{bicr}.

\begin{proposition}\label{au2}
Let  $p \in (1, \infty)$.
Suppose that $A$ is a closed subalgebra
of a scaled $L^p$-operator algebra $B$,
with a common cai.
Then $A$ is scaled.
\end{proposition}

\begin{proof}
We may assume that $A$ is nonunital (unital algebras
are scaled).
We may view $A^1 \subseteq B^1$.
Any state $\varphi$ of $A^1$ extends to a state of $B^1$,
and the restriction of the latter to $B$
equals $\lambda \psi$
for some $\lambda \in [0, 1]$ and $\psi \in S (B)$.
However $\psi |_A \in S (A)$ since
Corollary \ref{HBsm}~(\ref{HBsm_EqState})
implies that $\psi (e_t) \to 1$,
where $(e_t)_{t \in \Ld}$ is the common cai.
Since $\varphi |_A = \lambda \psi |_A$ we are done.
\end{proof}

\begin{remark}\label{auiias}
Approximately unital ideals in a scaled $L^p$-operator algebra $A$
need not be scaled,
for example $\Kdb (L^p([0,1]))$ in $B (L^p([0,1]))$.
(The latter
is scaled as is any unital Banach algebra,
and we showed above that $\Kdb (L^p([0,1]))$
is not scaled.)
However if the approximately unital ideal
is also an $M$-ideal in $A$, then it is scaled
by Theorem \ref{mth}.
\end{remark}

\section{Kaplansky density}\label{Kaps}

One may ask if in an approximately unital $L^p$-operator algebra
there are
Kaplansky Density Theorems analogous to the ones established by
the first author and Read
for approximately unital $L^2$-operator algebras.
See
e.g.\  \cite[Theorem 5.2 and Proposition 6.4]{BOZ}
for a more general variant of the latter.   As we said in the introduction,
 the usual Kaplansky density theorem variants
for $C^*$-algebras can be shown
to follow easily from the weak* density of the subset of interest
in  $A$ within the matching set in $A^{**}$; and our Kaplansky density theorems have this 
 flavor.

In the following result,
for an approximately unital $L^p$-operator algebra $A$
we take ${\mathfrak{r}}_{A^{**}}$
to be the accretive elements in the unital Banach algebra $A^{**}$.
This is different from the definition
after Lemma~{\rm{2.5}} in~\cite{BOZ}.
The two definitions do coincide
if also $A$ is scaled, 
by Proposition~\ref{scas}
and Proposition \ref{2extension}~(\ref{2extension_2}).

\begin{proposition}\label{L_7618_KapScal}
Let $p \in (1, \infty)$
and let $A$ be an approximately unital $L^p$-operator algebra. 
The following are equivalent:
\begin{enumerate}
\item\label{Item_Kap_rposA}
${\mathfrak{r}}_A$ is weak* dense in ${\mathfrak{r}}_{A^{**}}$.
\item\label{Item_Kap_Ball}
${\mathfrak{r}}_A \cap \Ball (A)$
is weak* dense in ${\mathfrak{r}}_{A^{**}} \cap \Ball (A^{**})$.
\item\label{Item_KpD_FA}
${\mathfrak{F}}_A$ is weak* dense in ${\mathfrak{F}}_{A^{**}}$.
\item\label{Item_Kap_scal}
$A$ is scaled.
\end{enumerate}
\end{proposition}

\begin{proof}
Since the definition of ${\mathfrak{r}}_{A^{**}}$
in~\cite{BOZ}
coincides with ours when $A$ is scaled
(as pointed out above),
that (\ref{Item_Kap_scal}) implies~(\ref{Item_Kap_rposA})
follows from Proposition 2.11 of~\cite{BOZ}.
The proof of Lemma 6.4 of~\cite{BOZ}
works just as well for our version of ${\mathfrak{r}}_{A^{**}}$
as for the one there,
and thus shows that
(\ref{Item_Kap_rposA}) implies~(\ref{Item_Kap_Ball}).
By our Corollary \ref{C_7X15_mlem2_Lp}
and by Theorem 5.2 of~\cite{BOZ}
it follows that (\ref{Item_Kap_scal}) implies~(\ref{Item_KpD_FA}).
That (\ref{Item_KpD_FA}) implies (\ref{Item_Kap_rposA}) follows easily
from Proposition~\ref{P_7702_rAAndFA}.

Assuming (\ref{Item_Kap_Ball}) we will prove~(\ref{Item_Kap_scal})
by showing that every nontrivial
real positive functional $\varphi$
(see Definition~\ref{D_7702_RealPos})
is a nonnegative multiple of a state.
We may assume that $A$ is nonunital.
The canonical weak* continuous extension
$\widetilde{\varphi}$ of $\varphi$ to  $A^{**}$ is
real positive by our assumption (\ref{Item_Kap_Ball})
and a standard approximation argument.
Since $A^{**}$ is unital it is scaled,
so that $\widetilde{\varphi} = t \psi$
for a state $\psi$ on $A^{**}$ and some $t > 0$.
Thus $\varphi = t \psi |_{A}$.
The span of $A$ and the identity of $A^{**}$
is the multiplier unitization of $A$
by the last paragraph of Section 1 of~\cite{BOZ}.
Hence $\psi |_{A}$ is a state on~$A$.
\end{proof}

These hold in particular if $A$ is unital.
Such results also hold if $A$ has the following property:
with $1$ being the identity of some unitization of~$A$,
given $\varepsilon > 0$ there exists $\delta > 0$
such that if  $y \in A$ with $\| 1 - y \| < 1 + \delta$
then there is $z \in A$
with $\| 1 - z \| = 1$ and $\| y - z \| < \varepsilon$.
This follows from \cite[Proposition 6.4]{BOZ}
and the proof of \cite[Theorem 5.2]{BOZ}.
It may be interesting to ascertain
which $L^p$-operator algebras have this property.

We end by mentioning some of what seem to us
to be the most important open questions related to the approach of
our paper.
\begin{enumerate}
\item\label{Item_7Z21_Kap}
Is there a Kaplansky density type theorem
for a non-scaled 
approximately unital $L^p$-operator algebra $A$?
(See Proposition \ref{L_7618_KapScal} for the scaled case.)
For example, one may ask if
${\mathfrak{r}}_A$ is weak* dense
in ${\mathfrak{r}}_{(A^1)^{**}} \cap A^{**}$.
\item\label{Item_7Z21_Scaled}
Is every approximately unital subalgebra of $B (l^p)$ scaled?

\item\label{Item_7Z21_Sub}
Let $A$ be an approximately unital $L^p$-operator algebra.
Is ${\mathfrak{r}}_A - {\mathfrak{r}}_A$ always a subalgebra?
\item\label{Item_7Z21_RPcai}
Is every bi-approximately unital  $L^p$-operator algebra scaled?
Does it have a real positive cai?
More drastically, if an  $L^p$-operator algebra
possesses a real positive cai then is it scaled?
\end{enumerate}

\section{Index}\label{Sec_Index}

For the readers' convenience we list, alphabetically but compactly,
some of the main definitions
in this paper and where they may be found 
(the Definition number, which is usually their first occurrence).  

Accretive: \ref{D_7702_RealPos};
approximately unital Banach algebra: \ref{D_7618_cai};
Arens products, Arens regular: \ref{D_Arens};
bi-approximately unital ideal: \ref{biau_Lp};
bi-approximately unital algebra: \ref{biau_AReg};
bicontractive idempotent: \ref{D_bicontrinvisom};
$\mathfrak{c}_{A^*}$: \ref{D_7702_RealPos};
cai: \ref{D_7618_cai};
decomposable: \ref{D_locmeasdecomp};
dual $L^p$-operator algebra: \ref{D_7620_DualLpOpAlg};
${\mathfrak{F}}_A$: \ref{D_7702_FA};
Hahn-Banach smooth: \ref{D_7Z20_HBS};
hermitian: \ref{DHerm};
invertible isometry: \ref{D_bicontrinvisom};
locally measurable, locally a.e.: \ref{D_locmeasdecomp}; 
$L^p$-operator algebra: \ref{D_7618_LpOpAlg};
$M$-approximately unital: \ref{D_7X14_MAppUnital};
$M$-ideal, $M$-projection,
$M$-summand: beginning of Section \ref{Sec_5};
multiplier unitization~$A^1$: \ref{D_7702_Unitization};
order on idempotents $e \leq_{\mathrm{r}} f$, $e \leq f$: \ref{Dorder};
powers and roots $b^t$: \ref{N_7919_nthroot};
quasistate space $Q(A)$: \ref{Dstate};
real positive: \ref{D_7702_RealPos};
${\mathfrak{r}}_A$: \ref{D_7702_RealPos};
scaled: beginning of Section \ref{Sec_6}; 
smooth: \ref{D_smoothconv}; 
$SQ_p$-algebra, $SQ_p$ space: the introduction;
state space $S (A)$: \ref{Dstate};
strictly convex: \ref{D_smoothconv};
support idempotent $s (x)$: \ref{D_7Z17_SuppId};
unital Banach algebra: \ref{D_7618_UBA};
unitization: \ref{D_7720_Unitization}.

Other definitions may be found in the introduction, or
in the sections where they first appear (often at the
start of the section), but are not
specifically numbered.

\section{Acknowledgements}\label{Sec_Ackx}

The second author wishes to
thank the Centre de Recerca Matematica,
Universitat Aut\`{o}noma de Barcelona
for support during a visit in summer 2017.
We also thank Eusebio Gardella for several conversations
after some of the material from this paper
and the sequel was presented at a conference in Houston
in August 2017, and for several comments on a draft of our paper. 
Finally we thank the referee for
some suggestions which improved the exposition.

\end{document}